\documentclass[12pt]{amsart}
\usepackage{amsthm}
\usepackage{amsmath,amsxtra,amssymb,amsfonts,latexsym, mathrsfs, dsfont,bm,mathtools}
\usepackage[initials]{amsrefs}
\usepackage[usenames]{color}
\usepackage[hyperindex,pagebackref]{hyperref}

\numberwithin{equation}{section}

\makeatletter

\newcommand{\Rmnum}[1]{\expandafter\@slowromancap\romannumeral #1@}
\makeatother
\newcommand{\8}{\infty}

\renewcommand{\for}{\begin{eqnarray*}}
\newcommand{\mel}{\end{eqnarray*}}
\def\fr{\begin{align*}}

\newcommand{\kl}{\pl \le \pl}
\newcommand{\gl}{\pl \ge \pl}

\newcommand{\lel}{\pl = \pl}

\newcommand{\Hess}{\operatorname{Hess}}
\newcommand{\Dom}{\operatorname{Dom}}
\newcommand{\Prob}{\operatorname{Prob}}

\newcommand{\nz}{{\mathbb N}}
\newcommand{\tz}{{\mathbb T}}

\newcommand{\rz}{{\mathbb R}}
\newcommand{\zz}{{\mathbb Z}}

\newcommand{\cz}{{\mathbb C}}
\newcommand{\fz}{{\mathbb F}}

\newcommand{\ten}{\otimes}

\newcommand{\pl}{\hspace{.1cm}}

\newcommand{\Om}{\Omega}
\newcommand{\om}{\omega}
\newcommand{\al}{\alpha}
\newcommand{\si}{\sigma}

\newcommand{\tet}{\theta}

\newcommand{\La}{\Lambda}
\newcommand{\la}{\lambda}
\newcommand{\bt}{\beta}
\newcommand{\eps}{\varepsilon}

\newcommand{\Lc}{{\mathcal L}}
\newcommand{\F}{{\mathcal F}}

\newcommand{\A}{{\mathcal A}}
\newcommand{\B}{{\mathcal B}}
\newcommand{\D}{{\mathcal D}}
\newcommand{\M}{{\mathcal M}}
\newcommand{\K}{{\mathcal K}}
\newcommand{\R}{{\mathcal R}}
\newcommand{\Z}{\mathcal{Z}}

\newcommand{\Ha}{{\mathcal H}}

\newcommand{\Fix}{\operatorname{Fix}}

\newcommand{\var}{\operatorname{var}}
\newcommand{\Ent}{\operatorname{Ent}}
\newcommand{\td}{\widetilde}
\mathchardef\dash="2D

\newcommand{\N}{{\mathcal N}}
\newcommand{\Ga}{\Gamma}

\newcommand{\ga}{\gamma}
\newcommand{\de}{\delta}

\newcommand{\U}{{\mathcal U}}

\newcommand{\pz}{\mathbb{P}}

\newcommand{\ez}{{\mathbb E}\vspace{0.1cm}}

\newtheorem{lemma}{Lemma}[section]
\newtheorem{prop}[lemma]{Proposition}
\newtheorem{theorem}[lemma]{Theorem}
\newtheorem{cor}[lemma]{Corollary}

\theoremstyle{remark}
\newtheorem{rem}[lemma]{Remark}

\newtheorem{prob}[lemma]{Problem}

\theoremstyle{definition}
\newtheorem{defi}[lemma]{Definition}
\newtheorem{exam}[lemma]{Example}

\newcommand{\re}{\begin{rem}\rm}
  \newcommand{\mar}{\end{rem}}
\newcommand{\ex}{\begin{exam}\rm}
\newcommand{\amp}{\end{exam}}

\newcommand{\prf}{\begin{proof}[\bf Proof:]}
\newcommand{\xspace}{\hbox{\kern-2.5pt}}

\newcommand{\rge}{\rangle }
\newcommand{\lge}{\langle }

\allowdisplaybreaks
\title[Martingale deviation and Poincar\'e type inequalities]{Noncommutative martingale deviation and Poincar\'e type inequalities with applications}
\address{Department of Mathematics, University of Illinois, Urbana, IL 61801}
\email{junge@math.uiuc.edu, \quad zeng8@illinois.edu}
\thanks{The first author was Partially supported by NSF Grant DMS-1201886}
\author{Marius Junge}
\author{Qiang Zeng}
\date{\today}
\begin{document}
\begin{abstract}
We prove a deviation inequality for noncommutative martingales by extending Oliveira's argument for random matrices. By integration we obtain a Burkholder type inequality with satisfactory constant. Using continuous time, we establish noncommutative Poincar\'e type inequalities for ``nice'' semigroups with a positive curvature condition. These results allow us to prove a general deviation inequality and a noncommutative transportation inequality due to Bobkov and G\"otze in the commutative case. To demonstrate our setting is general enough, we give various examples, including certain group von Neumann algebras, random matrices and classical diffusion processes, among others.
\end{abstract}

\maketitle
\section*{Introduction}
In probability theory it is well known that martingale inequalities can be used to prove and extend classical inequalities such as Riesz transforms and Poincar\'{e} inequalities to larger setting. Moreover, once a true probabilistic argument has been found, it is then often easier to prove dimension free estimates. This applies in particular to Riesz transforms; see Gundy \cite{Gun}, Pisier \cite{Pis}. The impressive work of Lust-Piquard shows that, whenever the method applies, it provides the optimal constant for Riesz transforms and Poincar\'{e} inequalities. The only drawback here is that the setup for Pisier's method is so  special that it requires ingenuity to establish it in every single case.

Our aim here is to establish a method which applies in a fairly general noncommutative situations. Let us first set up the framework. Let $\N$ be a finite von Neumann algebra equipped with a normal faithful tracial state $\tau:\N\to \cz$, i.e. $\tau(1)=1$ and $\tau(xy)=\tau(yx)$. Let $(\N_k)_{k=1, \cdots ,n}\subset \N$ be a filtration of von Neumann subalgebras with conditional expectation $E_k: \N\to \N_k$. For a martingale sequence $(x_k)$  with $x_k\in\N_k$, we write $dx_k=x_k-x_{k-1}$ for the martingale differences. The starting point here is the Burkholder inequality first proved in \cite{JX03}
 \begin{equation}
 \label{burk}  \Big\|\sum_k dx_k\Big\|_p
 \kl c(p) \Big(\Big(\sum_k \|dx_k\|_p^p\Big)^{\frac1p}+\Big\|\Big(\sum_k  E_{k-1} (dx_k^*dx_k + dx_kdx_k^*)\Big)^{1/2}\Big\|_p\Big) \pl.
 \end{equation}
The optimal order of constant here is $c(p)\sim cp$ in the noncommutative setting, and is due to Randrianantoanina \cite{Ran}. In the commutative, dissipative setting, Barlow and Yor \cite{BY82} showed a better constant in Burkholder--Davis--Gundy inequality
 \begin{align}\label{sqp}
  \|X_T\|_p \kl c\sqrt{p} \pl \|\lge X,X\rge_T^{1/2}\|_p \pl
  \end{align}
by reducing it with a time change to the case of Brownian motion. Here $\lge X,X\rge$ is the quadratic variation of the continuous (local) martingale $X$; see e.g. \cite{RY99}. Since the nature of the Brownian motion in the noncommutative setting is so vast (see \cite{CJ}), it is inconceivable that such an easy argument could work in a noncommutative situation. In fact stationarity of the Brownian motion is certainly required to perform a time change, and can no longer be guaranteed for noncommutative martingales. For many applications towards deviation inequalities it is enough to use the $L_\8$ norm on the right hand side of \eqref{sqp}. Following Oliveira's idea \cite{Ol}, we are able to use Golden--Thompson inequality to prove the following result. Throughout this paper $c,C$ and $C'$ will always denote positive constants which may vary from line to line.
\begin{theorem} Let $2\le p<\infty$ and $x_n=\sum_{k=1}^n dx_k$ be a discrete mean zero martingale. Then
\[ \|x_n\|_p \kl
 C \sqrt{p} \|(\sum_{k=1}^n E_{k-1}(dx_k^*dx_k+dx_kdx_k^*))^{1/2}\|_{\infty}+ C p \sup_k \|dx_k\|_{\infty} \pl. \]
If $x_n$ is self-adjoint, then
\[
\tau\left(1_{[t,\infty)}\left(x_n\right)\right)
 \kl \exp\left(-\frac{t^2}{8\|\sum_{k=1}^n E_{k-1}(dx_k^2)\|_\8 +{4t \sup_k\|dx_k\|_\8}}\right)\pl.
\]
\end{theorem}

Note that in the commutative context
 \[ \tau(1_{[t,\infty)}(x))\lel  \Prob(x\ge t) \pl .\]
In the future we will simply take this formula as a definition. The deviation inequality is a martingale version of noncommutative Bernstein inequality proved in \cite{JZ}. Tropp \cite{Tro} obtained better constants for tail estimate of random matrix martingales by using Lieb's concavity theorem. However, it seems Lieb's result is only applicable for commutative randomness.

Let us now indicate how to use this result in connection with curvature condition, more precisely the $\Ga_2$-condition introduced by Bakry--Emery \cite{BE85}. Here we assume that $(T_t)_{t\ge0}$ is a semigroup of completely positive trace preserving maps on a finite von Neumann algebra $\N$ with infinitesimal generator $A\ge0$. We assume in addition that $(T_t)$ is a noncommutative diffusion process (in short $(T_t)$ is nc-diffusion), namely that Meyer's famous gradient form
 \begin{equation}\label{grad}
  2\Gamma(x,x) \lel A(x^*)x+x^*A(x)-A(x^*x)  \in L_1(N)  \end{equation}
for all $x\in\Dom(A)\cap \N$ (To be more precise, for all $x\in\Dom(A^{1/2})\cap\N$ by extension). In this paper, $\Dom(A)$ denotes the domain of the operator $A$ in the underlying Hilbert space. Recall that according to {\cite{CSa}*{Section 9},} we always have $\Gamma(x,x)\in \B^*$, where $\B=\Dom(A^{1/2})\cap\N$ is a $*$-algebra by Davies and Lindsay \cite{DL}*{Proposition 2.8}. Thus \eqref{grad} is a regularity assumption, which in general is much weaker than assuming that $(T_t)$ is a usual diffusion semigroup. The standard example for a nc-diffusion semigroup which is not a diffusion is the Poisson semigroup on the circle. Let $\al>0$ be a constant in what follows.

\begin{theorem} Let $(T_t)$ be a nc-diffusion semigroup on a von Neumann algebra $\N$ satisfying the $\Gamma_2$-condition
 \begin{equation}\label{gam2}
  \tau( \Gamma(T_tx,T_tx)y)\kl C  e^{-2\al t} \tau( T_t\Gamma(x,x)y)
  \end{equation}
for all $x\in \Dom(A^{1/2})$ and all positive $y\in \N$. Then for self-adjoint $x$ we have
 \[ \|x-E_{{\rm Fix}}(x)\|_p \kl C' {\al^{-1/2}}\min\{\sqrt{{p}} \pl \|\Gamma(x,x)^{1/2}\|_\8, ~p \|\Gamma(x,x)^{1/2}\|_p \}\pl .\]
Here ${\rm Fix}=\{x: T_tx=x\}$ is the fixed point von Neumann subalgebra given by $T_t$ and $E_{{\rm Fix}}$ the corresponding conditional expectation.
\end{theorem}
Condition \eqref{gam2} is usually formulated in the form $\Ga_2(x,x)\ge \al \Ga(x,x)$ where
\[ 2\Gamma_2(x,y) \lel \Gamma(Ax,y)+\Gamma(x,Ay)-A\Gamma(x,y)\pl. \]
As in the commutative case, this result implies the deviation inequality and exponential integrability.
\begin{cor}
  Under the hypotheses above, we have for $t>0$,
  \begin{equation*}
  \tau(e^{t(x-E_{\rm Fix}x})\kl \exp\Big(c(\al) t^2 \|\Ga(x,x)\|_\8\Big).
 \end{equation*}
 and
 \begin{equation*}
 \Prob(x-E_{\rm Fix} x\ge t) \kl  \exp\Big(-\frac{t^2}{C(\al)\|\Ga(x,x)\|_\8}\Big)\pl.
\end{equation*}
\end{cor}

Quite surprisingly, these results apply to many commutative and noncommutative examples which cannot be treated with the usual commutative diffusion semigroup approach. Moreover, Bobkov and G\"otze's \cite{BG} application to the $L_1$ Wasserstein distance
\[W_1(f,h) \lel \sup\{ |\phi_f(x)-\phi_h(x)|: x \text{ self-adjoint },\|\Ga(x,x)\|_\8\le 1\} \]
for normal state $\phi_f(x)=\tau(fx)$,  $\phi_h(x)=\tau(hx)$, remains applicable in the noncommutative setting. This leads to the transportation inequality.

\begin{cor} Under the assumptions above
 \[ W_1(f,E_{\rm Fix} f) \kl C(\al) \sqrt{\tau(f\ln  f)} \]
for all normal states $\phi_f(x)=\tau(fx)$.
\end{cor}

The Wasserstein distance has been extensively studied in the noncommutative setting; see\cites{BV,Rie, OR}. This probabilistic connection which provides universal upper bound given by the entropy functional, however, seems to be new. Our definition of Wasserstein distance in the noncommutative setting is closely related to the metric used by Rieffel to define his quantum metric space. Inspired by Connes' work in noncommutative geometry \cite{Con}, Rieffel defined the metric on the state space of a $*$-algebra $A$
\[
 \rho_L(\phi,\psi)\lel \sup\{|\phi(a)-\psi(a)|: L(a)\le 1, a\in A\}
\]
where $\phi$,$\psi$ are states and $L(a)$ is a seminorm. For Connes' spectral triple, $L(a)=\|[D,a]\|$ where $D$ is a self-adjoint operator; see \cite{Rie} and the references therein for more details. It is an interesting question but beyond the scope of this paper to determine whether the transportation inequality is possible for $\rho_L$.

At this point it seems helpful to compare our approach with previous ones. Using classical diffusion theory, it is proved by Bakry and Emery \cite{BE85} that the $\Ga_2$-criterion implies the logarithmic Sobolev inequality (LSI). Bobkov and G\"otze \cite{BG} deduced an exponential integrability (EI) result based on a variant of LSI and showed that the EI is equivalent to the transportation inequality (TI). The relation can be illustrated by the following
\[
 \Ga_2\mbox{-criterion}\xRightarrow{ \rm diffusion} \mbox{ LSI } \Rightarrow \mbox{ EI }\Leftrightarrow \mbox{ TI } \pl.
\]
We refer the reader to the lecture notes \cite{Gui} for more details on this subject and its applications to random matrices. In the noncommutative setting, however, the $\Ga_2$-criterion no longer implies LSI; see Example \ref{cond} below. But we can still use our $L_p$ Poincar\'e inequality ($L_p{\rm PI}$) to deduce EI and TI. In particular, this gives an alternative proof in the commutative case. Our approach is illustrated as follows
\[
 \Ga_2\mbox{-criterion}\xRightarrow{\rm nc\dash diffusion} L_p{\rm PI} \Rightarrow \mbox{ EI }\Leftrightarrow \mbox{ TI } \pl.
\]
At the time of this writing, we are not sure whether this alternative is known or not in the commutative case. In addition, a simple argument shows that EI would hold (thus TI follows) provided the space has finite diameter. This gives a criterion for the validity of TI when we only have $\Ga_2(x,x)\ge0$ instead of $\Ga_2(x,x)\ge \al \Ga(x,x)$.

At the end of the paper we consider an algebraic version of the $\Gamma_2$-condition and say that $\Gamma_2\gl \al \Gamma$ (in the form sense) if
 \[ [\Gamma_2(x_j,x_k)]_{j,k}  \gl \al [\Gamma(x_j,x_k)]_{j,k} \]
for all finite families in a weakly dense $A$ invariant algebra $\A\subset \N$. Then we collect/prove the following facts
 \begin{itemize}
 \item $\Gamma_2\gl \Gamma$ for a suitable semigroup on group von Neumann algebra of the free group $\fz_n$ and the noncommutative tori.
 \item $\Gamma_2\gl \frac{n+2}{2n} \Gamma$ for suitable semigroups on group von Neumann algebra of the  discrete Heisenberg group and the hyperfinite $II_1$ factor.
 \item Let $\N=L_{\infty}(\{-1,1\})$ and $T_t(1)=1$, $T_t(\eps)=e^{-t}\eps$ where $(1,\eps)$ is the orthonormal basis of $L_2(\{-1,1\})$. Then $\Gamma_2\gl \Gamma$.
 \item Let $\N=L_{\infty}(\{1, \cdots ,n\})$ and $T_t(e^{\frac{2\pi ik}{n}})=e^{-t(1-\delta_{k,0})} e^{\frac{2\pi ik}{n}}$. Then $\Gamma_2\gl \frac{n+2}{2n}\Gamma$.
 \item $\Gamma_2\gl \Gamma$ for all $q$-Gaussian random variables and the number operator.
 \item $\Gamma_2\gl \al \Gamma$ for compact Riemannian manifolds with strictly positive Ricci curvature.
 \item $\Gamma_2\gl \al \Gamma$ is stable under tensor products.
 \item $\Gamma_2\gl \al \Gamma$ is stable under free products.
 \item $\Gamma_2\gl \frac{n+2}{2n}\Gamma$ for a suitable semigroup on random matrices $M_n$.
 \end{itemize}

We hope that this ample evidence that Poincar\'{e} type inequalities occur frequently in the commutative and the noncommutative setting even without assuming the strong diffusion assumption used in the Bakry--Emery theory justifies our new noncommutative theory. As special cases of our general theory, previous results obtained by Efraim/Lust-Piquard \cite{ELP} and Li \cite{Li08} are generalized or improved and many new inequalities are established in different contexts. For instance, the following deviation inequality for product probability spaces is an easy consequence of these examples: for $f\in L_\8(\Om_1\times \cdots \times\Om_n, \pz_1\otimes \cdots  \otimes\pz_n)$,
 \[
  \pz(f-\ez(f)\ge t) \kl \exp\Big(-\frac{ct^2}{\sum_{i=1}^n\| f-\int fd\pz_i\|_\8^2+\|\int (|f|^2-|\int f d\pz_i |^2 ) d\pz_i\|_\8}\Big)\pl.
 \]
 where $\pz=\pz_1\otimes \cdots  \otimes \pz_n$ and $\ez$ is the corresponding expectation operator.

The paper is organized as follows. The martingale deviation inequality and Burkholder type inequality are proved in Section 1. After recalling some results in the continuous filtration in von Neumann algebras, we deduce two BDG type inequalities in Section 2. The Poincar\'e type inequalities and the transportation inequalities are proved in Section 3, which is also the most technical section. Then the group von Neumann algebras are considered in Section 4. In section 5 we prove that the $\Ga_2$-criterion is stable under tensor products and free products with amalgamation. The general theory is applied to classical diffusion processes in Section 6.

\section{Noncommutative martingale deviation inequality}
Our proof of the martingale deviation inequality relies on the well known Golden--Thompson inequality. The fully general case is due to Araki \cite{Ar}. The version for semifinite von Neumann algebras we used here was proved by Ruskai in \cite{Ru}*{Theorem 4}.
\begin{lemma}[Golden--Thompson inequality]
Suppose that $a,b$ are self-adjoint operators, bounded above and that $a+b$ are essentially self-adjoint (i.e. the closure of $a+b$ is self-adjoint). Then $$\tau(e^{a+b})\kl\tau(e^{a/2}e^be^{a/2})\pl.$$
Furthermore, if $\tau(e^a)<\infty$ or $\tau(e^b)<\infty$ then
\begin{equation}\label{GT}
\tau(e^{a+b})\kl\tau(e^ae^b)\pl.
\end{equation}
\end{lemma}

\begin{lemma}\label{expin}
Let $(x_k)$ be a self-adjoint martingale sequence with respect to the filtration $(\N_k, E_k)$ and $d_k:=dx_k=x_k-x_{k-1}$ be the associated martingale differences such that
\begin{enumerate}
\item[i)] $\tau(x_k)=x_0=0$; {\rm ii)} $\|d_k\|\le M$; {\rm iii)} $\sum_{k=1}^n E_{k-1}(d_k^2) \le D^2 1$.
\end{enumerate}
Then \[\tau(e^{\la x_n}) \kl \exp[(1+\eps)\la^2 D^2]\] for all $\eps\in(0,1]$ and all $\la\in[0, \sqrt{\eps}/(M+M\eps)]$.
\end{lemma}
\begin{proof}
We follow Oliveira's original proof for matrix martingales \cite{Ol} and generalize it to the fully noncommutative setting. With the help of functional calculus, we actually have fewer technical issues. Let $\eps\in(0,1]$. Put $y_n=\sum_{k=1}^n E_{k-1}(d_k^2)$. Then $y_n\le D^21$. We simply write $D^2$ for the operator $D^2 1\in\N$ in the following. Let us first assume $M=1$. Since $e^{-((1+\eps)\la^2D^2-(1+\eps)\la^2y_n)}\le 1$, it follows from \eqref{GT} that
\begin{align*}
\tau\big(e^{\la x_n}\big)&\kl \tau\big(\exp[\la x_n+(1+\eps)\la^2D^2-(1+\eps)\la^2y_n]\exp[-((1+\eps)\la^2D^2-(1+\eps)\la^2y_n)]\big) \\
&\kl \tau\big(\exp[\la x_n+(1+\eps)\la^2D^2-(1+\eps)\la^2y_n]\big) \pl.
\end{align*}
Put $r_n=E_{n-1}d_n^2$. Then $y_n=y_{n-1}+r_n$. Using \eqref{GT} again we find
\begin{align*}
  &\tau\big(\exp[\la x_n+(1+\eps)\la^2D^2-(1+\eps)\la^2y_n]\big)\\
   \lel & \tau\big(\exp[\la x_{n-1}+\la d_n+(1+\eps)\la^2D^2-(1+\eps)\la^2y_{n-1}-(1+\eps)\la^2r_n]\big)\\
  \kl& \tau\big(\exp[\la d_n-(1+\eps)\la^2 r_n]\exp[\la x_{n-1}+(1+\eps)\la^2D^2-(1+\eps)\la^2y_{n-1}]\big)\pl.
  \end{align*}
Since $x_{n-1},y_{n-1}\in \N_{n-1}$ and $E_{n-1}$ is trace preserving, we obtain
\begin{equation}\label{indu}
\begin{split}
&\tau\big(\exp[\la d_n-(1+\eps)\la^2 r_n]\exp[\la x_{n-1}+(1+\eps)\la^2D^2-(1+\eps)\la^2y_{n-1}]\big)\\
\lel &\tau\big(E_{n-1}[\exp(\la d_n-(1+\eps)\la^2 r_n)]\exp[\la x_{n-1}+(1+\eps)\la^2D^2-(1+\eps)\la^2y_{n-1}]\big)\pl.
\end{split}
\end{equation}

We claim that $E_{k-1}[\exp(\la d_k-(1+\eps)\la^2 r_k)]\le 1$ for all $k=1, \cdots ,n$ and $0\le\la \le\sqrt{\eps}/(1+\eps)$. Indeed, $$\|\la d_k-(1+\eps)\la^2 r_k\|\le \frac{\sqrt{\eps}}{1+\eps}+(1+\eps)\frac{\eps}{(1+\eps)^2}=\frac{\sqrt{\eps}+\eps}{1+\eps}\le 1.$$
Note that $e^x\le 1+x+x^2$ for $|x|\le 1$. It follows from functional calculus that $e^A\le 1+A+A^2$ for any self-adjoint operator $A$ with $\|A\|\le 1$. Plugging in $A=\la d_k-(1+\eps)\la^2 r_k$ and using $r_k\in \N_{k-1}$ and $E_{k-1}d_k=0$ we obtain
\begin{align*}
&E_{k-1}[\exp(\la d_k-(1+\eps)\la^2 r_k)]\\
\kl & E_{k-1} [ 1+\la d_k-(1+\eps)\la^2 r_k + \la^2 d_k^2-(1+\eps)\la^3 d_k r_k - (1+\eps)\la^3 r_k d_k + (1+\eps)^2\la^4 r_k^2 ] \\
\lel & 1-\eps\la^2 r_k+(1+\eps)^2\la^4r_k^2 \pl.
\end{align*}
An elementary calculation shows that $\eps\la^2x-(1+\eps)^2\la^4x^2\ge 0$ for all $x\in[0,1]$ and $\la\in(0,\sqrt{\eps}/(1+\eps)]$. Using functional calculus of $r_k$ again, we find
\[ \eps\la^2 r_k-(1+\eps)^2\la^4r_k^2\gl 0 \]
which gives the claim. Combining with \eqref{indu}, we obtain
\begin{align*}
&\tau\big(\exp[\la x_n+(1+\eps)\la^2D^2-(1+\eps)\la^2y_n]\big)\\
\kl& \tau\big(\exp[\la x_{n-1}+(1+\eps)\la^2D^2-(1+\eps)\la^2y_{n-1}]\big).
\end{align*}
Iteratively using \eqref{GT} and the claim $n-1$ times yields
\[\tau(e^{\la x_n}) \kl \tau(\exp[(1+\eps)\la^2D^2])\lel\exp[(1+\eps)\la^2D^2] \] which completes the proof for $M=1$. For arbitrary $x_k$, considering $x_k'=x_k/M$ leads to the conclusion.
\end{proof}
We remark that the exponential inequality in this lemma is crucial for the proof of law of the iterated logarithms for noncommutative martingales by the second named author in \cite{Ze12}.
\begin{theorem}\label{fman}
Let $(x_k)$ be a self-adjoint martingale sequence with respect to the filtration $(\N_k, E_k)$ and $d_k:=dx_k=x_k-x_{k-1}$ be the associated martingale differences such that
\begin{enumerate}
\item[i)] $\tau(x_k)=x_0=0$; {\rm ii)} $\|d_k\|_\8\le M$; {\rm iii)} $\sum_{k=1}^n E_{k-1}(d_k^2) \le D^2 1$.
\end{enumerate}
Then for $t\ge 0$,
\[ \Prob\left(x_n\ge t\right)
 \kl \exp\left(-\frac{t^2}{4(1+\eps)D^2+{2(1+\eps)t M}/{\sqrt{\eps}}}-\frac{\sqrt{\eps}Mt^3}{2(1+\eps)(2\sqrt{\eps}D^2+M t)^2}\right)\pl \] for all $0<\eps\le1$.
\end{theorem}
Note that if $\eps=1$ the first term in our upper bound reduces to the same estimate as Oliveira's. In fact, the first term is always dominating.
\begin{proof}
We assume $M=1$ first. Let $\eps\in(0,1]$.  By exponential Chebyshev's inequality we have $\tau\left(1_{[t,\infty)} \left(x_n\right)\right)
 \le e^{-\la t}\tau\big(e^{\la x_n}\big)$ for $t>0$.
It follows from Lemma \ref{expin} that
\[\tau\left(1_{[t,\infty)}\left(x_n\right)\right) \kl \exp(-\la t+(1+\eps)\la^2D^2).\]
Now we set
\[\la \lel \frac{t}{2(1+\eps)D^2+{(1+\eps)t}/{\sqrt{\eps}}}\] which is less than $\sqrt{\eps}/(1+\eps)$. Then,
\begin{align*}
&-\la t+(1+\eps)\la^2D^2\lel -t^2\cdot\frac{1+t/(\sqrt{\eps} D^2)}{4(1+\eps)D^2[1+t/(2\sqrt{\eps}D^2)]^2}\\
\lel& -\frac{t^2}{4(1+\eps)D^2[1+t/(2\sqrt{\eps}D^2)]} - \frac{\sqrt{\eps}t^3}{2(1+\eps)(2\sqrt{\eps}D^2+t)^2}\pl.
%\lel -\frac{t^2}{4(1+\eps)D^2+2t(1+\eps)/\sqrt{\eps}}
\end{align*}
Replacing $t$ and $D$ with $t/M$ and $D/M$ respectively yields the assertion.
\end{proof}

Similar to the classical probability theory, we have for positive $a\in\M$ and for all $0<p<\8$,
\begin{equation}\label{pnorm}
\|a\|_p^p\lel p\int_0^\infty t^{p-1}{\rm Prob}(a>t) dt\pl.
\end{equation}
From here it is routine to estimate the $p$-th moment of $x_n$ using Theorem \ref{fman}.
\begin{prop}\label{pmom}
Under the assumption of Theorem \ref{fman}, for $2\le p<\infty$ we have
\begin{equation}\label{pmes}
\|x_n\|_p\kl
2^{3/2}(1+\eps)^{1/2}\sqrt{p}\Big\|\sum_{i=1}^nE_{i-1}(dx_i^2)\Big\|_\8^{1/2} + 2^{5/2}\Big(\frac{1+\eps}{\sqrt{\eps}}\Big)p\sup_{i=1, \cdots ,n}\|dx_i\|_\8\pl
\end{equation}
for all $0<\eps\le1$.
\end{prop}
\begin{proof}
Our strategy is to integrate the first term in Theorem \ref{fman}. The proof is similar to that of \cite{JZ}*{Corollary 0.3}. Note that it follows from symmetry that
\[
{\rm Prob}\left(|x_n| \ge t\right)\kl 2\exp\left(-\frac{t^2}{4(1+\eps)D^2+{2(1+\eps)t M}/{\sqrt{\eps}}}\right)\pl.
\]
Using \eqref{pnorm}, we obtain
\[\frac{\|x_n\|_p^p}{2p}\kl \int_0^{\frac{2\sqrt{\eps}D^2}{M}} t^{p-1}\exp\left(-\frac{t^2}{8(1+\eps)D^2}\right)dt+\int_\frac{2\sqrt{\eps}D^2}{M}^\8 t^{p-1} \exp\left(-\frac{t\sqrt{\eps}}{4(1+\eps) M}\right) dt.\]
Let us estimate the first term on the right hand side. Using the fact that $\Ga(x)\le x^{x-1}$ for $x\ge1$, we have
\begin{align*}
&\int_0^{\frac{2\sqrt{\eps}D^2}{M}} t^{p-1}\exp\left(-\frac{t^2}{8(1+\eps)D^2}\right)dt \lel 2^{3p/2-1}(1+\eps)^{p/2}D^p\int_0^{\frac{\eps D^2}{2M^2(1+\eps)}}r^{p/2-1}e^{-r} dr\\
\kl & 2^{3p/2-1}(1+\eps)^{p/2}D^p \int_0^\8 r^{p/2-1}e^{-r} dr \kl 2^{3p/2-1}(1+\eps)^{p/2}D^p ({p}/2)^{p/2-1}\\
\kl & 2^{p}(1+\eps)^{p/2}D^p p^{p/2-1}\pl.
\end{align*}
%But \[\int_0^{\frac{\eps D^2}{2M^2(1+\eps)}}r^{p/2-1}e^{-r} dr\kl \int_0^pr^{p/2-1}e^{-r} dr+\int_p^\8 r^{p/2-1}e^{-r} dr\kl 2p^{p/2-1}+2p^{p/2-1}\]
For the second term on the right hand side,
\begin{align*}
&\int_\frac{2\sqrt{\eps}D^2}{M}^\8 t^{p-1} \exp\left(-\frac{t\sqrt{\eps}}{4(1+\eps) M}\right) dt \\
\kl& 4^p\Big(\frac{1+\eps}{\sqrt{\eps}}\Big)^pM^p\int_0^\8 r^{p-1}e^{-r}dr
\kl  4^p\Big(\frac{1+\eps}{\sqrt{\eps}}\Big)^pM^p p^{p-1}\pl.
\end{align*}
Hence, we find
\[ \|x_n\|_p^p\kl 2^{p+1}(1+\eps)^{p/2}D^p p^{p/2}+2^{2p+1}\Big(\frac{1+\eps}{\sqrt{\eps}}\Big)^pM^p p^{p}.\]
This yields
\begin{align*}
\|x_n\|_p&\kl 2^{1+1/p}(1+\eps)^{1/2}D \sqrt{p}+ 2^{2+1/p}\Big(\frac{1+\eps}{\sqrt{\eps}}\Big)Mp  \\
&\kl 2^{3/2}(1+\eps)^{1/2}D \sqrt{p} + 2^{5/2}\Big(\frac{1+\eps}{\sqrt{\eps}}\Big)Mp\pl.
\end{align*}
Setting $D^2=\big\|\sum_{i=1}^n E_{i-1}(dx_i^2)\big\|$ and $M=\sup_{i=1, \cdots ,n}\|dx_i\|$ gives the assertion.
\end{proof}
Another way to obtain \eqref{sqp} would be an improved Burkholder inequality for noncommutative martingales:

\begin{prob}
  Is it true that for some function $f(p)$  and constant $C$,
 \begin{equation}
 \label{bbur}  \|\sum_k dx_k\|_p
 \kl C\sqrt{p}\|(\sum_k dx_k^*dx_k + dx_kdx_k^*)^{1/2}\|_p+f(p) \big(\sum_k \|dx_k\|_p^p\big)^{1/p}
 \end{equation}
holds for all noncommutative martingales.
\end{prob}
For independent increments this has recently been proved in \cite{JZ}. One would actually expect $f(p)=p$. As will become clear in the following, the validity of \eqref{bbur} would improve our main results and imply a number of results in different contexts. At the time of this writing we are unable to decide whether \eqref{bbur} holds. However, the commutative case was known to be true due to the work of Pinelis \cite{Pin}, who attributed it to Hitczenko.

\section{Noncommutative Burkholder--Davis--Gundy type inequalities}
We refer the readers to \cites{JKo, JMe, JP} for further details about the facts mentioned in this section.
Let $x=(E_1x, \cdots ,E_nx)$ be a (finite) martingale sequence with martingale differences $dx_k$. For $1\le p\le \8$, we define
%\|x\|_{H_p^c}\lel \left\|\Big(\sum_{k}d_kx^* d_kx\Big)^{1/2}\right\|_p,
\[ \|x\|_{h_p^c} \lel \left\|\Big(\sum_{k}E_{k-1}(dx_k^* dx_k)\Big)^{1/2}\right\|_p,\quad \|x\|_{h_p^r}=\|x^*\|_{h_p^c}\pl,\]
and \[\|x\|_{h_p^d}=\Big(\sum_{k}\|dx_k\|_p^p\Big)^{1/p}. \]

We are going to use the continuous filtration $(\N_t)_{t\ge0}\subset \N$ in the following. Recall that a martingale $x$ is said to have almost uniform (or a.u. \!for short) continuous path if for every $T>0$, every $\eps>0$ there exists a projection $e$ with $\tau(1-e)<\eps$ such that the function $f_e:[0,T]\to \N$ given by $f_e(t)=x_te\in \N$ is norm continuous. Let $\si=\{0=s_0, \cdots ,s_n=T\}$ be a partition of the interval $[0,T]$ and $|\si|$ its cardinality. Put
\[
\|x\|_{h_p^c([0,T];\si)} \lel \Big\|\sum_{j=0}^{|\si|-1}E_{s_{j}}|E_{s_{j+1}}x-E_{s_j}x|^2 \Big\|^{1/2}_{p/2}, \quad 2\le p\le \8\pl,
\]
\[
 \|x\|_{h_p^d([0,T];\si)} \lel \Big(\sum_{j=0}^{|\si|-1}\|E_{s_{j+1}}x-E_{s_j}x\|_p^p \Big)^{1/p}, \quad 2\le p<\8\pl,
\]
and $\|x\|_{h_p^r([0,T];\si)}=\|x^*\|_{h_p^c([0,T];\si)}$. Let $\U$ be an ultrafilter refining the natural order given by inclusion on the set of all partitions of $[0,T]$. Let $x\in L_p(\N)$. For $2\le p< \8$, we define
\[\langle x,x\rangle_T\lel \lim_{\si, \U}\sum_{i=0}^{|\si|-1}E_{s_i}|E_{s_{i+1}}x-E_{s_i}x|^2\pl.\]
Here the limit is taken in the weak* topology and it is shown in \cite{JKo} that the convergence is also true in  $L_{p}$ norm $\|\cdot\|_{p/2}$ for all $2<p<\8$. We define the continuous version of $h_p$ norms for $2\le p<\8$,
\[\|x\|_{h_p^c([0,T])}\lel\lim_{\si, \U}\|x\|_{h_p^c([0,T];\si)} \pl,\]
\[\|x\|_{h_p^d([0,T])}\lel\lim_{\si, \U}\|x\|_{h_p^d([0,T];\si)} \pl.\]
and $\|x\|_{h_p^r([0,T])}=\|x^*\|_{h_p^c([0,T])}$ for $2\le p < \8$. Then for all $2<p<\8$
\begin{equation}\label{hpc}
 \|x\|_{h_p^c([0,T])}\lel \|\lge x,x\rge_T\|^{1/2}_{p/2}\pl.
\end{equation}
A martingale $x$ is said to be of vanishing variation if $\|x\|_{h^d_p([0, T])}=0$ for all $T>0$ and all $2<p<\infty$.
We also write
\[\var_p(x)=\|x\|_{h_p^d([0,T])},\quad \]
and let $V_p(\N)$ denote the $L_2(\N)$ closure of $\{x\in L_p(\N): \var_p(x)=0\}$.

The following results are proved in \cite{JKo}. For any $y\in L_p(\N)$, we write $d_j y= E_{s_{j}}y-E_{s_{j-1}}y$. Put $\|x\|_{L_p(\var)}=\sup_{\si} \|(d_jx)\|_{L_p(\ell_1)}$, where the supremum is taken over all finite partitions of $[0,T]$, and the norm $\|\cdot\|_{L_p(\ell_1)}$ was defined in \cite{Ju}, which we will not use after the next result.
\begin{theorem}\label{decom}
Let $2<p<\8$ and $x\in L_p(\N_T)$. Then for all $\de>0$, there exists a decomposition $x=y^\de+z^\de$ satisfying the following
\begin{enumerate}
 \item $\var_p(y^\de)<\de$, $z^\de\in L_p(\var)$.
 \item Let $P(x)=w^*\dash\lim_\de y^\de$. Here $w^*\dash\lim$ denotes the weak* limit. Then $P: L_p(\N)\to V_p(\N)$ is an orthogonal projection.
 \item $P(x)=x$ for all $x$ with vanishing variation.
\end{enumerate}
One {may take} $y^\de=w^*\dash\lim_\si \sum_{j=1}^{|\si|} d_j(d_jx1_{[|d_jx|\le \de]})$ where $1_B$ is the spectral projection of $d_jx$ restricted to the Borel set $B$.
\end{theorem}
\begin{lemma}\label{auct}
 If $x$ has a.u. continuous path, then it is of vanishing variation.
\end{lemma}

Now let us prove the main results of this section, which can be regarded as the noncommutative version of Burkholder--Davis--Gundy type inequalities.

\begin{theorem}\label{ctpm}
Let $x$ be a mean 0 martingale with a.u. \hspace{-0.2em}continuous path. Then for every $T>0$, we have
\begin{enumerate}
 \item For $2\le p<\8$, if $x$ is self-adjoint, then
\[\|E_T x\|_p\kl C\sqrt{p}\liminf_{\si,\U}\|x\|_{h_\8^c([0,T];\si)}\pl.\]
 If $x$ is not necessarily self-adjoint, then
\[\|E_T x\|_p\kl C\sqrt{p}\liminf_{\si,\U}\big(\|x\|_{h_\8^c([0,T];\si)}+\|x\|_{h_\8^r([0,T];\si)}\big)\pl,\]
where we may take $C=2\sqrt{2}$.
\item For all $2\le p < \8$,
\[\|E_T x\|_p\kl C'{p}\max\big\{\|x\|_{h_p^c([0,T])},\pl\|x\|_{h_p^r([0,T])}\big\}\pl.
\]
\end{enumerate}
\end{theorem}
\begin{proof}
(1) First assume that $x$ is self-adjoint and that $x\in \N_T$. We follow the strategy used in the proof of Theorem \ref{decom}. Fix a partition $\si$ of $[0,T]$. We write $h_p(\si)$ for $h_p([0,T];\si)$ in the following proof. Let $\de>0$. We have
$d_jx = d_jx 1_{[|d_jx|>\de]}+d_jx 1_{[|d_jx|\le\de]}$. Conditioning again, we obtain
\[d_jx\lel d_j(d_jx 1_{[|d_jx|>\de]})+d_j(d_jx 1_{[|d_jx|\le \de]})\pl.\]
Put $z^\de_\si = \sum_{j=1}^{|\si|}d_j(d_jx 1_{[|d_jx|>\de]}) $ and $y^\de_\si=\sum_{j=1}^{|\si|}d_j(d_jx 1_{[|d_jx|\le \de]})$. Then clearly
\[\sup_{j=1, \cdots , |\si|} \|d_j(d_jx 1_{[|d_jx|\le \de]})\|_\8 \kl 2\de \pl. \]
Using Proposition \ref{pmes} for some fixed $0<\eps \le 1$, we find
\begin{equation}\label{part}
 \|y^\de_\si-\tau(y^\de_\si)\|_p\kl 2^{3/2}(1+\eps)^{1/2}\sqrt{p}\|y^\de_\si\|_{h^c_\8(\si)}+2^{7/2} \Big(\frac{1+\eps}{\sqrt{\eps}}\Big) p\de\pl.
 \end{equation}
%Note that $1_{[|d_jx|\ge\eps]}$ and $1_{[|d_jx|<\eps]}$ are orthogonal, we have
Note that
\begin{align*}
  0&\kl E_{s_{j-1}}|d_j(d_jx 1_{[|d_jx|<\de]})|^2\lel E_{s_{j-1}}[(d_jx 1_{[|d_jx|<\de]})^2]-[E_{s_{j-1}}(d_jx 1_{[|d_jx|<\de]})]^2\\
  &\kl E_{s_{j-1}}[(d_jx 1_{[|d_jx|<\de]})^2]\kl E_{s_{j-1}}[(d_jx)^2]\pl.
  \end{align*}
Then we have
\begin{align*}
  &\|y^\de_\si\|_{h^c_\8(\si)} \lel \Big\|\sum_{i=0}^{|\si|-1}E_{s_{j-1}}|d_j(d_jx 1_{[|d_jx|<\de]})|^2\Big\|_{\8}^{1/2}\\
  %\kl &  \Big\|\sum_{i=0}^{|\si|-1}E_{s_{j-1}}|d_j(d_jx 1_{[|d_jx|\ge \eps]})|^2+ \sum_{i=0}^{|\si|-1}E_{s_{j-1}}|d_j(d_jx 1_{[|d_jx|<\eps]})|^2\Big\|^{1/2}_\8\\
  \kl& \Big\|\sum_{i=0}^{|\si|-1}E_{s_{j-1}} |d_jx|^2\Big\|^{1/2}_\8\lel \|x\|_{h^c_\8(\si)}.
  \end{align*}
According to Theorem \ref{decom} (in our context, $y^\de =w^*\dash\lim_\si y^\de_\si$) and Lemma \ref{auct}, we have $x=w^*\dash\lim_{\de\to 0} y^\de$. Hence for any $\la_i\ge0,\sum_{i=1}^k \la_i=1$, we have
\[x\lel w^*\dash\lim_{\substack{\de_i\to 0\\ i=1, \cdots ,k}}\sum_{i=1}^k \la_i y^{\de_i}\pl.\]
Since in a Banach space the weak closure and the norm closure of a convex set are the same, by the reflexivity of $L_p(\N)$ we can find a net $x_\al$ in the convex hull of $\{y^\de\}$ such that $x_\al \to x$ in $L_p(\N)$. Therefore by sending $\de\to 0$, we deduce from \eqref{part} that
\[\|x\|_p\kl 2\sqrt{2}(1+\eps)^{1/2}\sqrt{p}\liminf_{\si,\U}\|x\|_{h_\8^c(\si)}\pl,\]
for all $0<\eps\le 1$. Sending $\eps\to 0$ yields the first assertion.
If $x$ is not self-adjoint, we write $x=\Re(x)+i\Im(x)$ where $\Re(x)=\frac{x+x^*}2$ and $\Im(x)=\frac{x-x^*}{2i}$. Then the second assertion follows from the self-adjoint case by triangle inequality.

(2) Since $x$ is of vanishing variation, $\|x\|_{h_p^d([0,T])}=0$ for all $2< p<\8$. Using \eqref{burk}, we have
\[
 \|x\|_p\kl C'p (\|x\|_{h_p^c([0,T];\si)}+\|x\|_{h_p^r([0,T];\si)})+p\|x\|_{h_p^d([0,T];\si)}\pl.
\]
Taking limits on the right hand side yields the assertion for $2<p<\8$. The case $p=2$ is proved by sending $p\downarrow 2$.
\end{proof}

\section{Poincar\'{e} type inequalities and applications}
\subsection{Poincar\'{e} type inequalities}
Let $(T_t)_{t\ge0}$ be a semigroup of operators acting on a finite von Neumann algebra $(\N,\tau)$ where $\tau(1)=1$. Following \cites{JMe, JMe2} we say $(T_t)$ is a {\em standard} semigroup if it satisfies the following assumptions:
\begin{enumerate}
 \item Every $T_t$ is a normal completely positive map on $\N$ such that $T_t(1) = 1$;
 \item Every $T_t$ is self-adjoint, i.e. $\tau(T_t (x) y)=\tau(x T_t(y))$ for all $x,y\in\N$.
 \item The family $(T_t)$ is pointwise weak* continuous. Equivalently, $\lim_{t\to t_0} T_t x = T_{t_0}x$ with respect to the strong operator topology in $\N$ for any $x\in \N$; see \cite{MS10}.
% \item There exists a weakly dense self-adjoint subalgebra $\A\subset \N$ such that $A(\A)\subset \A$ and $T_t(\A)\subset \A$.
\end{enumerate}
It is well known that Assumption (3) is further equivalent to that $(T_t)$ is a strongly continuous semigroup on $L_2(\N,\tau)$, where $T_t$ extends to $L_2(\N,\tau)$ by $T_t \La(x)=\La(T_t x)$ for $x\in \N$ and $\La: \N\to L_2(\N,\tau)$ is the natural embedding. By \cite{DL}, $(T_t)$ extends to a strongly continuous contraction semigroup on $L_p(\N)$ for every $1\le p< \8$ with generator $A$, i.e. $T_t=e^{-tA}$. Write for $1\le p<\8$
\[ \Dom_p(A)=\{f\in L_p(\N): \lim_{t\to0} (f-T_tf)/t \mbox{ converges in } L_p(\N)\}.\]
Then the classical semigroup theory asserts that $\Dom_p(A)$ is dense in $L_p(\N)$ and that if $x\in\Dom_p(A)$ then $T_tx\in\Dom_p(A)$. We also denote $\Dom(A)=\Dom_2(A)$. Note that $A$ is a positive operator on $L_2(\N,\tau)$. The standard assumptions also imply that $\tau(T_tx)=\tau(x)$ and thus $T_t$'s are faithful. In addition, $T_t$ is a contraction on $\N$. Indeed, for $x\in\N$, we have
\[\|T_t x\|_\8=\sup_{\|y\|_1\le 1} |\tau((T_tx)y)|=\sup_{\|y\|_1\le 1} |\tau(x(T_ty))|\le \sup_{\|y\|_1\le 1} \|T_ty\|_1\|x\|_\8\le \|x\|_\8\pl.
\]

Recall that $T_t$ is said to admit a reversed Markov dilation if
\begin{enumerate}
\item[(H1)] there exists a larger finite von Neumann algebra $\M$ and a family $\pi_t: \N\to \M$
of trace preserving $*$-homomorphism;
\item[(H2)] there is a decreasing filtration $(\M_{[s})_{0\le s<\8}$ with $\pi_r(x)\in \M_{[s}$ for all $r>s$ such that $E_{[s}(\pi_t(x))=\pi_s(T_{s-t}x)$ for all $t<s$ and $x\in \N$.
\end{enumerate}
Here we have $\M_{[t}=E_{[t}(\M)$. For elements $x,y\in \Dom(A)$ we may define the gradient form, which is called Meyer's ``carr\'e du champ'' in the commutative theory,
\[2\Ga(x,y)=A(x^*)y+x^*A(y)-A(x^*y)\]
and for $x,y \in \Dom(A^2)$ the second order gradient
\[2\Ga_2(x,y)=\Ga(Ax,y)+\Ga(x, Ay)-A\Ga(x,y).\]
Recall that $(T_t)$ is called a noncommutative diffusion (or nc-diffusion for short) semigroup if $\Ga(x,x)\in L_1(\N)$ for all $x\in \Dom (A^{1/2})$. If $(T_t)$ is nc-diffusion, then $\Ga(x,x)\in L_1(\N)$ is well-defined for $x\in\Dom(A^{1/2})$ by extension. By duality, $\Ga(x,x)\in L_p(\N)$ for $1\le p< \8$ if and only if there exists $C>0$ such that $|\tau(\Ga(x,x)y)|\le C \|y\|_{p'}$ for all $y$ and $1/p+1/p'=1$.

We will use the following crucial results proved by Junge, Ricard, and Shlyakhtenko in \cite{JRS}, which is a noncommutative version of the Stroock--Varadhan martingale problem.
\begin{theorem}\label{jrs}
Suppose that $(T_t)_{t\ge0}$ is a standard nc-diffusion semigroup. Then $T_t$ admits a reversed Markov dilation $({\pi}_t)$ with a.u.\! continuous path, i.e. in addition to (H1) and (H2), for all $x\in \Dom(A)$ and all $S>0$,
$${m}_s(x):={\pi}_s(T_s(x)),\quad 0\le s \le S $$
is a (reversed) martingale with a.u.\! continuous path.
\end{theorem}
\re\label{doma}
Let $2\le p <\8$. For the purpose of our main result, we extend the theorem to $x\in\Dom(A^{1/2})$. Indeed, since $\Dom(A^{1/2})\cap \N$ is a $*$-subalgebra of $\N$ by \cite{DL} and $\Dom(A)$ is dense in $L_2(\N)$, there exists a sequence $(x_n)\in \Dom(A)$ such that $\lim_{n\to\8}\|x_n-x\|_2=0$. But $E_{[r}({\pi}_s T_s(x_n))={\pi}_r T_r(x_n)$ for $s<r$. Taking limits on both sides, we find $E_{[r}({\pi}_s T_s(x))={\pi}_r T_r(x)$ in $L_2(\N)$. According to \cite{Luc}, the set of a.u. \!continuous path martingales is closed in $L_2(\N)$. Hence ${\pi}_s T_s(x)$ has a.u. \!continuous path. Similar argument applies to the forward martingales, but we only need the reversed martingales in this paper.
\mar
%It follows from \cite{JMe}*{Lemma 2.2.1} that $m_t$ is a martingale and a straightforward verification shows that $\td{m}_t$ is a (reversed) martingale.
Put $L_p^0(\N)=\{x\in L_p(\N): \lim_{t\to \8}T_tx=0 \}$ for $1\le p\le \8$. Here the limit is taken with respect to $\|\cdot\|_{L_p(\N)}$ for $1\le p<\8$ and with respect to the weak* topology for $p=\8$. Let ${\rm Fix}=\{x\in\N: T_t x=x \mbox{ for all }t\ge0\}$. Then it was shown in \cite{JX07} that ${\rm Fix}$ is a von Neumann subalgebra and ${\rm Fix}^\perp=L^0_\8(\N)$. Denote by $E_{\rm Fix}: \N\to{\rm Fix}$ the conditional expectation which extends to a contraction on $L_p(\N)$. Then for all $x\in L_p(\N)$ we have $x-E_{\rm Fix} x\in L_p^0(\N)$ and $L_p^0(\N)$ is a complemented subspace of $L_p(\N)$.
\begin{lemma}\label{brac}
Let $2\le p<\8$ and $(T_t)_{t\ge0}$ be a standard nc-diffusion semigroup. Then for all $0\le s<t\le \8$, and $x\in \Dom(A^{1/2})\cap L_p^0(\N)$ with $\Ga(T_rx,T_rx)$ uniformly bounded for $r\ge 0$ in $L_p(\N)$  we have
\[\|{m}(x)\|_{h^c_p([s,t])}\lel \left\|2\int_s^t{\pi}_r(\Ga(T_{r}x, T_{r}x))dr\right\|^{1/2}_{p/2}\pl.\]
\end{lemma}
\begin{proof}
By Theorem \ref{jrs}, Remark \ref{doma} and Lemma \ref{auct}, $\var_p(m)=0$ for all $2<p<\8$. \eqref{hpc} implies for $2<p<\8$,
\begin{equation*}
 \|m\|_{h_p^c([s,t])}\lel \|\lge m,m\rge_t- \lge m, m \rge_s\|^{1/2}_{p/2}\pl.
\end{equation*}
It follows from \cite{JMe}*{Lemma 2.4.1} and uniform boundedness that
\[\lge m, m\rge_s-\lge m, m\rge_t \lel 2\int_s^t{\pi}_r(\Ga(T_{r}x, T_{r}x))dr\pl.\]
Here the integral when $t=\8$ is well-defined for $x\in L_p^0(\N)$ according to \cite{JMe}*{Proposition 2.4.3}.
This gives the assertion for $2<p<\8$. The case $p=2$ follows by sending $p\downarrow 2$.
\end{proof}

We are now ready to state our main result of this section.
\begin{theorem}\label{poin}
Suppose $2\le p<\8$. Let $T_t=e^{-tA}$ be a standard nc-diffusion semigroup and $\Ga$ the gradient form associated with $A$. Assume $x\in L_p(\N)\cap \Dom(A^{1/2})$ satisfies
\begin{equation}\label{ga2}
\tau(y\Ga(T_t x,T_t x))\kl e^{-2\al t} \tau(y T_t\Ga(x,x)), \quad  y\in \N, y\ge0\pl,
\end{equation}
for some $\al>0$. Then we have the following Poincar\'e type inequalities
\begin{equation}\label{pcr1}
 \|x-E_{\rm Fix} x\|_p \kl C\sqrt{p/\al}\max\{\|\Ga(x,x)^{1/2}\|_\8, \Ga(x^*,x^*)^{1/2}\|_\8\},
 \end{equation}
\begin{equation}
 \label{pcr2}
 \|x-E_{\rm Fix} x\|_p \kl C'\al^{-1/2}p\max\{\|\Ga(x,x)^{1/2}\|_p,\Ga(x^*,x^*)^{1/2}\|_p\},
\end{equation}
where we can take $C=4\sqrt{2}$ in general and $C=2\sqrt{2}$ if $x$ is self-adjoint.
%Moreover, the order $\sqrt{p}$ in \eqref{pcr1} is optimal.
\end{theorem}
\begin{proof}
First assume $2<p<\8$. Notice that $E_{\rm Fix} x$ is in the multiplicative domain of $T_t$. Then $\Ga(x,x)=\Ga(x-E_{\rm Fix} x,x-E_{\rm Fix} x)$. Without loss of generality we may assume $x\in L_p^0(\N)$, which implies  $\lim_{t\to\8}T_tx=E_{\Fix}(x)=0$ in $L_p$.  Fix a constant $0<M<\8$ and consider the reversed martingale ${m}_t(x)$ in Theorem \ref{jrs} for $t\in[0,M]$. By Theorem \ref{ctpm} (applied to reversed martingales), noticing that ${m}_0(x)={\pi}_0(x)$, we have
\begin{align*}
&\|{m}_0-E_{[M}({m}_0)\|_p \\
\kl& C\sqrt{p}\liminf_{\si,\U}\big(\|{m}_0-E_{[M}({m}_0)\|_{h_\8^c([0,M];\si)} +\|{m}_0-E_{[M}({m}_0)\|_{h_\8^r([0,M];\si)}\big)\pl.
\end{align*}
%&\lel 4C\sqrt{p}\left\|\int_0^\8 \td{\pi}_r(\Ga(T_{r}x, T_{r}x))dr \right\|_\8^{1/2}\pl.
Using the reversed Markov dilation and \cite{JMe}*{Lemma 2.4.1 (iii)}, we find (see \cite{JMe}*{(2.12)})
\begin{align*}
 \|{m}_0-E_{[M}({m}_0)\|_{h_\8^c([0,M];\si)}&\lel \Big\|\sum_{j=0}^{|\si|-1} E_{[s_{j+1}}|{m}_{s_j}(x)-{m}_{s_{j+1}}(x)|^2\Big\|^{1/2}_{\8}\\
 &\lel \Big\|\sum_{j=0}^{|\si|-1} E_{[s_{j+1}}(\pi_{s_j}(|T_{s_j}x|^2)) -\pi_{s_{j+1}}(|T_{s_{j+1}}x|^2)\Big\|^{1/2}_{\8}\\
 &\lel \Big\|\sum_{j=0}^{|\si|-1} \pi_{s_{j+1}}(T_{s_{j+1}-s_j}|T_{s_j}x|^2 -|T_{s_{j+1}}x|^2)\Big\|^{1/2}_{\8}\\
 &\lel \Big\|2\sum_{j=0}^{|\si|-1}\pi_{s_{j+1}}\Big(\int_{0}^{s_{j+1}-s_j} T_{s_{j+1}-s_j-r} (\Ga(T_{r+s_j} x,T_{r+s_j} x)) dr\Big) \Big\|^{1/2}_{\8}\\
 &\lel \Big\|2\sum_{j=0}^{|\si|-1}\int_{s_j}^{s_{j+1}}E_{[s_{j+1}} {\pi}_r(\Ga(T_rx,T_rx)) dr \Big\|^{1/2}_{\8}\pl.
 \end{align*}
Since $E_{[s_{j+1}}$ and ${\pi}_r$ are contractions, we deduce from \eqref{ga2} that
\begin{align*}
   &\|{m}_0-E_{[M}({m}_0)\|_{h_\8^c([0,M];\si)} \kl \sqrt{2}\Big(\sum_{j=0}^{|\si|-1}\int_{s_j}^{s_{j+1}} \|\Ga(T_rx,T_rx)\|_\8 dr\Big)^{1/2} \\
   \lel& \sqrt{2}\Big(\sum_{j=0}^{|\si|-1}\int_{s_j}^{s_{j+1}}\sup_{y\ge0,y\in \N, \|y\|_1\le 1}\tau(y\Ga(T_rx,T_rx))dr\Big)^{1/2}\\
   \kl & \sqrt{2}\Big(\sum_{j=0}^{|\si|-1}\int_{s_j}^{s_{j+1}}e^{-2\al r}\sup_{y\ge0,y\in \N, \|y\|_1\le 1}\tau(yT_r\Ga(x,x))dr\Big)^{1/2}\\
   \lel &\sqrt{2}\Big(\sum_{j=0}^{|\si|-1}\int_{s_j}^{s_{j+1}} e^{-2\al r} \|T_r\Ga(x,x)\|_\8 dr \Big)^{1/2}
   \kl \sqrt{2}\Big(\int_0^M  e^{-2\al r} \|\Ga(x,x) \|_\8 dr \Big)^{1/2}\\
   \kl& \al^{-1/2}\|\Ga(x,x)\|_\8^{1/2}\pl.
   \end{align*}
 Similarly,
 \[\|{m}_0-E_{[M}({m}_0)\|_{h_\8^r([0,M];\si)} \kl \al^{-1/2}\|\Ga(x^*,x^*)\|_\8^{1/2}\pl.\]
Hence we have
\[ \|{m}_0-E_{[M}({m}_0)\|_p \kl C\Big(\frac{p}{\al}\Big)^{1/2} (\|\Ga(x,x)\|_\8^{1/2}+\|\Ga(x^*,x^*)\|_\8^{1/2})\pl.\]
By the reversed Markov dilation,
\[\|E_{[M}({m}_0)\|_p=\|E_{[M}({\pi}_0(x))\|_p=\|\pi_M T_M x\|_p\le \|T_M x\|_p\pl.\]
Note that $\lim_{M\to\8}\|T_M x\|_p=0$ and that $\|x\|_p=\|{m}_0\|_p\le  \|{m}_0-E_{[M}({m}_0)\|_p+\|E_{[M}({m}_0)\|_p$.
Sending $M\to \8$ gives the first assertion for $2<p<\8$. Sending $p\downarrow 2$ gives the case $p=2$.

For \eqref{pcr2}, note that \eqref{ga2} implies $\Ga(T_t x,T_t x)$ is uniformly bounded in $L_p(\N)$. Then Theorem \ref{ctpm} and Lemma \ref{brac} imply that for $M>0$, $2<p<\8$, and $x\in\Dom(A^{1/2})$, we have
\begin{align*}
&\|{m}_0-E_{[M}({m}_0)\|_p\\
\kl &\sqrt{2}C'{p}\max\left\{\Big\|\int_0^M {\pi}_r(\Ga(T_{r}x, T_{r}x))dr \Big\|_{p/2}^{1/2},\pl\Big\|\int_0^M {\pi}_r(\Ga(T_{r}x^*, T_{r}x^*))dr \Big\|_{p/2}^{1/2}\right\}\pl.
\end{align*}
Similar to the above argument, \eqref{ga2} yields
\[
\Big\|\int_0^M {\pi}_r(\Ga(T_{r}x, T_{r}x))dr \Big\|_{p/2}^{1/2} \kl \frac1{\sqrt{2\al}}\|\Ga(x,x)\|_{p/2}^{1/2}\pl.
\]
The rest of proof is the same as that of the first assertion.
\end{proof}
If $A$ has a spectral gap in $L_p$, we can deduce the second inequality \eqref{pcr2} from the main result of \cite{JMe} on the noncommutative Riesz transform. However, we have explicit order $p$ here. So far as we know, no previous method has achieved the order $\sqrt{p}$ in the first inequality in the noncommutative setting.
\re
In fact, if \eqref{ga2} holds for all $x\in\Dom(A^{1/2})$, then $T_t$ is a nc-diffusion semigroup. Indeed, it was proved in \cite{JRS} that {$T_t\Ga(x,x)\in L_1(\N)$ for $t > 0$.} Then \eqref{ga2} implies that $\Ga(T_tx, T_tx)\in L_1(\N)$. Taking limit gives $\Ga(x,x)\in L_1(\N)$.
\mar
Condition \eqref{ga2} is not convenient to check. In practice, we may pose stronger assumptions which are easy to verify. The following lemma is of course well known in the commutative case.

\begin{lemma}\label{ga22}
 Let $T_t=e^{-tA}$ be a standard nc-diffusion semigroup. Let $x\in\N$ be such that $\Ga_2(x,x)$ is well-defined. Then $\Ga_2(x,x)\ge \al\Ga(x,x)$ implies \eqref{ga2}.
\end{lemma}
\begin{proof}
 Since $T_t$ is positive, $\al T_t\Ga(x,x)\le T_t\Ga_2(x,x)$. Let $\td{T}_t=e^{2\al t}T_t$. Consider the function $$f(s)=\td{T}_{t-s}\Ga(\td{T}_sx,\td{T}_sx)=e^{2\al(t+s)}T_{t-s}\Ga(T_sx,T_sx).$$ Due to the assumption, $f(s)$ is differentiable. Then
\begin{align*}
f'(s)\lel& 2\al e^{2\al(t+s)}T_{t-s}\Ga(T_sx,T_sx)+e^{2\al(t+s)}T_{t-s}A\Ga(T_sx,T_sx)\\
&-e^{2\al(t+s)}T_{t-s}[\Ga(AT_sx,T_sx) +\Ga(T_sx,AT_sx)]\\
\lel& 2\al e^{2\al(t+s)}T_{t-s}\Ga(T_sx,T_sx) - 2e^{2\al(t+s)}T_{t-s}\Ga_2(T_sx,T_sx)\kl 0
\end{align*} for all $0< s< t$.
We have by continuity $\Ga(\td{T}_t x,\td{T}_tx)=f(t)\kl f(0)=\td{T}_t\Ga(x,x)$, or $\Ga(T_t x,T_t x)\le e^{-2\al t} T_t\Ga(x,x)$ which implies \eqref{ga2}.
\end{proof}
For the purpose of future development, let us recall the definition of positive forms. Suppose $\Theta: \N\times \N\to L_1(\N,\tau)$ is a sesquilinear form whose domain is a weakly dense $*$-subalgebra $\Dom(\Theta)$ such that $1\in \Dom(\Theta)$. In this paper, we follow the convention that a sesquilinear form is conjugate linear in the first component. $\Theta$ is said to be positive if for all $n\in \nz$, $x_1, \cdots ,x_n\in \Dom(\Theta)$, $(\Theta(x_i,x_j))_{i,j=1}^n$ is positive in $M_n(L_1(\N))\cong L_1(M_n\bar\otimes \N)$. Given another sesquilinear form $\Phi$, $\Theta\ge\Phi$ if $\Theta-\Phi\ge0$. We refer the readers to \cites{Sau, Pet}  for more details. For any $n\in\nz$, and any $a_1, \cdots ,a_n\in \Dom(A)$ the $n\times n$ matrix $(\Ga(a_i,a_j))_{i,j=1}^n$ with entries in $\N$ is positive in $M_n(\N)$. The following useful fact was due to Peterson \cite{Pet}; see also \cite{Sau} for the implication ``$\Rightarrow$''.
\begin{theorem}\label{cps}
 Let $(T_t)$ be a strongly continuous semigroup on $L_2(\N)$. Then $(T_t)$ is a completely positive semigroup if and only if $\Ga$ is a positive form.
\end{theorem}

As in the commutative case, the domain of $\Ga$ and $\Ga_2$ is a delicate issue. Theorem \ref{poin} avoided this difficulty by considering individual element. In many cases we are interested in the Poincar\'e type inequalities for the whole space. Our next result is meant for this purpose.
\begin{cor}\label{ga2a}
 Let $T_t=e^{-tA}$ be a standard nc-diffusion semigroup. Suppose that there exists a weakly dense self-adjoint subalgebra $\A\subset \N$ such that \\
 i) $A(\A)\subset \A$; ii) $T_t(\A)\subset \A$; iii) $\A$ is dense in $\Dom(A^{1/2})$ in the graph norm of $A^{1/2}$.\\
 Assume $\Ga_2(x,x)\ge \al \Ga(x,x)$ for some $\al>0$ and for all $x\in\A$. Then \eqref{ga2} holds for all $x\in \Dom(A^{1/2})$. Moreover, all $x\in L_p(\N)$ satisfies \eqref{pcr1} and \eqref{pcr2}.
\end{cor}
\begin{proof}
For $x\in \Dom(A^{1/2})$ we deduce from assumption iii) that there exist $(x_n)\subset \A$  with $\Ga_2(x_n,x_n)\ge\al\Ga(x_n,x_n)$ such that $\|x_n-x\|_2\to 0$ and $\| A^{1/2}x_n-A^{1/2}x\|_2\to 0$ as $n\to \8$. By \cite{CSa}*{Section 9}, we have $\|\Ga(x,x)\|_1=\lge A^{1/2} x, A^{1/2}x\rge_{L_2(\N)}$. Since $\Ga$ is a complete positive form, we have for $x,y\in\Dom(A^{1/2})$,
\[
 \left(\begin{array}{cc}
  \Ga(x,x)&  \Ga(x,y)\\
  \Ga(y,x) & \Ga(y,y)
 \end{array}
\right)\ge 0\pl.
\]
Note that $\Ga(x,y)^*=\Ga(y,x)$. Then $\|(\Ga(x,x)+\eps 1 )^{-1/2}\Ga(x,y)(\Ga(y,y)+\eps 1)^{-1/2}\|_\8\le 1$ for any $\eps>0$. Hence
\begin{align*}
&\|\Ga(x,y)\|_1\\
\kl& \|(\Ga(x,x)+\eps 1)^{1/2}\|_2 \|(\Ga(x,x)+\eps 1)^{-1/2}\Ga(x,y)(\Ga(y,y)+\eps 1)^{-1/2}\|_\8\|(\Ga(y,y)+\eps 1)^{1/2}\|_2\\
\kl& \|\Ga(x,x)+\eps 1\|_1^{1/2}\|\Ga(y,y)+\eps\|_1^{1/2}\pl.
\end{align*}
Sending $\eps\to 0$, we have $\|\Ga(x,y)\|_1\le \|\Ga(x,x)\|_1^{1/2}\|\Ga(y,y)\|_1^{1/2}$. It follows that
\begin{align*}
 &\|\Ga(x_n,x_n)-\Ga(x,x)\|_1\kl \|\Ga(x_n-x,x_n-x)\|_1+2\|\Ga(x_n-x,x)\|_1\\
 \kl& \lge A^{1/2} (x_n-x), A^{1/2}(x_n-x)\rge_{L_2(\N,\tau)} +2 \lge A^{1/2} (x_n-x), A^{1/2}(x_n-x)\rge^{1/2}_{L_2(\N,\tau)} \|\Ga(x,x)\|_1^{1/2} \pl.
\end{align*}
Hence $\lim_{n\to\8} \Ga(x_n,x_n)=\Ga(x,x)$ in $L_1(\N)$. Notice that $T_t$ and $A^{1/2}$ commute. Then for all $t\ge0$ and $x\in\Dom(A^{1/2})$, $T_t x\in\Dom(A^{1/2})$ and a similar argument as above gives that $\lim_{n\to\8} \Ga(T_t x_n,T_t x_n)=\Ga(T_t x,T_t x)$ in $L_1(\N)$. Since Lemma \ref{ga22} implies
$$\tau(y\Ga(T_t x_n,T_t x_n))\le e^{-2\al t} \tau(y T_t\Ga(x_n,x_n))$$
for all $y\in\N,y\ge0$, sending $n\to\8$ on both sides yields the first assertion. For the ``moreover'' part, note that we only need to prove \eqref{pcr1} and \eqref{pcr2} for
$$\max\{\|\Ga(x,x)^{1/2}\|_p, \|\Ga(x^*,x^*)^{1/2}\|_p\}<\8.$$
Recall that $\Gamma(x,x)$ is understood as the weak* limit of $\Gamma^{A_\eps}(x,x)$ where $A_\eps = (I+\eps A)^{-1} A$ (see \cite{CSa}*{(3.2)}). If this limit exists in $L_{p/2}$ for $p> 2$, then $\tau(\Gamma(x,x))$ is finite and hence $x\in \Dom(A^{1/2})$.  The individual result Theorem \ref{poin} then comes into play and completes the proof.
\end{proof}
The condition $\Ga_2(x,x)\ge \al \Ga(x,x)$ we posed here is usually called the curvature condition or $\Ga_2$-criterion.
The expression $\max\{\|\Ga(x,x)^{1/2}\|_\8,~\|\Ga(x^*,x^*)^{1/2}\|_\8\}$ is the so-called Lipschitz norm in the commutative theory. In the classical diffusion setting, Bakry and Emery \cite{BE85} showed that the $\Ga_2$-criterion implies logarithmic Sobolev inequality, which in turn yields the $L_p$ Poincar\'e inequalities with constant $C\sqrt{p}$ due to Aida and Stroock \cite{AS94}; see also \cite{AW13} for another proof. In the general non-diffusion setting, we will show that the first implication is not true. At the time of this writing, we do not know whether the $L_p$ Poincar\'e inequalities follow from LSI in full generality. However, adapting our theory to the classical diffusion setting will result in a shortcut. Namely, we can directly show that $\Ga_2$-criterion implies the $L_p$ Poincar\'e inequalities with constants $C\sqrt{p}$. Let $T_t=e^{-tL}$ be a symmetric classical diffusion semigroup with infinitesimal generator $L$ acting on a probability space $(\rz^d, \mu)$. By the well known diffusion theory (see e.g. \cite{RY99}*{Section VII.2}), under certain regularity conditions on $L$, one can always construct a diffusion process $X_t$ corresponding to $T_t$ by solving a martingale problem.

\begin{theorem}\label{rbes}
 Assume that $X_t$ is a classical diffusion process defined on $(\Om, \pz)$ for the semigroup $T_t=e^{-tL}$. Suppose $\Ga_2(f,f)\ge \al \Ga(f,f)$ for the real-valued function $f$. Then for $2\le p<\8$,
\[
\|f-E_{\Fix}f \|_p\le C\sqrt{p}\|\Ga(f,f)^{1/2}\|_p.
\]
\end{theorem}
\begin{proof}
The proof follows the same strategy as \eqref{pcr2}. Assume $2<p<\8$ and $\lim_{t\to \8}T_tf=0$ in $L_p$. By approximation, we may assume that $f$ is compactly supported. Then we obtain a martingale
$$M_t^f=f(X_t)-f(X_0)+\int_0^t Lf(X_s)ds
$$
adapted to $\F_t:=\si(X_s: s\le t)$. Fix a large constant $K>0$. For $0\le t\le K$, define $\td X_t =X_{K-t}$, $\td\F_{[t}=\F_{K-t}$, $N_t^f=(T_t f)(X_{K-t})$. Note that in this setting, the reversed Markov dilation is given by $\td\pi_t f= f(\td X_t)$. Then $(N_t^f)_{0\le t\le K}$ is a reversed martingale with respect to the filtration $\td\F_{[t}$. Indeed, by the Markov property, for $s<t$
\begin{align*}
  E[ N_s^f|\td\F_{[t}] &=E[(T_sf) (X_{K-s})|\F_{K-t}]\\%=E_{X_{K-t}}[(T_sf)(X_{t-s})]\\
  &=T_{(K-s)-(K-t)}T_sf(X_{K-t})=T_tf(X_{K-t})=N_t^f.
\end{align*}
By Lemma \ref{brac} with $\td\pi_r f= f(X_{K-r})$, we have
\[
\lge N^f,N^f\rge_K -\lge N^f,N^f\rge_0=2\int_0^K \Ga(T_rf,T_rf)(X_{K-r})dr.
\]
Applying the BDG inequality \eqref{sqp} (see \cite{BY82}*{Proposition 4.2}) to $N_t^f$ on $[0,K]$, we have
\[
\|T_Kf(X_0)-f(X_K)\|_p\le C\sqrt{p} \|\lge N^f,N^f\rge_K -\lge N^f,N^f\rge_0\|_{p/2}^{1/2}.
\]
Here we used the continuity of the sample paths of the diffusion process $X_t$ so that the conditional square function and the unconditional square functions coincide in continuous time. Since $\mu$ is the invariant measure, we have for any $0\le t\le K$,
\[
\|f(X_t)\|_p^p=\int |f(X_t)|^p d\pz=\int \ez_x|f(X_t)|^p \mu(dx)=\int T_t|f|^p(x)\mu(dx)=\int |f|^pd\mu \pl.
\]
It follows that $\|(T_Kf)(X_0)\|_p=\|T_K f\|_p \to 0$ as $K\to \8$. By the triangle inequality,
\[
\|f\|_p\lel \|f(X_K)\|_p\le \|f(X_K)-T_Kf(X_0)\|_p+\|T_K f(X_0)\|_p.
\]
The rest of the proof is the same as that of \eqref{pcr2}.
\end{proof}

Our first example is very simple, but it clarifies that $\Ga_2$-criterion no longer implies LSI in the general non-diffusion setting. Let us recall the following generalized Schwartz inequality, which is called Choi's inequality; see \cite{Ch74}*{Corollary 2.8}.
\begin{lemma}\label{ccpp}
 Let $\phi: \M\to \N$ be a contractive completely positive map between von Neumann algebras. Then $[\phi(x_i^*x_j)]\ge [\phi(x_i^*)\phi(x_j)]$ for any $n\in\nz$ and any $x_1, \cdots ,x_n\in\M$.
\end{lemma}
\begin{proof}
  Since $\phi$ is complete positive, $\phi\otimes I_n$ is 2-positive, here $I_n$ is the identity matrix. Let
  \[
  X=\left(\begin{array}{cccc}
  x_1 &x_2&\cdots &x_n\\
  0& 0&\cdots&0\\
  \vdots& &\ddots& \\
  0&0&\cdots &0
  \end{array}\right).
  \]
  By \cite{Pau}*{Exercise 3.4}, $[\phi\otimes I_n (X)]^*[\phi\otimes I_n (X)]\le \phi\otimes I_n (X^*X)$. The proof is complete.
\end{proof}

\begin{exam}[Conditional expectation]\label{cond}\rm
Let $E:\M\to\N$ be the conditional expectation and $A=I-E$. For $x,y\in\M$, a calculation gives
 \[
  2\Ga(x,y)\lel x^* y-E(x^*)y-x^*E(y)+E(x^*y)\pl.
 \]
By Lemma \ref{ccpp}, $2[\Ga(x_i,x_j)]\ge [(x_i-Ex_i)^*(x_j-Ex_j)]\ge 0$ for $x_1, \cdots ,x_n\in\M$. We deduce from Theorem \ref{cps} that $A$ generates a completely positive semigroup $T_t=e^{-tA}$ acting on $\M$. It is easy to check $T_t$ is a standard nc-diffusion semigroup. Let $\Ga,\Ga_2$ be the gradient forms associated to $T_t$.
\begin{prop}
 $\Ga_2\ge \frac12\Ga$ in $\M$.
\end{prop}
\begin{proof}
Note that $AE=EA=0$. We find
\[
 4\Ga_2(x,y)\lel x^*y-E(x^*)y-x^*E(y)-2E(x^*)E(y)+3E(x^*y)\pl.
\]
Hence $(\Ga_2-\frac12\Ga)(x,y)=\frac12(E(x^*y)-E(x^*)E(y))$. Since $E$ is contractive completely positive, it follows from Lemma \ref{ccpp} that $\Ga_2-\frac12\Ga$ is a positive form.
\end{proof}
The logarithmic Sobolev inequality fails, however. Indeed, the LSI reads as follows in this case: for $x\ge0$,
\[
 \tau(x^2\ln (|x|/\|x\|_2))\kl C\Big(\|x\|_2^2-\frac12[\tau(E(x^*)x)+\tau(x^*E(x))]\Big)\pl.
\]
It is easy to see this is not true. Indeed, let us consider the Lebesgue probability space $([0,1],dt)$ with $E$ being the expectation, i.e. $E(x)=\int_{[0,1]}x(t)dt=\tau(x)$. Set $x_n(t)=\sqrt{n}1_{[0,1/n]}(t)$. Then $\|x_n\|_2=1$ and the left-hand side is $\tau(x_n^2\ln x_n)=\frac12\ln n$. However, the right-hand side is less than a constant $C>0$, which is impossible for large $n$.
\end{exam}
\subsection{Deviation and transportation inequalities}
As is well-known, the Poincar\'e inequality with constant $\sqrt{p}$ implies the concentration phenomenon. We are going to prove a noncommutative version of exponential integrability due to Bobkov and G\"otze \cite{BG} in the commutative case. The following variant was due to Efraim and Lust-Piquard in the case of Walsh system.
\begin{cor}\label{twos}
 Under the assumptions of Theorem \ref{poin}, we have
 \begin{equation}\label{con1}
  \tau(e^{|x-E_{\rm Fix} x|})\kl 2\exp\Big(\frac{C}\al\max\{\|\Ga(x,x)\|_\8, \|\Ga(x^*,x^*)\|_\8\}\Big)\pl,
 \end{equation}
and for $t>0$
\begin{equation}\label{con2}
 \Prob(|x-E_{\rm Fix} x|\ge t) \kl 2 \exp\Big(-\frac{\al t^2}{4C\max\{\|\Ga(x,x)\|_\8,\|\Ga(x^*,x^*)\|_\8\}}\Big)\pl.
\end{equation}
We may take $C=32e$ in general and $C=8e$ for $x$ self-adjoint.
\end{cor}
\begin{proof}
 We follow the proof in the commutative case; see \cite{ELP}*{Corollary 4.1 and 4.2}. Since $\Ga(x,x)=\Ga(x-E_{\rm Fix} x, x- E_{\rm Fix} x)$, we may assume $E_{\rm Fix}(x)=0$. Put $M=\max\{\|\Ga(x,x)^{1/2}\|_\8, \Ga(x^*,x^*)^{1/2}\|_\8\}$. Note that $\frac{k^k}{(2k-1)!!}\le \big(\frac{e}2 \big)^k$ for all $k\in\nz$. By functional calculus and \eqref{pcr1},
 \begin{align*}
    \frac12 \tau(e^{|x|})&\kl \tau(\cosh x) = 1+\sum_{k=1}^\8 \frac1{(2k)!}\|x\|_{2k}^{2k}\\
    & \kl 1+\sum_{k=1}^\8 \frac{C^{2k} (2k)^k}{\al^k(2k)!}M^{2k} \kl 1+\sum_{k=1}^\8 \frac{k^k (CM)^{2k}}{\al^k k!(2k-1)!!}\\
    &\kl 1+\sum_{k=1}^\8 \frac{(e/2)^k (CM)^{2k}}{\al^k k!} \lel \exp\Big(\frac{eC^2M^2}{2\al}\Big)\pl.
    \end{align*}
 We have proved the first assertion for $C=32e$ and we can take $C=8e$ if $x$ is self-adjoint. For the second inequality, we deduce from Chebyshev inequality that
 \[\tau(1_{[t,\8)}(|x|)) \kl e^{-\la t} \tau(e^{\la |x|})\kl 2 e^{-\la t + C\la^2M^2/\al}\pl.\]
 Then the assertion follows from minimizing the right hand side with respect to $\la$.
\end{proof}
The improvement in the situation of commutative diffusion in Theorem \ref{rbes} also gives an intermediate term in \eqref{con1} for self-adjoint element $x$, i.e.
\begin{align*}
 \tau(e^{|x-E_{\rm Fix} x|})&\kl 2\tau\exp\Big(\frac{C'}\al\Ga(x,x)\Big)
\kl 2\exp\Big(\frac{C'}\al\|\Ga(x,x)\|_\8\Big).
\end{align*}
We do not have such an intermediate term in the fully noncommutative generality without the help of \eqref{bbur}.  However, it seems the Lipschitz norm is the right choice in application to the concentration inequality. In this sense, we did not lose much even if we use a larger norm. The following result is simply a one side version of Corollary \ref{twos}. We record it here for future references.
\begin{prop}
 Under the hypotheses of Theorem \ref{poin}, assume further that $x$ is self-adjoint. Then for $t\in\rz$
 \begin{equation}\label{exin2}
   \tau(e^{t(x-E_{\rm Fix} x)})\kl e^{c \|\Ga(x,x)\|_\8 t^2}\pl,
  \end{equation}
 and for $t>0$,
  \begin{equation}\label{devi}
 \Prob(x-E_{\rm Fix} x \ge t) \kl  \exp\Big(-\frac{t^2}{4c\|\Ga(x,x)\|_\8}\Big)
\end{equation}
where the constant $c$ only depends on $\al$.
\end{prop}

\begin{proof}
Again it suffices to consider \eqref{exin2} for $x$ with $E_{\rm Fix}(x)=0$ since $\Ga(x,x)=\Ga(x-E_{\rm Fix} x, x-E_{\rm Fix} x)$. From the proof of \eqref{con1}, we know there exists $C>0$ such that for $t\in \rz$
 \[ \tau(e^{tx})\kl \tau(e^{tx})+\tau(e^{-tx})\kl2 e^{C \|\Ga(x,x)\|_\8 t^2/\al}\pl.\]
 Then for $t^2\|\Ga(x,x)\|_\8\ge 1$, we have $\tau(e^{tx})\le e^{(\ln 2+C/\al)\|\Ga(x,x)\|_\8t^2}$. For $t^2\|\Ga(x,x)\|_\8< 1$,
 \begin{align*}
   \tau(e^{tx})& \lel 1+\sum_{k=2}^\8\frac{t^k\tau(x^k)}{k!}\kl 1+\sum_{k=2}^\8\frac{t^kC^k k^{k/2}\|\Ga(x,x)\|_\8^{k/2}}{\al^{k/2}k!}\\
   &\kl 1+ c \|\Ga(x,x)\|_\8 t^2\kl e^{c\|\Ga(x,x)\|_\8t^2}
  \end{align*}
for some constant $c=c(\al)$ since $\tau(x)=0$ and the series $\sum_{k=2}^\8 \frac{k^{k/2}}{k!}$ converges. The second assertion follows in the same way as \eqref{con2}.
\end{proof}
The exponential integrability result \eqref{exin2} was proved by Bobkov and G\"otze \cite{BG} in the commutative case by using a variant of LSI. They also deduced a transportation inequality from \eqref{exin2}. We will follow their approach to obtain a noncommutative version of transportation inequality. Since LSI is not available in our noncommutative theory, our Poincar\'e inequalities might be a more universal approach to the transportation inequality. Let us first define Wasserstein distance and entropy in the noncommutative setting.
\begin{defi}
 Let $\rho$ and $\si$ be positive $\tau$-measurable operators (e.g. density matrices) affiliated with $(\M,\tau)$. The noncommutative entropy of $\rho\in L_1(\M,\tau)$ is given by
 \[\Ent(\rho)\lel \tau(\rho\ln  (\rho/\tau(\rho))\pl.\]
 Let $\phi$ and $\psi$ be states on $\M$. The $L_1$-Wasserstein distance between $\phi$ and $\psi$ is defined by
 \[W_1^A(\phi,\psi) \lel \sup\{ |\phi(x)-\psi(x)|: x \text{ self-adjoint },\|\Ga(x,x)\|_\8\le 1\}\pl.\]
 The $L_1$-Wasserstein distance between $\rho$ and $\si$ is $W_1^A(\rho,\si)=W_1^A(\phi_\rho,\phi_\si)$ for $\phi_\rho(\cdot)=\tau(\cdot\rho)/\tau(\rho)$ and $\phi_\si(\cdot)=\tau(\cdot\si)/\tau(\si)$.
\end{defi}
Here the superscript $A$ in $W_1^A$ is to emphasize the dependence on the generator of the semigroup $T_t$. We may ignore the superscript $A$ for simplicity in the following. It is easy to check that $W_1^A$ is a pseudometric but may not be a metric in general. Our definition of Wasserstein distance coincides with the classical definition in the commutative case due to the Kantorovich--Rubinstein theorem; see e.g. \cite{Vil}*{Theorem 5.10}. It is also closely related to the quantum metric in the sense of Rieffel \cite{Rie}. Now we state a general fact on the relationship between conditional expectation and entropy.
\begin{lemma}\label{exen}
 Let $\rho\in L_1(\M,\tau)$ with $\rho\ge 0$ and $\tau(\rho)=1$ and $E: \M\to\N$ the conditional expectation onto subalgebra $\N$. Then
 \[\tau(E\rho \ln  E\rho)\kl \tau(\rho\ln \rho)\pl.\]
\end{lemma}
\begin{proof}
Let $\rho_n=\rho 1_{[0,n]}(\rho)$. Then $\rho_n\in L_p(\M,\tau)$ for all $p\ge 1$. It is easy to see that $\rho_n\to\rho$ in the measure topology. Notice that $\tau[(\rho_n/\tau(\rho_n))\ln (\rho_n/\tau(\rho_n))] = \lim_{p\downarrow 1}\frac{\|\rho_n/\tau(\rho_n)\|_p^p -1}{p-1}$. This yields
 \begin{align*}
  &\tau\Big(\frac{E\rho_n}{\tau(\rho_n)} \ln  \frac{E\rho_n}{\tau(\rho_n)}\Big) \lel \lim_{p\downarrow 1}\frac{\|E\rho_n/\tau(\rho_n)\|^p_p -1}{p-1}\\
 \kl &\lim_{p\downarrow 1}\frac{\|\rho_n/\tau(\rho_n)\|^p_p -1}{p-1} \lel \tau\Big(\frac{\rho_n}{\tau(\rho_n)}\ln \frac{\rho_n}{\tau(\rho_n)}\Big)\pl.
 \end{align*}
Let $\mu$ be the distribution of $\rho$. Then
\[ \tau\Big(\frac{\rho_n}{\tau(\rho_n)}\ln \frac{\rho_n}{\tau(\rho_n)}\Big)\lel \frac1{\tau(\rho_n)}\int_0^n x\ln  x \mu (dx)-\ln \tau(\rho_n)\to \tau(\rho\ln \rho).\]
Following \cite{FK}, we denote the generalized singular number of $\rho$ by $\mu_t(\rho)$. Note that $\|E\rho_n-E\rho\|_1\le\|\rho_n-\rho\|_1\to0$ as $n\to \8$. We have for every $t>0$,
\[
 \mu_t(E\rho_n-E\rho)\kl t^{-1}\int_0^t \mu_s(E\rho_n- E\rho) ds\kl t^{-1}\|E\rho_n-E\rho\|_1\pl.
\]
Then $\lim_{n\to\8}\mu_t(E\rho_n-E\rho)=0$. By \cite{FK}*{Lemma 3.1}, $(E\rho_n)$ converges to $E\rho$ in the measure topology. Since $0\le E\rho_n\le E\rho$, by \cite{FK}*{Lemma 2.5}, $\mu_t(E\rho_n)\le \mu_t(E\rho)$. We deduce from \cite{FK}*{Lemma 3.4} that $\lim_{n\to\8}\mu_t(E\rho_n)=\mu_t(E\rho)$ for $t>0$. Now consider $g(x)=x\ln  x = x\ln  x 1_{[1,\8)(x)}- (-x\ln  x 1_{(0,1)})$. Both functions in the decomposition are nonnegative Borel functions vanishing at the origin. It follows from \cite{FK}*{(3)} that $\tau(E\rho\ln  E\rho)=\int_0^1 \mu_t(E\rho)\ln  \mu_t(E\rho) dt$. Since for fixed $\eps>0$, $[\mu_t(E\rho_n)\ln  \mu_t(E\rho_n)]_n$ is uniformly bounded on $t\in[\eps, 1]$, we have
\begin{align*}
  &\tau(E\rho\ln  E\rho)\lel \sup_{\eps>0} \int_{\eps}^1 \mu_t(E\rho)\ln \mu_t(E\rho)dt\\
 \lel &\sup_{\eps>0}\lim_{n\to\8}\int_{\eps}^1\mu_t(E\rho_n)\ln \mu_t(E\rho_n)dt\kl \limsup_{n\to\8}\tau(E\rho_n\ln  E\rho_n)\\
 \kl &\limsup_{n\to\8}\tau(\rho_n\ln  \rho_n)\pl.
\end{align*}
This completes the proof.
 \end{proof}

The next result in the commutative setting is well known; see e.g. \cite{DZ}*{Section 6.2}.
\begin{lemma}
Let $\si$ be a self-adjoint $\tau$-measurable operator. Then,
\begin{equation}\label{endu}
 \ln  \tau(e^{\si})\lel \sup\{\tau(\rho\si)-\tau(\rho\ln \rho):\rho \ge 0, \tau(\rho)=1\}\pl.
 \end{equation}
 Therefore, for all positive $\rho\in L_1(\M,\tau)$
 \begin{equation}\label{ent}
  \Ent(\rho)\lel \sup\{\tau(\si\rho): \si \text{ self-adjoint}, \tau(e^\si)\le 1\}\pl.
  \end{equation}
 \end{lemma}
\begin{proof}
Let $\si$ be a self-adjoint operator $\tau$-measurable operator. Consider the von Neumann subalgebra $\N$ generated by $\{f(\si):f: \cz\to\cz \text{ bounded measurable}\}$. Then there exists a conditional expectation $E:\M\to\N$ which can extend to a contraction $L_p(\M,\tau)\to L_p(\N,\tau)$ for all $1\le p\le \8$. Assume $\tau(\rho)=1$. Then $E(\rho)\in L_1(\N,\tau)$. But $\N$ is commutative and $\tau(E(\rho))=\tau(\rho)=1$. After identifying $\tau$ with a probability measure denoted still by $\tau$, we use Jensen's inequality for the measure $E(\rho) d\tau$ to deduce that
 \[\tau(\si E(\rho))-\tau(E(\rho)\ln  E(\rho))\lel \tau(\ln (e^\si E(\rho)^{-1}) E(\rho)))\kl \ln \tau(e^\si)\pl.\]
 Using Lemma \ref{exen} and noticing that $\tau(\si\rho)=\tau(\si E(\rho))$, we find
 \[\tau(\si\rho)- \tau(\rho\ln  \rho)\kl \ln \tau(e^\si)\pl.\]
For the reverse inequality, put $\si_n=\si 1_{(-\8,n]}(\si)$ for $n\in \nz$ where $1_{(-\8,n]}(\si)$ is the spectral projection of $\si$. Plugging $\rho_n=e^{\si_n}/\tau(e^{\si_n})$ into the right hand side of \eqref{endu}, we have
 \[\tau(\si\rho_n)- \tau(\rho_n\ln  \rho_n)\lel\frac{\tau((\si-\si_n)e^{\si_n})}{\tau(e^{\si_n})}+ \ln \tau(e^{\si_n}) \pl.\]
 By the spectral decomposition theorem of $\si$, $\tau((\si-\si_n)e^{\si_n})\ge 0$. Then for all $n$ we have
 \[\sup\{\tau(\rho\si)-\tau(\rho\ln \rho):\rho \ge 0, \tau(\rho)=1\} \gl \ln \tau(e^{\si_n})\pl. \]
 By Fatou's lemma \cite{FK}*{Theorem 3.5} $\liminf_{n\to\8}\ln \tau(e^{\si_n})\gl \ln \tau(e^\si)$. This proves \eqref{endu}.
   For \eqref{ent}, note that, by \eqref{endu}, $\tau(e^\si)\le 1$ implies $\tau(\si\rho)\le \tau(\rho\ln  \rho)$ for all positive $\rho\in L_1(\M,\tau)$ with $\tau(\rho)=1$.  If $\tau(\rho)\neq 1$, we consider $\rho'=\rho/\tau(\rho)$ and find $\tau(\si\rho)\le\tau(\rho\ln  (\rho/\tau(\rho)))$. The equality is achieved by $\si=\ln \rho-\ln \tau(\rho)$. This proves the second assertion.
 \end{proof}

\begin{theorem} \label{tran}
Let $(\M,\tau)$ be a noncommutative probability space. Then
 \begin{equation}\label{wsen}
 W_1(\rho,1)\kl \sqrt{2c \Ent(\rho)},
 \end{equation}
 for all $\rho\ge0$ with $\tau(\rho)=1$ if and only if for every self-adjoint $\tau$-measurable operator $x$ affiliated with $\M$ such that $\|\Ga(x,x)\|_\8\le1$ and $\tau(x)=0$,
 \begin{equation}\label{exin}
   \tau(e^{tx})\kl e^{ct^2/2}, \text{ for all } t\in \mathbb{R}_+\pl.
   \end{equation}
\end{theorem}
\begin{proof}
  Thanks to the preceding two lemmas, the proof is the same as that in the commutative case in \cite{BG}. We provide it here for completeness. Setting $\si=tx-ct^2/2$ in \eqref{ent} and assuming \eqref{exin} we find
 \begin{equation}\label{bobk}
  \tau((tx-ct^2/2)\rho)\kl \Ent(\rho)\pl.
  \end{equation}
 Since $\tau(x)=0$ and $\tau(\rho)=1$, it follows that
 $\tau(x\rho-x)\le \frac{ct}2+\frac1t \Ent(\rho).$
 Minimizing right hand side gives
 \begin{equation}\label{etpy}
   \tau(x\rho-x)\le\sqrt{2c\Ent(\rho)}\pl.
  \end{equation}
 Note that $\tau(x\rho-x)=\tau(\mathring{x}\rho-\mathring{x})$ for all $x$ where $x=\mathring{x}+\tau(x)$. Taking $\sup$ over all self-adjoint $x$ with $\|\Ga(x,x)\|_\8\le 1$ on the left hand side of \eqref{etpy} gives \eqref{wsen}. For the other direction, note that \eqref{bobk} is equivalent to \eqref{wsen} by reversing the above argument. Then \eqref{exin} follows from \eqref{endu} by setting $\si=tx-ct^2/2$.
\end{proof}
If ${\rm Fix} = \cz 1$ (i.e. the system $(\M,T_t)$ is ergodic), then combining the above theorem with \eqref{exin2}, we find the transportation inequality \eqref{wsen} under the assumptions of Theorem \ref{poin}. In fact, we even have a non-ergodic version of transportation inequality.
\begin{cor}\label{tran1}
Suppose $\tau(e^{t(x-E_{\rm Fix} x)})\le e^{ct^2}$ for any $\tau$-measurable self-adjoint operator $x$ affiliated to $\M$ such that $\|\Ga(x,x)\|_\8\le 1$. Then
\begin{equation}\label{wse2}
 W_1(\rho, E_{\rm Fix} \rho)\kl \sqrt{2c \Ent(\rho)}\pl.
\end{equation}
for all $\rho\ge0$ with $\tau(\rho)=1$. In particular, \eqref{wsen} holds under the additional assumption $E_{\rm Fix} \rho =1$.
\end{cor}
\begin{proof}
 The proof modifies a little that of Theorem \ref{tran}. Since $\tau(\rho)=1$, we have $\tau([t(x-E_{\rm Fix} x)-ct^2/2]\rho)\le \Ent(\rho)$. Then we deduce that $\tau(\rho x -\rho E_{\rm Fix}(x))\le \sqrt{2c\Ent(\rho)}$. Since $\tau(\rho E_{\rm Fix}(x))=\tau(E_{\rm Fix}(\rho)x)$, we have
 \[
  \tau(\rho x-E_{\rm Fix}(\rho)x)\kl\sqrt{2c\Ent(\rho)}\pl.
 \]
Taking $\sup$ over all self-adjoint $x$ with $\|\Ga(x,x)\|_\8\le 1$ gives the assertion.
\end{proof}
It is easy to see that the assumptions are fulfilled by the hypotheses of Theorem \ref{poin}. The point here is that even though the fixed point algebra ${\rm Fix}$ is not trivial we still have a transportation inequality although in certain situation the inequality does fail.
\re
Let $\rho$ be a positive operator with $\tau(\rho)=1$. For $\rho\in{\rm Fix}$, define $B(\rho)=\{f\in L_1(\N):E_{\Fix}(f)=\rho\}$. Then for $f_1,f_2\in B(\rho)$, we have $W_1(f_1,f_2)\le W_1(f_1,\rho)+ W_1(f_2,\rho)<\8$. However, if $f_1\in B(\rho_1), f_2\in B(\rho_2)$ and $\rho_1\neq \rho_2$, then
\[
 W_1(\rho_1,\rho_2)\ge \sup \{|\tau(\rho_1 x-\rho_2 x)|:{x\in {\rm Fix}, \|\Ga(x,x)\|_\8\le 1}\}\lel \8\pl.
\]
It follows that $W_1(f_1,f_2)\ge |W_1(\rho_1,\rho_2)-W_1(f_1,\rho_1)-W_1(f_2,\rho_2)|=\8$. This yields an interesting geometric picture: operators in the same ``fiber'' $B(\rho)$ have finite distance between one another while operators belonging to different ``fibers'' have infinite distance.
\mar

The following simple result provides another way (under the assumption of finite diameter) to obtain the transportation inequality.
\begin{cor}\label{tran2}
Suppose for self-adjoint $x\in \N$, $E_{\rm Fix}(x)=0$ and $\|\Ga(x,x)\|_\8\le 1$ imply $\|x\|_\8\le K$. Then, \eqref{wse2} holds with $c=K^2$ for all $\rho\ge0$ such that $\tau(\rho)=1$.
\end{cor}
\begin{proof}
 A calculation gives $e^x-x\le e^{x^2}$. Assume $E_{\rm Fix}(x)=0$ and $\|\Ga(x,x)\|_\8\le1$. Then for $t>0$, $\tau(e^{tx})=\tau(e^{tx}-tx)\le \tau(e^{t K}-t K)\le e^{K^2t^2}$. The claim now follows from Corollary \ref{tran1}.
\end{proof}
%Please be very careful about taking sup over all x such that $E_{Fix} x=0$ !!
Suppose in Theorem \ref{poin} we only have $\Ga_2\ge0$ but not the $\Ga_2$-condition. Junge and Mei proved in \cite{JMe} as the main result
$$\|A^{1/2} x\|_p\le c(p)\max\{\|\Ga(x,x)^{1/2}\|_p,\|\Ga(x^*,x^*)^{1/2}\|_p\}$$
in this setting. Using the proof of {Theorem 1.1.7} in the same paper  \cite{JMe}, it can be shown that if
\begin{equation}\label{ctra}
 \|T_t: L_1^0(\N)\to L_\8(\N)\|\le Ct^{-n/2},
 \end{equation}
then $\|A^{-1/2}:L_p^0(\N)\to L_\8(\N)\|\le C(n)$ for $p>n$. Indeed, we consider the composition of operators
\[
L_p^0(\N)\xrightarrow{ A^{-\al}} L_q^0(\N)\hookrightarrow L_{s,1}^0(\N)\xrightarrow{ A^{-\bt}} L_\8(\N),
\]
where $s,q,p,\al,\bt$ are chosen so that
$$ 1<s<q, n<p<q<\8, \al+\bt=\frac12, \al=\frac{n}2(\frac1p-\frac1q), \bt=\frac{n}{2s}.$$
For example, $p=2n,q=4n,s=\frac{4n}3$ satisfy these conditions. By \cite{JMe}*{Corollary 1.1.4}, $\|A^{-\al}:L_p^0(\N)\to L_q^0(\N)\|\le C(n)$. Using \cite{JMe}*{Lemma 1.1.3}, we have $\|A^{-\bt}: L_{s,1}^0(\N)\to L_\8(\N) \|\le C(n)$. The embedding $L_q^0(\N)\hookrightarrow L_{s,1}^0(\N)$ follows from interpolation theory. Hence, $\|A^{-1/2}:L_p^0(\N)\to L_\8(\N)\|\le C(n)$. This gives $\|x\|_\8\le C(n)\|A^{1/2}x\|_p$ for large $p$ and $E_{\rm Fix}(x)=0$. Assuming $\Ga(x,x)\le 1$ for self-adjoint $x$, it follows that $\|x\|_\8\le C(n,p)$. By choosing e.g. $p=2n$, the constant $C(n,p)$ actually only depends on $n$. In light of Corollary \ref{tran2}, we obtain the following result.
\begin{cor}\label{tran3}
 Let $T_t$ be a standard nc-diffusion semigroup acting on $\N$ with $\Ga_2\ge0$. Then \eqref{ctra} with finite dimension $n$ implies the transportation inequality \eqref{wse2} for all $\rho\ge0$ such that $\tau(\rho)=1$.
\end{cor}

In the commutative theory, the transportation inequality \eqref{wsen} implies isoperimetric type inequality by Marton's argument in \cite{BG}. So far it is not clear what isoperimetric inequality means in noncommutative probability. We hope to give a noncommutative analog of isoperimetric inequality.
\begin{defi}
 Let $e,f\in(\M,\tau)$ be projections. The distance between $e$ and $f$ is
 \[d(e,f)\lel \inf\{W_1(\phi,\psi): \phi \text{ and } \psi \text{ are states, }s(\phi)=e,s(\psi)=f\},\]
 where $s(\phi)$ is the support of $\phi$.
\end{defi}
Here our definition generalizes directly the distance of sets in the commutative theory. Thus in general $d$ is not a metric, as in the commutative setting. Then the following result follows from the same proof as in the commutative setting given in \cite{BG}.
\begin{prop}
 Let $e,f\in(\M,\tau)$ be projections. Then under the assumptions of Theorem \ref{poin} and assuming ${\rm Fix}=\cz 1$,
 \[d(e,f)\kl \sqrt{-2c\ln \tau(e)}+\sqrt{-2c\ln  \tau(f)}\pl.\]
 Equivalently, for every $h>\sqrt{-2c\ln \tau(e)}$ and every projection $p$ such that $d(p,e)>h$,
 \[
 \tau(p)\kl \exp\Big(-\frac1{2c}\Big(h-\sqrt{-2c\ln \tau(e)}\Big)^2\Big)\pl.
 \]
\end{prop}
\begin{proof}
 Put $\phi_e(\cdot)=\tau(e\cdot)/\tau(e)$ and $\phi_f=\tau(f\cdot)/\tau(f)$. It is easy to see that $d(e,f)\le W_1(\phi_e,\phi_f)$. Then triangle inequality and \eqref{wsen} yield $d(e,f)\le \sqrt{2c\Ent(e)}+\sqrt{2c\Ent(f)}$. By spectral decomposition theorem of the identity, $\Ent(e)= \int_0^1\frac{1_{A}}{\tau(e)}\ln \frac{1_{A}}{\tau(e)} d\mu$ where $A$ is a Borel set such that $1_A(Id)=e$. Hence we find $\Ent(e)=-\ln \tau(e)$, which gives the first assertion. The equivalent formulation is a simple calculation.
\end{proof}

To conclude this section, we remark that the best possible $\al$ in $\Ga_2\ge\al\Ga$ sometimes characterizes the dynamical system $(\M,T_t)$; see the example of hyperfinite $II_1$ factor below.

\section{Application to the group von Neumann algebras}
Starting from this section, we will investigate a variety of examples which satisfy the assumptions of Theorem \ref{poin}. All the Poincar\'e type, deviation and transportation inequalities we derived previously will hold in these examples. The key point is to check $\Ga_2\ge\al\Ga$. It may be of independent interest because it means a strictly positive Ricci curvature from the geometric point of view. We consider the $\Ga_2$-criterion for group von Neumann algebras in this section.

Let $G$ be a countable discrete group. In this paper we say that $\psi:\psi\to \rz_+$ is a conditionally negative definite length (cn-length) function if it vanishes at the identity $e$, $\psi(g)=\psi(g^{-1})$ and is conditionally negative which means that $\sum_{g}\xi_g=0$ implies $\sum_{g,h}\bar\xi_g\xi_h\psi(g^{-1}h)\le 0$.

Let $\la: G\to B(\ell_2(G))$ be the left regular representation given by $\la(g)\de_h=\de_{gh}$ where $\de_g$'s form the unit vector basis of $\ell_2(G)$. Let $\Lc(G) = \la(G)''$, the von Neumann algebra generated by $\{\la(g): g\in G\}$. Any $f\in\Lc(G)$ can be written as
\[f\lel \sum_{g\in G}\hat f(g)\la(g)\pl.\]
It is well known that $\tau(f)=\lge\de_e, f \de_e \rge$ defines a faithful normal tracial state and $\tau(f)=\hat f(e)$ where $e$ is the identity element. In what follows, we define the semigroup associated to the cn-length function $\psi$ by $T_t(\la(g))=T_t^\psi(\la(g))=\phi_t(g)\la (g)$ for $g\in G$, where $\phi_t(g)=e^{-t\psi(g)}$. The infinitesimal generator of $T_t$ is given by $A\la(g)=\psi(g)\la(g)$. Recall that the Gromov form is defined as
\[ K(g,h)\lel K_{g,h} \lel\frac12(\psi(g)+\psi(h)-\psi(g^{-1}h)), \quad g, h\in G\pl.\]
Given $f=\sum_{x\in G}\hat f(x)\la(x)$ and $g=\sum_{y\in G}\hat g(y)\la(y)$, a straightforward calculation gives
 \[ \Ga(f,g)\lel \sum_{x,y\in G}\bar{\hat f}(x)\hat g(y)K(x,y)\la(x^{-1}y)\pl,\]
 \[ \Ga_2(f,g)\lel \sum_{x,y\in G}\bar{\hat f}(x)\hat g(y)K(x,y)^2\la(x^{-1}y)\pl.\]

Let $\A$ be the subalgebra of $\Lc(G)$ which consists of elements that can be written as finite combination of $\la_g, g\in G$. Then $\A$ is weakly dense in $\Lc(G)$ such that $A\A\subset \A$ and $T_t\A\subset \A$.
\begin{lemma}\label{gpvn}
 $\A$ is dense in $\Dom(A^{1/2})$ in the graph norm of $A^{1/2}$ and $T_t$ is a standard nc-diffusion semigroup acting on $\Lc(G)$.
\end{lemma}
\begin{proof}
For $f=\sum_{g_i\in G}\hat f(g_i)\la(g_i)\in \Dom(A^{1/2})$, put $f_n=\sum_{i=1}^n \hat f(g_i)\la(g_i)$. Note that $f\in L_2(\Lc(G))$. Then
\[
 \|f_n-f\|_{L_2(\Lc(G))}\lel \tau((f_n^*-f^*)(f_n-f)) \lel \sum_{i=n+1}^\8 |\hat f(g_i)|^2\to 0, \text{ as } n\to\8\pl.
\]
Since $\lge Af,f\rge =\sum_{i=1}^\8\psi(g_i)|\hat f(g_i)|^2 <\8$, we have
\[
 \lge A(f_n-f), f_n-f\rge_{L_2(\Lc(G),\tau)} \lel \sum_{i=n+1}^\8 \psi(g_i) |\hat f(g_i)|^2\to 0, \text{ as } n\to\8\pl.
\]
Therefore $\A$ is dense in the graph norm. Since $\psi$ is conditionally negative, Schoenberg's theorem implies that $T_t$ is completely positive; see e.g. \cite{BO}*{Appendix D}. It can be directly checked that $T_t$ is normal and unital. Since $\psi(g)=\psi(g^{-1})$,
\[
\tau(T_t(x)y)\lel \lge  \de_e,\sum_{g}e^{-t\psi(g)}\hat x(g)\la_g\sum_h \hat y(h)\la_h\de_e\rge \lel \sum_{g} e^{-t\psi(g)}\hat x(g)\hat y(g^{-1}) \lel \tau(xT_ty)\pl.
\]
Hence $T_t$ is self-adjoint. To check that $(T_t)$ is weak* continuous on $\Lc(G)$, it suffices to verify that $(T_t)$ is a strongly continuous semigroup on $L_2(\Lc(G))$. For $f\in L_2(\Lc(G))$, we have
\[
 \|T_t f-f\|^2_{L_2(\Lc(G))}\lel \sum_{g\in G} (e^{-t\psi(g)}-1)^2|\hat f(g)|^2 \to 0 \mbox{ as } t\to 0\pl.
\]
We have proved that $(T_t)$ is a standard semigroup. Let $f\in\A$. Then
$$\|\Ga(f,f)\|_1 \lel \lge \de_e,\sum_{g\in G}|\hat f(g)|^2K_{g,g} \de_e \rge \lel \sum_{g\in G}\psi(g)|\hat f(g)|^2 \lel \|A^{1/2}f\|_2<\8\pl. $$
Since $\A$ is dense in $\Dom(A^{1/2})$ in the graph norm, by an approximation argument (see the proof of Lemma \ref{ga2a}), the above equality holds for all $f\in\Dom(A^{1/2})$ and thus $(T_t)$ is a nc-diffusion semigroup.
\end{proof}
%\begin{lemma}
 %Suppose $\psi\ge 0$ is integer valued and $\psi(g) = 0$ only if $g$ is the empty word. Then for all $1\le p \le \8$, we have $\lim_{t\to \8} \|T_t f-\tau(f)\|_p = 0$ for all $f\in \Lc(G)$.
%\end{lemma}
%\begin{proof}
 %Let $f\in \Lc(G)$. Note that $$\|T_tf-\hat f(\eps)\|_\8 = \Big\|\sum_{g\neq \eps}e^{-t\psi(g)}\hat f (g)\la(g)\Big\|_\8\kl e^{-t}\sum_{g\neq \eps}\|\hat f (g)\la(g)\|_\8\to 0$$ as $t\to \8$, where $\eps$ denotes the empty word. Then the convergence holds for all $1\le p\le\8$.
%\end{proof}
By virtue of Lemma \ref{gpvn} and Corollary \ref{ga2a}, our Poincar\'e inequalities will follow if the $\Ga_2$-criterion holds. Put ${\rm Fix}=\{f\in\Lc(G):\psi(f)=0\}$.
\begin{cor}\label{pgp}
  Let $2\le p<\8$ and assume $\Ga_2(f,f)\ge\al \Ga(f,f)$ for $f\in\A$. Then there exists a constant $C$ such that for all self-adjoint $f\in L_p(\Lc(G),\tau)$,
\[ \|f-E_{\rm Fix}(f)\|_p\kl C\al^{-1/2}\min\{\sqrt{p}  \|\Ga(f,f)^{1/2}\|_\8,\pl{p} \|\Ga(f,f)^{1/2}\|_p\}\pl.\]
 If $\psi(g) = 0$ only if $g$ is the identity element, then $E_{\rm Fix}(f)=\tau(f)$ for $f\in\Lc(G)$.
\end{cor}
Among the examples we will consider below, the free group on $n$ generators $\fz_n$ satisfies $E_{\rm Fix}(f)=\tau(f)$ but the finite cyclic group $\zz_n$ has nontrivial ${\rm Fix}$. With the help of Lemma \ref{gpvn} and Corollary \ref{ga2a}, we only need to check the $\Ga_2$-criterion on the finitely supported elements in order to fulfill the hypotheses of our main theorem. We call \[ [\Ga_2(f^i, f^j)]\ge\al [\Ga(f^i,f^j)] \mbox{ for any } n\in\nz \mbox{ and } f^1, \cdots ,f^n\in\A \]
the algebraic $\Ga_2$-condition (or $\Ga_2$-criterion) and abbreviate it to ``$\Ga_2\ge \al \Ga$ in $\Lc(G)$''. This is the theme of two sections from now on. This condition is seemingly stronger than needed. However, $\Ga_2(f,f)\ge\al \Ga(f,f)$ for all $f=\sum_{g \text{ finite}} \hat f_g\la(g)\in \Lc(G)$ amounts to check $[\Ga_2(\la(g_i), \la(g_j))]\ge\al [\Ga(\la(g_i), \la(g_j))]$ for $g_i\in G$. This algebraic condition is also easier to check because it can be reduced to check the positivity of certain matrices as will be shown below. The following technical lemmas will be used repeatedly.
\begin{lemma}\label{kform}
 Suppose $K=(K_{g,h})_{g,h\in G}$ is a matrix indexed by $G$ with entries in $\cz$ and define a sesquilinear form $\Theta:\A\times \A\to \A$, $\Theta(f^i,f^j)=\sum_{g,h\in G} \bar{\hat {f^i}}(g) \hat{f^j}(h)K_{g,h} \la(g^{-1}h)$. Then $K$ is nonnegative definite if and only if $\Theta$ is positive.
\end{lemma}
\begin{proof}
 Our proof is based on Lance \cite{Lan}*{Proposition 2.1}. Assume $K$ is nonnegative definite. Write $K=X^*X$ for $X=(x_{g,h}), x_{g,h}\in\cz$. Then $K=\sum_{l} (\sum_{g,h}x_{lg}^* x_{lh}\otimes e_{g,h})$. Given $f^1, \cdots ,f^n\in \A$, we have
 \begin{align*}
   \Theta(f^i,f^j) &\lel \sum_{l\in G}\sum_{g,h\in G} \bar{\hat{f^i}}(g)\hat{f^j}(h)x_{lg}^* x_{lh}\la(g^{-1}h) \\
 &\lel \sum_{l\in G} \Big(\sum_g {\hat{f^i}}(g)x_{lg}\la_g\Big)^*\Big(\sum_h {\hat{f^j}}(h)x_{lh}\la_h\Big)\pl.
 \end{align*}
Here we understand all indices are finite. Put $\tet^i_l = \sum_g {\hat{f^i}}(g)x_{lg}\la_g$. We then have
\begin{align*}
\Theta \lel \sum_{l\in G} \sum_{i,j=1}^n{\tet^i_l}^* \tet^j_l\otimes e_{i,j}
\lel \sum_{l\in G} \Big(\sum_{i=1}^n \tet^i_l\otimes e_{1,i}\Big)^*\Big(\sum_{j=1}^n \tet^j_l\otimes e_{1,j}\Big)\pl,
\end{align*}
which is positive in $M_n(\Lc(G))$. Conversely, let $(x_i)=\sum_i x_i\de_{k_i}\in \ell_2(G)$ and write $g_i= {k_i}^{-1}\in G$. Then $(x_i)=\sum_{i\in \nz} x_i \de_{g_i^{-1}}$. Let $f^i=\la(k_i)$ so that $(\Theta(f^i,f^j))=(K_{i,j}\la(k_i^{-1}k_j))$ is positive in $M_n(\Lc(G))$ for all $n\in \nz$. Then for $h_i=x_i\de_{g_i}\in \ell_2(G)$, we have for all $n\in\nz$
\begin{align*}
  0&\kl \lge [\Theta(f^i,f^j)] (h_1, \cdots ,h_n), (h_1, \cdots ,h_n)\rge_{\ell_2^n(\ell_2(G))}\\
  &\lel \sum_{i,j=1}^n \lge K_{i,j}\la(k_i^{-1}k_j)x_j\de_{g_i}, x_i\de_{g_i}\rge_{\ell_2(G)} \lel \sum_{i,j=1}^n K_{i,j}x_j \bar{x_i} \pl,
\end{align*}
which implies that $K$ is nonnegative definite.
\end{proof}

The next lemma is useful when we deal with the product of groups. Note that
$$\Lc(\prod_{i=1}^m G_i)\cong\bar\otimes_{i=1}^m\Lc(G_i).$$
The identification is given by $\la(g_1, \cdots ,g_m)\mapsto \la(g_1)\otimes \cdots \otimes \la(g_m)$ for $g_i\in G_i$. We associated the form $\Ga_2$ to a matrix $K$ as follows: $\Ga_2^K(f,g)=\sum_{x,y\in G} \bar{\hat f}(x)\hat g(y)K^2_{x,y}\la(x^{-1}y)$. In what follows the matrix $K$ will be the Gromov form. If $K=K_1\otimes K_2$, then it is easy to check $\Ga^K(f^i\otimes g^i, f^j\otimes g^j)=\Ga^{K_1}(f^i,f^j)\otimes \Ga^{K_2}(g^i,g^j)$ for $f_i\in \Lc(G_1), g_i\in \Lc(G_2)$.
\begin{lemma}\label{tens}
 Let $(K_i)_{i=1}^m$ be nonnegative definite matrices and $\Ga^{K_i}$ the associated gradient forms in the sense of Lemma \ref{kform}. Suppose $\Ga^{K_i}_2\ge\al \Ga^{K_i}$. Then
 \[\Ga_2^{K}\ge \al \Ga^{K}, \]
where $K=\sum_{i=1}^m \mathds{1}\otimes  \cdots \otimes K_i\otimes \cdots  \otimes \mathds{1}$ with $K_i$ in the $i$th position and in what follows $\mathds{1}$ always denotes the matrix with every entry equal to 1.
\end{lemma}
\begin{proof}
In light of Lemma \ref{kform} it suffices to verify $K\bullet K\ge \al K$. Here and in the following $A\bullet B$ denotes the Schur product of matrix. Note that trivially $\mathds{1}\ge0$. Since $K_i\ge 0$, all the ``cross terms'' of the form $$\mathds{1}\otimes \cdots \otimes K_{i_1}\otimes \cdots \otimes K_{i_2} \otimes  \cdots  \otimes \mathds{1}$$
are nonnegative matrices for all $1\le i_1 < i_2 \le m$. It follows that
\[ K\bullet K\gl \sum_{i=1}^m \mathds{1}\otimes  \cdots \otimes (K_i\bullet K_i) \otimes \cdots  \otimes \mathds{1} \gl \al K\pl.
 \qedhere\]
\end{proof}

\subsection{The free groups}\label{frgp}
Let $\fz_n$ denote the free group on $n$ generators with length function $\psi=|\cdot|$, where for $g\in \fz_n$, $|g|$ is the length of (the freely reduced form of) $g$.  Note that the Gromov form $K(g,h) = |\min(g,h)|:= \max\{|w|: g=wg',h=wh'\}$ where $\min(g,h)$ is the longest common prefix subword of $g$ and $h$. It is well known that $\psi$ is conditionally negative due to Haagerup \cite{Haa}.

\begin{prop}\label{frcur}
 $ \Ga_2 \ge \Ga$ holds in $\Lc(\fz_n)$ for the semigroup $(e^{-t\psi})$ where $\psi$ is defined as above.
\end{prop}
\begin{proof}
For a freely reduced word $x\in \fz_n$, write $g_i\prec g$ for the prefix subword of $g$ with length $i$. Following Haagerup's construction, we define a map
 \[V: \fz_n \to \ell_2(\fz_n), \quad g \mapsto V(g)=\sum_{g_i\prec g}\sqrt{2(i-1)}\de_{g_i}\pl.\]
 Then we have
 \[ \td{K}_{g,h}:=K^2_{g,h}-K_{g,h} \lel \lge V(g), V(h)\rge_{\ell_2(\fz_n)}\lel V(g)^*V(h)\pl,\]
 where $V(g)^*$ is a row vector and $V(h)$ a column vector. It follows that $\td{K}=(\td{K}_{g,h})_{g,h}$ is a nonnegative definite matrix. We deduce from Lemma \ref{kform} that $\Ga_2\ge \Ga$.
\end{proof}

The particular case $n=1$ gives some interesting results in classical Fourier analysis. Indeed, $\Lc(\fz_1)=\Lc(\zz)=L_\8(\tz)$ and $L_p(\Lc(\fz_1))=L_p(\tz)$ after identifying $\la(k)(x)=e^{2\pi i kx}$. In this case
\[
 K(j,k)=\left\{
\begin{array}{ll}
\min(|j|,|k|) , & jk>0,\\
0, & \text{otherwise.}
         \end{array}\right.\]
\begin{cor}
 Let $2\le p<\8$. Then there exists constants $C$ and $C'$ such that for all $f\in L_p(\tz)$, we have
 \[ \|f-\hat f(0)\|_p\kl C\sqrt{p}\left\|\sum_{j,k\in\zz,jk>0}\bar{\hat f}(j)\hat f(k)\min(|j|,|k|)e^{2\pi i(k-j)\cdot}\right\|^{1/2}_\8,\]
 \[ \|f-\hat f(0)\|_p\kl C'{p}\left\|\sum_{j,k\in\zz,jk>0}\bar{\hat f}(j)\hat f(k)\min(|j|,|k|)e^{2\pi i(k-j)\cdot}\right\|^{1/2}_{p/2}.\]
\end{cor}

\re
Observe that this example is purely commutative. However, commutative probability theory seems insufficient to establish these inequalities. Intuitively, the multiplier $|j|$ corresponds to $\Delta^{1/2}$. The Markov process generated by $\Delta^{1/2}$ is the Cauchy process with discontinuous path. The classical diffusion theory does not apply here.  But it is still nc-diffusion so that our noncommutative theory is essential in this regard. In general, whenever the process has discontinuous path but its semigroup still satisfies our assumptions, the noncommutative theory seems to be a natural choice due to the existence of Markov dilation with a.u. continuous path as stated in Theorem \ref{auct}. We will have more examples of this kind in the following.
\mar

\re
It was shown in \cite{JMe}*{Remark 1.3.2} that
\[
 \|T_t: L_1^0(\Lc(\fz_n))\to L_\8(\Lc(\fz_n))\|\le C t^{-3}.
\]
Therefore, Theorem \ref{tran} and Corollary \ref{tran3} give two different ways to prove the transportation inequality \eqref{wsen} for $\Lc(\fz_n)$.
\mar

\subsection{Application to the noncommutative tori $\R_\Theta$}
We recall the definition following \cite{JMP}. Let $\Theta$ be a $d\times d$ antisymmetric matrix with entries $0\le \tet_{ij}<1$. The noncommutative torus (or the rotation algebra) with $d$ generators associated to $\Theta$ is the von Neumann algebra $\R_\Theta$ generated by $d$ unitaries $u_1, \cdots ,u_d$ satisfying $u_j u_k= e^{2\pi i\tet_{jk}}u_k u_j$. Every element of $\R_{\Theta}$ is in the closure of the span of words of the form $w_k=u_1^{k_1} \cdots u_d^{k_d}$ for $k=(k_1, \cdots ,k_d)\in\zz^d$. $\R_\Theta$ admits a unique normal faithful trace $\tau$ given by $\tau(x)=\hat x(0)$ where $x=\sum_{k\in\zz^d} \hat x(k) u_1^{k_1} \cdots u_d^{k_d}\in\rz_\Theta$. Our goal is to show that $\R_\Theta$ admits a standard nc-diffusion semigroup with the $\Ga_2$-criterion. We start with the von Neumann algebra of $\zz^d$. It is well known that $\Lc(\zz^d)\cong L_\8(\tz^d)$.
We define $\psi(k)=\|k\|_1=\sum_{i=1}^d |k_i|$ for $k=(k_1, \cdots ,k_d)\in \zz^d$. Clearly, $\psi$ is a cn-length function and thus generate a standard nc-diffusion semigroup $P_t$ by Lemma \ref{gpvn}. In fact $P_t=\td{P}_t^{\otimes d}$ where $\td{P}_t$ is the Poisson semigroup on $L_\8(\tz)$.
\begin{prop}\label{cuzd}
 Let $\Ga$ be the gradient form associated to $P_t$. Then $\Ga_2\ge \Ga$ in $\Lc(\zz^d)$.
\end{prop}
\begin{proof}
Let $K^d$ be the Gromov form associated with $\psi$. A calculation shows that $K^d(j,k)=K(j_1,k_1)+ \cdots +K(j_d,k_d)$ for $j=(j_1, \cdots ,j_d),k=(k_1, \cdots ,k_d)\in\zz^d$, where $K$ is the Gromov form of $\zz=\fz_1$ considered in the proceeding subsection. Alternatively, we may write $K^d=\sum_{i=1}^d \mathds{1}\otimes \cdots \otimes K\otimes \cdots \otimes \mathds{1}$ where $K$ is in the $i$th position. But we know from Proposition \ref{frcur} that $\Ga_2\ge\Ga$ in $\Lc(\zz)$. The assertion follows from Lemma \ref{kform}.
\end{proof}
  \begin{prop}
  $\R_\Theta$ admits a standard nc-diffusion semigroup with $\Ga_2\ge \Ga$.
  \end{prop}
  \begin{proof}
  Let $k\in \zz^d$. Consider an action $\al: \tz^d\to \mathrm{Aut}(\R_\Theta)$ given by: for $s\in\tz^d$,
  $$\al_s(u_1^{k_1} \cdots u_d^{k_d}) = e^{2\pi i \sum_{j=1}^d k_j s_j}u_1^{k_1} \cdots u_d^{k_d}.$$
   It is easy to check that $\al_s$ is a trace preserving automorphism.
  Define a map
  \[
  \pi: \R_\Theta\to L_\8(\tz^d)\bar{\otimes} \R_\Theta, \quad w_k=u_1^{k_1} \cdots u_d^{k_d}\mapsto \pi(w_k)(s)=\al_s(w_k)=e^{2\pi i \lge k,s\rge } u_1^{k_1} \cdots u_d^{k_d}\pl.
  \]
  Then $\pi$ is an injective $*$-homomorphism. Define $T_t:\R_\Theta\to\R_\Theta, T_t(w_k)=e^{-t\|k\|_1}w_k$. We claim that $(T_t)_{t\ge0}$ is the desired semigroup. Indeed, by Lemma \ref{tform} and Proposition \ref{cuzd}, $P_t\otimes Id$ acting on $L_\8(\tz^d)\bar\otimes \R_\Theta$ is a standard nc-diffusion semigroup and satisfies $\Ga_2\ge\Ga$. Then since $\pi$ is injective and
  \[
  P_t\otimes Id (\pi(w_k))\lel e^{-t\|k\|_1} e^{2\pi i \lge k,\cdot\rge } \otimes u_1^{k_1} \cdots u_d^{k_d}\lel \pi(T_t(w_k))
  \]
 leaves $\pi(\R_\Theta)$ invariant, we deduce that $T_t$ is a standard nc-diffusion semigroup acting on $\R_\Theta$ with $\Ga_2\ge \Ga$.
  \end{proof}

\subsection{The finite cyclic group $\zz_n$}\label{cyc}
We consider the group von Neumann algebra $\Lc(\zz_n)$ in this subsection. Let $(e_j)_{j=1}^n$ be the standard basis of $\cz^n$. Each $e_j$ can be regarded as a vector in $\rz^{2n}$ by canonical identification. Given $k\in \zz_n$, define the $2n\times 2n$ diagonal matrix $\al_k = (e^{2\pi i kj/n})_{j=0}^{n-1}$ where each $e^{2\pi i kj/n}$ is on diagonal and is identified with the $2\times 2$ rotation matrix $$\left(\begin{array}{cc}
\cos(2\pi  kj/n) & -\sin(2\pi  kj/n)\\
 \sin(2\pi  kj/n)&\cos(2\pi  kj/n)                                                                                                                                                                                                                                                                                                        \end{array}\right).
$$
Consider the finite cyclic group $\zz_n$ with 1-cocycle structure $(b,\al, \rz^{2n})$, where
\[b(k)\lel \frac1{\sqrt{n}}\Big(\sum_{j=1}^n\al_k(e_j)-e_j\Big) \lel \frac1{\sqrt{n}}\sum_{j=1}^n {\cos(2\pi  k(j-1)/n)-1\choose  \sin (2\pi  k(j-1)/n)}\otimes e_j\pl.\]
Then the length function $\psi$ given by $\psi(g)=\|b(g)\|_2^2$ is conditionally negative; see e.g. \cite{BO}*{Appendix D}.
\begin{lemma}
Let $K(k,h)$ be the Gromov form. Then $K(k,h)=\lge b(k),b(h)\rge$.
\end{lemma}
\begin{proof}
 Since the length function $\psi(k)=\|b(k)\|^2$, by the cocycle property,
 \begin{align*}
 K(k,h)&\lel \frac12(\|b(k)\|^2+\|b(h)\|^2-\|b(h-k)\|^2)\\
 &\lel \frac12(\|b(-k)\|^2+\|\al_{-k}(b(h))\|^2-\|\al_{-k}(b(h))+b(-k)\|^2)\\
 &\lel -\lge b(-k), \al_{-k}(b(j))\rge = -\lge \al_k(b(-k)), b(h)\rge\\
 &\lel \lge b(k),b(h)\rge\pl. \qedhere
 \end{align*}
\end{proof}
Clearly, $K(k,h)=0$ if $k=0$ or $h=0$. For $k,h\neq 0$, a computation gives
\begin{align*}
K(k,h)&\lel \frac1n\sum_{j=0}^{n-1} \big[(1-\cos(2\pi  kj/n))(1-\cos(2\pi  hj/n)) + \sin (2\pi  kj/n)\sin (2\pi  hj/n)\big]\\
&\lel 1+\frac1n\sum_{j=0}^{n-1}\cos\Big(\frac{2\pi(k-h)j}{n}\Big)=1+\de_{k,h}\pl,
\end{align*}
where $\de_{k,h}$ is the Kronecker delta function. It follows that $\psi(k)=2(1-\de_{k,0})$. For reasons that will become clear later, we normalize $\psi$ and still denote it by $\psi$ so that $\psi(k)=1-\de_{k,0}$ for $k\in\zz_n$. Then the associated Gromov form satisfies $K_{k,h}=\frac12(1+\de_{k,h})$ for $k,h\neq 0$ and  $(K_{k,h}^2-\frac12 K_{k,h}) \ge0$. It is an immediate consequence of Lemma \ref{kform} that $\Ga_2\ge\frac12\Ga$ in $\Lc(\zz_n)$. In fact, we can do better.
\begin{prop}\label{zncur}
For all $0<\al\le \frac{n+2}{2n}$, we have $\Ga_2\gl \al \Ga$ in $\Lc(\zz_n)$. Moreover, $\al_n= \frac{n+2}{2n}$ is the largest possible $\al$ with the $\Ga_2$-criterion.
\end{prop}
\begin{proof}
Note that the $n\times n$ matrix $K$ can be written as a block matrix
\[K\lel \left(\begin{array}{cc}
              0& 0\\
              0& \frac12(I_{n-1}+\mathds{1}_{n-1})
              \end{array}\right)
\]
where $I_{n-1}$ is the $n-1$ dimensional identity matrix and every entry of $\mathds{1}_{n-1}$ is 1. Write $\widehat K=\frac12(I_{n-1}+\mathds{1}_{n-1})$. Since $\mathds{1}_{n-1}\le (n-1) I_{n-1}$, for $0<\al\le \frac{n+2}{2n}$ we have
\begin{align*}
  &4\widehat K\bullet \widehat K-4\al \widehat K\lel (3-2\al)I_{n-1}- (2\al-1)\mathds{1}_{n-1}\\
 \gl& (2+n-2\al n)I_{n-1}\gl 0 \pl.
  \end{align*}
  Plugging $\bm x=(\frac1{\sqrt{n-1}}, \cdots ,\frac1{\sqrt{n-1}})$ into $\bm{x}'(\widehat K\bullet \widehat K-\al \widehat K)\bm{x}\ge 0$ reveals that $\al_n = \frac{n+2}{2n}$ is sharp.
 Then Lemma \ref{kform} leads to the $\Ga_2$-criterion.
\end{proof}

\subsection{The discrete Heisenberg group $H_3(\zz_n)$}
We consider the Heisenberg group $H_3(\zz_n)=\zz_n\times \zz_n\times\zz_n$ over $\zz_n$ as a subalgebra of $M_3(\zz_n)$ as follows:
\[
\left(\begin{array}{ccc}
1&b&a\\
0&1&c\\
0&0&1
\end{array}\right) \left(\begin{array}{ccc}
1&b'&a'\\
0&1&c'\\
0&0&1
\end{array}\right) =
\left(\begin{array}{ccc}
1&b+b'&a+a'+bc'\\
0&1&c+c'\\
0&0&1
\end{array}\right);
\]
see e.g. \cite{Dav}*{Section VII.5} for more details. We will write $H$ for $H_3(\zz_n)$ as long as there is no confusion. The multiplication here is given by
\[ (a,b,c)(a',b',c')\lel (a+a'+bc', b+b',c+c'),\quad (a,b,c),(a',b',c')\in H\pl.\]
Other multiplications have been considered in the literature.
\begin{prop}\label{hscur}
 Let $\psi(a,b,c)=2-\de_{b,0}-\de_{c,0}$. Then
 \begin{enumerate}
  \item $\psi$ is conditionally negative and thus the semigroup $(T_t)$ determined by $\psi$ is a standard nc-diffusion semigroup.
  \item Let $\Ga$ be the gradient form associated to $\psi$. Then $\Ga_2\ge \frac{n+2}{2n}\Ga$ in $\Lc(H)$.
   \end{enumerate}
\end{prop}
\begin{proof}
 (1) The length function of $\zz_n$ considered in Subsection \ref{cyc} is given by $\td{\psi}(k)=(1-\de_{k,0})$, which extends to $\zz_n\times\zz_n$ as $\td{\psi}(k,l)=(1-\de_{k,0})+(1-\de_{l,0})$. Define a group homomorphism
 \[ \bt: H\to \zz_n\times\zz_n, \quad (a,b,c)\mapsto (b,c).\]
 Since $\td{\psi}$ is conditionally negative, it follows from the definition that $\psi = \td{\psi}\circ \bt$ is also conditionally negative. Lemma \ref{gpvn} yields that $(T_t)$ is a standard nc-diffusion semigroup.

 (2) Let $K$ and $\td{K}$ be the Gromov form of $(H,\psi)$ and $(\zz_n,\td{\psi})$ respectively. A calculation shows that for indices $(a,b,c), (a',b',c')\in H$,
 \begin{align*}
    K((a,b,c),(a',b',c'))&\lel \td{K}(b,b')+\td{K}(c,c')\\
    &\lel (\td{K}\otimes \mathds{1} + \mathds{1}\otimes \td{K})((b,b')\otimes(c,c'))\pl.
    \end{align*}
By Proposition \ref{zncur} and Lemma \ref{tens} with $m=2$, we have $\Ga_2^{K}\ge \frac{n+2}{2n}\Ga^{K}$ in $\Lc(H)$, as desired.
\end{proof}

Let $e_{i,j}$ be the standard basis of the matrix algebra $M_n(\cz)$ and $\de_j$ the standard basis of $\ell_2(\zz_n)$. Define the diagonal matrix $u_k=\sum_{j=1}^{n}e^{2\pi i k(j-1)/n}\otimes e_{j,j}$ and the shift operator $v_l(\de_j)=\de_{j+l}$ which is nothing but the left regular representation of $\zz_n$ on $\ell_2(\zz_n)$. It is easy to see that $u_k, v_l\in M_n=B(\ell_2(\zz_n))$ and they satisfy $u_k v_l=e^{2\pi ikl/n}v_l u_k$.
\begin{prop} \label{hsdc}
Let $\Lc(H)$ be the group von Neumann algebra of $H$. Then
\[ \Lc(H)\cong L_\8(\zz_n^2)\oplus M_n\oplus \M_2\oplus \cdots  \oplus \M_{n-1}\pl,\]
where $\M_x, x=2, \cdots ,n-1$ are von Neumann algebras acting on $\ell_2(\zz_n^2)$.
Moreover, if $T_t$ is the semigroup associated to $\psi(a,b,c)=2-\de_{b,0}-\de_{c,0}$, then $T_t$ leaves each component invariant and $T_t|_{\M_x}$is a standard nc-diffusion semigroup.
\end{prop}
\begin{proof}
Let us first determine the center of $\Lc(H)$ denoted by $\Z$. The identity
$$\la(a,b,c)\la(a',b',c')=\la(a',b',c')\la(a,b,c)$$
for all $(a',b',c')\in H$ holds if and only if $b=c=0$. Thus $\Lc(\zz_n,0,0)\subset \Z$. Let $\F$ denote the discrete Fourier transform of the first component on $\ell_2(H)$. For $\de_{(x,0,0)}\in\ell_2(\zz_n,0,0)$, we have
\[\F(\de_{(x,0,0)})\lel \frac1{\sqrt{n}}\sum_{k=0}^{n-1}e^{-\frac{2\pi i k x}n}\de_{(k,0,0)}\pl.\]
A calculation gives
\[\F\la(a,0,0)\F^{-1} \de_{(x,0,0)}\lel e^{-\frac{2\pi i a x}n}\de_{(x,0,0)}\pl.\]
This shows that $\F\Lc(\zz_n,0,0)\F^{-1}=\{e^{-\frac{2\pi i a \cdot}n}: a\in \zz_n\}'' = L_\8(\zz_n)$. Since for fixed $x\in\zz_n$, $\de_{(x,\cdot,\cdot)}\in \ell_2(\zz_n^2)$, we have the Hilbert space decomposition $\ell_2(H)=\bigoplus_{x\in \zz_n} \ell_2^x(\zz_n^2)$, where the superscript $x$ is used to distinguish different copies. We may drop $x$ if there is no ambiguity. Then by the decomposition theorem of von Neumann algebras for subalgebras of the center (see e.g. \cite{Tak}*{Theorem IV.8.21}),
\[\Lc(H)\cong \bigoplus_{x\in\zz_n} \M_x\pl,\]
where $\M_x$ is determined by the unitary $\F$, the discrete Fourier transform on the first component. Let $q_x\in \ell_\8(\zz_n,0,0)$ be the central projection given by $q_x: \ell_2(H)\to {\rm span}\{\de_{(x,y,z)}: y,z\in\zz_n\}$. Put $p_x=\F^{-1}q_x\F\in \Lc(\zz_n,0,0)$. Then $p_x$ is a central projection, $\sum_{x=0}^{n-1}p_x=1$ and $p_x \Lc(H)=\F^{-1}\M_x\F$. We observe that $T_t\la(a,0,0)=\la(a,0,0)$. Hence $\Lc(\zz_n,0,0)$ is contained in the multiplicative domain of $T_t$. Then by the property of multiplicative domain (see e.g. \cite{Pau}*{Theorem 3.18}) for all $x\in\zz_n,\xi\in\Lc(H)$, $T_t(p_x \xi)=p_xT_t(\xi)\in p_x\Lc(H)$. Let $\eta\in\M_x$ with $p_x\xi=\F^{-1}\eta\F$. We define $\td{T}_t\eta=\F T_t(p_x\xi)\F^{-1}$. Since $\F$ is a unitary, $\td{T}_t$ restricted to $\M_x$ is a standard nc-diffusion semigroup. By abuse of notation, we will write $T_t$ for $\td{T}_t$ on $\M_x$ in the future.

To get more precise description of $\M_x$, we define a family of maps for $x\in\zz_n$
\begin{align*}
 \pi_x: \Lc(H)\to B(\ell_2^x(\zz_n^2)),\quad \la(a,b,c)\mapsto\pi_x(\la(a,b,c)),
 \end{align*}
where $\pi_x(\la(a,b,c))$ acts on $\de_{(k,l)}$ by
\[
\F\la(a,b,c) \F^{-1}\de_{(x,k,l)}\lel e^{-\frac{2\pi i x (a+bl)}n}\la_{(b,c)}\de_{(k,l)}\pl.
\]
Here $\la(b,c)$ is the shift operator on $\ell_2(\zz_n^2)$ given by $\la{(b,c)}\de_{(k,l)}=\de_{(k+b,l+c)}$. Then \[\M_x=\{\pi_x(\la(a,b,c)): (a,b,c)\in H\}''=\{e^{-\frac{2\pi i ax}n}v_b\otimes(v_c u_{-xb}):(a,b,c)\in H\}''\pl.\]
Here we have used the convention $\la(b,c)=v_b\otimes v_c$. If $x=0$, we have
$$\M_0=\{\la(b,c):(b,c)\in\zz_n^2\}''=\Lc(\zz_n^2)=L_\8(\zz_n^2).$$
If $x=1$, it can be checked that $\{v_c u_{-b}: (b,c)\in \zz_n^2\}''=M_n$; see e.g. \cite{Dav}*{Theorem VII.5.1}. Define for $(b,c)\in \zz_n^2$
\[\rho(v_b\otimes(v_c u_{-b})) \lel v_c u_{-b}\pl.\]
Then $\rho$ is a $*$-isomorphism and thus $\M_1=M_n$.
\end{proof}
Consider the semigroup $T_t$ acting on $M_n(\cz)$ defined by $T_t|_{M_n(\cz)}$ in the preceding proposition. Explicitly, $T_t$ is determined by $T_t (v_c u_b) = e^{-t\psi(b,c)}(v_c u_b)$ where $\psi(b,c)=2-\de_{b,0}-\de_{c,0}$. Then the $\Ga_2$-criterion for $M_n$ follows from Proposition \ref{hscur}. We record this fact below.
\begin{prop}\label{mncur}
 $M_n$ admits a standard nc-diffusion semigroup $(T_t)_{t\ge 0}$. Let $\Ga^{M_n}$ be the gradient form associated to $T_t$. Then
$\Ga^{M_n}_2\ge \frac{n+2}{2n}\Ga^{M_n}$ in $M_n$.
\end{prop}

\subsection{Application to the generalized Walsh system}
Let us recall some basic facts about the Walsh system following \cite{ELP}. Let $\Om_n^m=\{1,e^{2\pi i/n}, e^{2\pi i 2/n}, \cdots ,e^{2\pi i(n-1)/n}\}^m$  be the $m$-dim discrete cube equipped with uniform probability measure $P$. Let $\om_j,j=1, \cdots ,m$ denote the $j$th coordinate function on $\Om_n^m$. For a nonempty subset $B\subset \{1, \cdots ,m\}$ and $\bm x=(x_1,\cdots,x_m)\in\zz_n^m$, define
\[\om_B(\bm x)\lel \prod_{j\in B}\om_j^{x_j},\]
and $\om_\emptyset = 1$. Put $G=\{\om_B(\bm x):B\subset \{1, \cdots ,m\}, \bm x\in \zz_n^m\}$. Then $G$ is clearly a group and $L_\8(\Om_n^m)$ is spanned by the elements of $G$.
%Denote by $G$ the unitary group of the algebra $L_\8(\Om_n^m)$ generated by$\{\om_B\}$ with $\om_B^{-1}=\om_B^{n-1}$ and  the elements of $G$ form an orthonormal basis of $L_2(\Om_n^m)$.

We consider the abelian group $\zz_n^m$. Define
\[\psi(x_1, \cdots ,x_m)\lel (m-\de_{x_1}- \cdots -\de_{x_m}),\quad (x_1, \cdots ,x_m)\in \zz_n^m \pl,\]
where $\de_x = \de_{x,0}$. Given $\bm x =(x_1, \cdots ,x_m)$, put $B_{\bm x}=\{i: x_i\neq0, i=1, \cdots ,m\}$.
Clearly $\psi(x_1, \cdots ,x_m)=|B_{\bm x}|$, where $|B|$ is the cardinality of $B$. An argument similar to the proof of Proposition \ref{hscur} shows that $\psi$ is a cn-length function and the associated $\Ga$ form satisfies
\begin{equation}\label{z2nc}
 \Ga_2\ge \frac{n+2}{2n}\Ga \text{ in } \Lc(\zz_n^m)\pl.
\end{equation}
Here the constant $\frac{n+2}{2n}$ is given by Proposition \ref{zncur}. Define a map
\[\bt : \zz_n^m\to G,\quad \bm x =(x_1, \cdots ,x_m)\mapsto \prod_{j\in B_{\bm x}}\om_j^{x_j}\pl. \]
It is easy to check that $\bt$ is a group isomorphism from $\zz_n^m$ to $G$. The idea here is to convert addition to multiplication. Under the identification $\bt$, $\Lc(\zz_n^m)=L_\8(\Om_n^m)$ and thus every $f\in \Lc(\zz_n^m)$ can be written as
\[f \lel \ez(f) + \sum_{\bm x\in \zz_n^m, \bm x\neq \bm 0} \hat f_{\bm x} \prod_{j\in B_{\bm x}}\om_j^{x_j}\pl, \]
where $\ez(f)=\tau(f)$ is the expectation associated to the uniform probability.
By abuse of notation, we still denote by $\psi$ the cn-length function induced by $\bt$ on $\{\om_B\}$, i.e.
\[\psi(\om_{B_{\bm x}})\lel \psi(\bt(\bm x)) := \psi(\bm x),\quad \bm x\in \zz_n^m\pl.\]
Then we have
\[\psi(\om_{B_{\bm x}})\lel \psi(x_1, \cdots ,x_m)\lel |B_{\bm x}|\pl.\]
Therefore the infinitesimal generator $A$ of the heat semigroup $T_t$ in this case is the number operator for the generalized Walsh system which counts non-zero elements
$$A\pl\om_B\lel |B|\om_B\pl.$$
Moreover, it follows from \eqref{z2nc} that
$\Ga_2\ge \frac{n+2}{2n}\Ga $ in $L_\8(\Om_n^m)$.
The case $n=2$ is of particular interest. Indeed, we may write $f\in \Lc(\zz_2^m)$ as
\[f \lel \ez(f) + \sum_{B\subset\{1, \cdots ,m\},B\neq\emptyset} \hat f_B \om_B\pl. \]
Note that in this case $\om_B^{-1}\om_C=\om_{B \triangle C}$. Then the Gromov form of $\{\om_B\}$ is given by
\[K(\om_B, \om_C) \lel \frac12(|B|+|C|-|B \triangle C|)\lel |B\cap C|\pl.\]
Hence, we find the gradient form
\[\Ga(f,f)\lel \sum_{B,C\subset \{1, \cdots ,m\}}\bar{\hat f}_B\hat f_C|B\cap C|\om_{B\triangle C}\pl.\]
Let $e_j=(1, \cdots ,-1, \cdots ,1)$ where $-1$ is only at the $j$th position. For $x\in\Om_n^m$, put $(\partial_j f) (x)=\frac12(f(x)-f(xe_j))$ and define the discrete gradient $\nabla f = (\partial_j f)_{j=1}^m$. Then a calculation gives $\Ga(f,f)=|\nabla f|^2$, where $|\cdot|$ is the Euclidean norm of a vector in $\cz^n$. If we simply write $|\nabla f|=\Ga(f,f)^{1/2}$ for any $n=2,3, \cdots $, our Poincar\'e inequalities for the generalized Walsh system is a dimension $m$ free estimate.
\begin{cor}
Let $2\le p<\8$. Then for all $f\in L_p(\Om_n^m,P)$,
\[ \|f-\ez(f)\|_p\kl C\sqrt{\frac{2n}{n+2}}\min\{\sqrt{p}\|\pl|\nabla f|\pl\|_\8,\pl{p}\|\pl|\nabla f|\pl\|_p\}\pl.\]
\end{cor}
 One of Efraim and Lust-Piquard's main results in \cite{ELP} asserts that for $2<p<\8$ and $f\in L_p(\Om_2^m,P)$,
\[ \|f-\ez(f)\|_p\kl C\sqrt{p}\|\pl|\nabla f|\pl\|_p\pl.\]
Our version for the case $n=2$ is weaker. But our approach works from general $n$ while their result is only for $n=2$. For arbitrary $n>2$, it is unclear to us whether it is possible to obtain similar results based on their method. Moreover, their concentration inequality \cite{ELP}*{Corollary 4.2} due to Bobkov and G\"otze is now a special case of our \eqref{con1} with the same order. Efraim and Lust-Piquard's inequality would follow from our general theory if \eqref{bbur} were true.

\subsection{The $q\dash$Gaussian algebras}
We first recall some definitions and basic facts following \cite{BKS}. Throughout this section $-1\le q\le 1$. Let $\Ha$ be a separable real Hilbert space with complexification $\Ha_\cz$. Let $(F_q(\Ha),\lge\cdot,\cdot\rge_q)$  be the $q$-Fock space with vacuum vector $\Om$ and $\Ga_q(\Ha)$ the $q$-Gaussian algebra which is the von Neumann algebra generated by $s(f)=l(f)+l^*(f)$ for $f\in\Ha$ where
\[l^*(f)f_1\otimes \cdots \otimes f_n\lel f\otimes f_1\otimes \cdots \otimes f_n\]
and
\[l(f)f_1\otimes \cdots \otimes f_n\lel \sum_{j=1}^n q^{j-1}\lge f_j,f\rge f_1\otimes \cdots \otimes f_{j-1}\otimes f_{j+1}\otimes  \cdots \otimes f_n\pl\]
are the creation and annihilation operators respectively. The vacuum vector gives rise to a canonical tracial state $\tau_q(X)=\lge X\Om,\Om\rge_q$ for $X\in \Ga_q(\Ha)$. The $q$-Ornstein--Uhlenbeck semigroup $T_t^q=\Ga_q(e^{-t}I_\Ha)$ is a standard semigroup and extends to a semigroup of contractions on $L_p$ spaces. The generator on $L_2$ is the number operator $N^q$ which acts on the Wick product by
\[N^q W(f_1\otimes \cdots \otimes f_n)= n W(f_1\otimes \cdots  \otimes f_n), \quad f_1, \cdots ,f_n\in\Ha_\cz\pl,\]
where $W$ is the Wick operator. It is easy to check that $(T_t)$ is a nc-diffusion semigroup.

Let $\ell_2^n$ be the real Hilbert space with dimension $n$ and $\{e_1, \cdots ,e_n\}$ an orthonormal basis. For $j=1, \cdots , n$, consider the embedding
\[\iota_j: \Ha\to \Ha\otimes\ell_2^n,\quad h\mapsto  h\otimes e_j\pl.\]
According to \cite{BKS}*{Theorem 2.11}, there exists a unique map $\Ga_q(\iota):\Ga_q(\Ha)\to \Ga_q(\Ha\otimes \ell_2^n)$ such that $\Ga_q(\iota_j)(s(h))=s(h\otimes e_j)$.  The map $\Ga_q(\iota)$ is linear, bounded, unital completely positive and preserves the canonical trace. Define $s_j^q(h)=s(h\otimes e_j)$. If $q=1$, we write $g_j(h)=s(h\otimes e_j)$ and it is well-known $g_j(h)$ is a standard Gaussian random variable if $\|h\|=1$. For $h\in\Ha$, put
\[u_n(h) \lel \frac1{\sqrt{n}}\sum_{j=1}^n s_j^q(h)\otimes g_j(h)\pl.\]
Write $\M_n^q = \Ga_q(\Ha\otimes \ell_2^n)$. We consider the von Neumann algebra ultraproduct $\M= \prod_\U \M_n^q\bar\otimes \M_n^1$. Any element of $\M$ can be written as $u_\om(h)=(u_n(h))^\bullet$. We need the following fact proved in \cite{AJ}.
\begin{lemma}\label{qext}
 The map $s(h)\mapsto u_\om(h)$ extends to an injective trace preserving *-homomorphism $\pi: \Ga_q(\Ha)\to \M$. Moreover, for $x\in\Ga_q(\Ha)$,
 \[\pi(T_t^q(x))\lel (Id\otimes T_t^1)^\bullet \pi(x)\pl.\]
\end{lemma}

\begin{prop}
 For $-1\le q\le 1$, $\Ga_2\ge \Ga$ in $\Ga_q(\Ha).$
\end{prop}
\begin{proof}
 Since $\pi$ is injective, it suffices to prove $\pi\Ga_2^{N^q}\ge \pi\Ga^{N^q}$ in $\M$. By Lemma \ref{qext}, we have $\pi N^q (x)=(Id\otimes N^1)^\bullet\pi(x)$. It follows that
 $$\pi(\Ga^{N^q}(x,y))\lel \Ga^{(Id\otimes N^1)^\bullet}(\pi(x),\pi(y)),$$
 and similar identity is true for $\Ga_2^{N^q}$. It is proved in \cite{AJ} by using the central limit theorem of Speicher \cite{Sp92} that $\Ga_2\ge \Ga$ in $\Ga_1(\Ha)$. Here $\Ga_1(\Ha)$ is the von Neumann algebra acting on the symmetric Fock space. It follows that for all $n\in\nz$ and in $\M_n^q\otimes \M_n^1$, $\Ga_2^{Id\otimes N^1}\gl  \Ga^{Id\otimes N^1}$.
 Hence, we find
 \[\Ga_2^{(Id\otimes N^1)^\bullet}(\pi(x_i),\pi(x_j))\gl \Ga^{(Id\otimes N^1)^\bullet}(\pi(x_i),\pi(x_j))\]
 where for any $m\in\nz$ and $x_i\in \Ga_q(\Ha), i=1, \cdots ,m$, as desired.
\end{proof}

\subsection{The hyperfinite $II_1$ factor}
Our goal in this subsection is to show that the hyperfinite $II_1$ factor $R$ admits different standard nc-diffusion semigroups with $\Ga_2$-criterion and that the best possible $\al$ characterizes the corresponding dynamical system. It is well known that $R$ can be approximated by matrix algebras $\{M_{n^k}:k\in\nz\}$. We will embed $M_{n^{m/2}}$ into the group von Neumann algebra of the generalized discrete Heisenberg group $H_n^{m+1}=\zz_n/2\times\zz_n^{m}$.

Let $\Theta=(\tet_{jk})$ be an antisymmetric $m\times m$ matrix with $\tet_{jk}=\frac12$ if $j<k$. The multiplication in $H_n^{m+1}=\zz_n/2\times\zz_n^{m}$ is given by
\[(x,\xi)(y,\eta)\lel (x+y+B(\xi,\eta), \xi+\eta)\pl,\]
where $B:\zz_n\times \zz_n\to \zz_n/2$ is a bilinear form given by $B(\xi,\eta)=\sum_{j,k=1}^m\tet_{jk}\xi_j\eta_k=\lge \xi,\Theta\eta\rge$. For $(r,\xi)\in H_n^{m+1}$, put $\psi(r,\xi)=\sum_{j=1}^m 1-\de_{\xi_j,0}=\#\{\xi_j\neq 0\}$. Define a semigroup acting on $\Lc(H_n^{m+1})$ by $T_t \la(r,\xi)=e^{-t\psi}\la(r,\xi)$ for $\la(r,\xi)\in\Lc(H_n^{m+1})$ and $t\ge0$. Using Lemma \ref{tens}, an argument similar to Proposition \ref{hscur} shows that $(T_t)_{t\ge0}$ is a standard nc-diffusion semigroup and that the associated gradient form satisfies $\Ga_2\ge\frac{n+2}{2n}\Ga$ in $\Lc(H_n^{m+1})$.
\begin{lemma} Let $m\ge2$ be an even integer and $n\ge 3$. We have
 \[\Lc(H_n^{m+1})\cong \bigoplus_{x\in \zz_n/2} \M_x\pl,\]
 where $\M_2=M_{n^{m/2}}$. Furthermore, $T_t$ leaves each $\M_x$ invariant.
\end{lemma}
\begin{proof} Most of the argument utilizes the proof of Proposition \ref{hsdc} and \cite{JMP}*{Lemma 5.3}. Note that $\la(\zz_n/2,0)$ lives in the center of $\Lc(H_n^{m+1})$. By the decomposition of von Neumann algebras for the subalgebras of the center we obtain the first assertion. Write $e_j,j=1, \cdots ,m$ for the canonical basis of $\zz_n^m$ and put $u_r^j=\la(0,re_j)$ for $r\in\zz_n$. Then these $u_r^j$'s generate $\la(0,\zz_n^m)$ and $u_r^ju_s^k(\de_{(x,\cdot)})=e^{{2\pi i rs\tet_{jk}x}/n} u_s^ku_r^j(\de_{(x,\cdot)})$. Acting on $H_2:={\rm span}\{\de_{(2,\cdot)}\}$, $u_r^j$'s satisfy $u_r^ju_s^k=e^{{2\pi i rs}/n} u_s^ku_r^j$ for  $j<k$ and $u_r^ju_s^k=e^{-{2\pi i rs}/n} u_s^ku_r^j$ for $j>k$. It is clear that $y(r_1, \cdots ,r_m)=u_{r_1}^1 \cdots  u_{r_m}^m$ is a basis for $\M_2$ which satisfies the equation
\[u_r^jy(r_1, \cdots ,r_m){u_r^j}^*\lel C(r,j,r_1, \cdots ,r_m)y(r_1, \cdots ,r_m)\pl,\]
where $C(r,j,r_1, \cdots ,r_m)=\exp(2\pi i r(r_1+ \cdots +r_{j-1}-r_{j+1}- \cdots -r_m)/n)$.
In order to determine the center of $\M_2$, we consider the equation $C(r,j,r_1, \cdots ,r_m)=1$ for all $r\in\zz_n, j=1,\cdots,m$. This leads to a linear system over $\zz_n$
\begin{align*}
 -r_2- \cdots -r_m&\lel 0,\\
 r_1-r_3- \cdots -r_m &\lel 0,\\
 \vdots \quad&\quad\vdots\\
 r_1+ \cdots +r_{m-1}&\lel 0.\\
\end{align*}
Solving this system, we find $r_1= \cdots =r_m=0$. Here we used the crucial assumption that $m$ is even. Hence $\M_2$ has trivial center. Since it has dimension $n^m$, it follows that $\M_2=M_{n^{m/2}}$, as desired. By restricting $T_t$ to $M_{n^{m/2}}$ and repeating the argument of Proposition \ref{hsdc}, we can prove the last assertion.
\end{proof}
It follows from the lemma that $M_{n^{m/2}}$ admits a standard nc-diffusion semigroup $T_t$ with $\Ga_2\ge\frac{n+2}{2n}\Ga$ in $M_{n^{m/2}}$ for all $m\in2\nz$. Since the hyperfinite $II_1$ factor $R$ is the weak closure of $\cup_{k=1}^\8 M_{n^{k}}$, we have proved the following result.
\begin{prop}\label{hyp}
 For any integer $n\ge 2$, there exists a standard nc-diffusion semigroup $T_t^n$ acting on $R$ such that the associated gradient form $\Ga^n$ satisfies $\Ga_2^n\ge \frac{n+2}{2n}\Ga^n$ in $R$. The constant $\al_n=\frac{n+2}{2n}$ is best possible.
\end{prop}
The last conclusion follows from Proposition \ref{zncur}. We want to show that the semigroups $T^n_t$ for different $n$ are different. Let us now recall a definition from dynamical systems; see e.g. \cite{Wal}*{Definition 2.4}. Let $(X,\B_X,\mu, T)$ be a measure-preserving dynamic system (MPDS) where $(X,\B_X,\mu)$ is a probability space and $T$ is a measure-preserving transformation. A MPDS $(Y,\B_Y, \nu,S)$ is said to be isomorphic to $(X,\B_X,\mu, T)$ if there exist (i) full measure sets $X_1\subset X$ and $Y_1\subset Y$ such that $T(X_1)\subset X_1$ and $S(Y_1)\subset Y_1$; and (ii) an invertible measure-preserving measurable map $\phi: X\to Y$ such that $\phi(T x)=S(\phi x)$ for all $x\in X_1$. This motivates our following definition.
\begin{defi}
 Let $S_t$ and $T_t$ be standard semigroups acting on noncommutative probability spaces $(\N,\tau)$ and $(\M,\tau')$ respectively. We say $(\M,T_t)$ and $(\N, S_t)$ are isomorphic if there exist $\la>0$ and a trace-preserving $*$-isomorphism $\phi:\M\to\N$ such that $S_{\la t}(\phi x)=\phi(T_t x)$ for all $x\in\M$.
\end{defi}
The following result shows that $R$ admits infinitely many non-isomorphic standard nc-diffusion semigroups.
\begin{prop}
Let $T_t^n$ be the semigroup considered in Proposition \ref{hyp}. If $(R, T_t^n)$ and $(R, T_{t}^{n'})$ are isomorphic, then $\al_n=\al_{n'}$.
\end{prop}
\begin{proof}
 There exists a trace-preserving $*$-isomorphism $\phi: R\to R$ such that
 \begin{equation}\label{isom}
  \phi(T_t^n x)=T_{\la t}^{n'}(\phi x)
 \end{equation}
 for $x\in\M$. Let $A^n$  be the generator of $T_t^n$. $\Ga_2^{A^{n'}}\ge \frac{n'+2}{2n'}\Ga^{A^{n'}}$ implies $\Ga_2^{\la A^{n'}}\ge \la\frac{n'+2}{2n'}\Ga^{\la A^{n'}}$. This together with \eqref{isom} gives
 $\Ga_2^{A^n}\ge \la \frac{n'+2}{2n'}\Ga^{A^n}$. But the best $\al$ is $\al_n=\frac{n+2}{2n}$. Hence we have $\frac{n+2}{2n}\ge \la\frac{n'+2}{2n'}$. It is clear that ${\rm sp}(A^n)=\nz$ and ${\rm sp}(\la A^n)=\la\nz$. Here ${\rm sp}(A^n)$ denotes the spectrum of $A^n$. \eqref{isom} implies ${\rm sp}(\la A^n)={\rm sp}(A^n)$ and thus $\la= 1$.  Hence $n'\ge n$. Repeating the argument by starting from $\Ga_2^{n}\ge \frac{n+2}{2n}\Ga^{n}$ gives $n\ge n'$.
\end{proof}

\section{Tensor products and free products}
In this section we will construct further examples with the $\Ga_2$-criterion based on the examples considered in the previous section. This is done via the powerful algebraic tools -- tensor products and free products. It is not difficult to see that the property ``standard nc-diffusion'' is stable under tensor products and free products. Due to the reason explained in the previous section, it suffices to consider the algebraic $\Ga_2$-condition. That is, we always work with a dense subalgebra contained in the domain of the form under consideration.
\subsection{Tensor products}
The following result is our starting point to understand tensor products.
\begin{lemma}\label{tform}
 Let $\Theta:\A\times \A\to \M$ and $\Phi:\B\times\B \to \N$ be positive sesquilinear forms, where $\A\subset \M$ and $\B\subset\N$ are dense subalgebras so that $\Theta$ and $\Phi$ are well-defined. Then $\Theta\otimes \Phi: \A\otimes \B\times \A\otimes \B\to \M\otimes \N$ is positive where for $\xi^i=\sum_{k=1}^{n_i} x_k^i\otimes y_k^i\in \A\otimes \B$,
 \[\Theta\otimes \Phi\Big( \xi^i, \xi^j\Big):=\sum_{k=1}^{n_i}\sum_{l=1}^{n_j}\Theta(x_k^i,x_l^j)\otimes\Phi(y_k^i,y_l^j)\pl.\]
\end{lemma}
\begin{proof}
  For $r\in\nz$, let $(x_k^i)\subset \A, (y_k^i)\subset \B$ where $k=1, \cdots ,n_i, i=1, \cdots ,r$. Put $m=\sum_{i=1}^r n_i$. Without loss of generality we may assume $n_i=n$ for $i=1,\cdots,r$. Suppose $\M$ and $\N$ act on Hilbert spaces $H$ and $K$ respectively. Then $(\Theta(x_k^i,x_l^j))_{k,l,i,j}=\sum_{i,j,k,l} \Theta(x_k^i,x_l^j)\otimes e_{(k,i),(l,j)}\ge 0$ as an operator on $\ell_2^m(H)$ where $1\le k,l\le n$. Similarly, $(\Phi(y_k^i,y_l^j))_{k,l,i,j}\ge0$ on $\ell_2^m(K)$. It follows that
  \begin{align*}
  &(\Theta(x_k^i,x_l^j))\otimes (\Phi(y_{k'}^{i'},y_{l'}^{j'}))\\
  \lel& \sum_{i,j,k,l,i',j',k',l'}\Theta(x_k^i,x_l^j)\otimes\Phi(y_{k'}^{i'},y_{l'}^{j'})\otimes e_{(k,i),(l,j)}\otimes e_{(k',i'),(l',j')}\gl 0
  \end{align*}
  on $\ell_2^m(H)\otimes\ell_2^m(K)$. Define
  \[v:\ell_2^m(H\otimes K)\to \ell_2^m(H)\otimes\ell_2^m(K), \quad \sum_{s}(\xi_t^s\otimes\eta_t^s)\otimes e_t\mapsto \sum_{s}(\xi_t^s\otimes e_t)\otimes (\eta_t^s\otimes e_t)\pl.\]
  Here $\xi_t^s\in H$, $\eta_t^s\in K$, and $(e_t)$ is the canonical basis of $\ell_2^m$ for $t=1,\cdots, m$. Then
  \[v^*\Big[\sum_s(\xi_t^s\otimes e_t)\otimes (\eta_t^s\otimes e_t)\Big]=\sum_{s}(\xi_t^s\otimes\eta_t^s)\otimes e_t\pl.\] It is clear that $v^*[(\Theta(x_k^i,x_l^j))\otimes (\Phi(y_{k'}^{i'},y_{l'}^{j'}))] v\ge 0$. But
  \begin{align*}
   v^*[(\Theta(x_k^i,x_l^j))\otimes (\Phi(y_{k'}^{i'},y_{l'}^{j'}))] v&\lel \sum_{i,j,k,l}\Theta(x_k^i,x_l^j)\otimes \Phi(y_{k}^{i},y_{l}^{j})\otimes e_{(k,i),(l,j)}\\
   &\lel [\Theta(x_k^i,x_l^j)\otimes\Phi(y_{k}^{i},y_{l}^{j})]_{(k,i),(l,j)}\pl.
   \end{align*}
  Let $\bm{1}_n$ be a $n\times 1$ column vector with each entry equal to $1_\M\otimes 1_\N$ and $I_r$ the $r\times r$ identity matrix. Define an operator $w: \ell_2^r(H\otimes K)\to \ell_2^m(H\otimes K)$ by
  \[
   w\lel \bm{1}_n\otimes I_r\lel \left(\begin{array}
            {cccc}
            \bm{1}_n& && \\
            &\bm{1}_n & &\\
            && \ddots &\\
            &&& \bm{1}_n
           \end{array}
\right)\pl.
  \]
$w$ is an $m\times r$ matrix. Note that $[\Theta(x_k^i,x_l^j)\otimes\Phi(y_{k}^{i},y_{l}^{j})]$ is an $m\times m$ matrix. Then
  \begin{align*}
   0&\kl w^* \Big[\sum_{i,j,k,l}\Theta(x_k^i,x_l^j)\otimes \Phi(y_{k}^{i},y_{l}^{j})\otimes e_{(k,i),(l,j)}\Big]w\\
   &\lel \sum_{i,j=1}^r \sum_{k=1}^{n_i}\sum_{l=1}^{n_j}\Theta(x_k^i,x_l^j)\otimes\Phi(y_k^i,y_l^j)\otimes e_{i,j}\\
   &\lel \sum_{i,j=1}^r \Theta\otimes \Phi(\xi^i,\xi^j)\otimes e_{i,j}\pl,
  \end{align*}
  which completes the proof.
\end{proof}

\begin{lemma}
 \label{vnte}
 Let $(T_t)_{t\ge0}$ and $(S_t)_{t\ge0}$ be standard semigroups with generator $A$ and $B$ acting on finite von Neumann algebras $\M$ and $\N$ respectively such that $\Ga^A_2\ge\al\Ga^A$ in $\M$ and $\Ga^B_2\ge\al \Ga^B$ in $\N$. Then
 $\Ga_2^{A\otimes I}\ge \al \Ga^{A\otimes I}$, $\Ga_2^{I\otimes B}\ge \al \Ga^{I\otimes B}$ and $\Ga_2^{A\otimes I+I\otimes B}\ge \al \Ga^{A\otimes I+I\otimes B}$ all in $\M\otimes \N$.
\end{lemma}
\begin{proof}
The first inequality follows from Lemma \ref{tform} with $\Theta=\Ga^A$ and $\Phi(y_1,y_2)=y_1^*y_2$. The second inequality can be shown similarly. For the last one, note that
 \begin{align*}
   &\Ga_2^{A\otimes I+I\otimes B}-\al\Ga^{A\otimes I+I\otimes B}\\
   \lel &(\Ga^{A\otimes I}_2-\al\Ga^{A\otimes I})+(\Ga^{I\otimes B}_2-\al\Ga^{I\otimes B})+2\Ga^{A}\otimes \Ga^{ B}\pl.
  \end{align*}
   Then the first two inequalities and Lemma \ref{tform} with $\Theta=\Ga^{A}$ and $\Phi=\Ga^{B}$ yield the assertion.
\end{proof}
\begin{prop}\label{tega}
 Let $A_j$ be self-adjoint generators of standard nc-diffusion semigroups $(T_t^{A_j})$ acting on $\N_j$ and $\Gamma^{A_j}_2\gl \al \Gamma^{A_j}$ respectively for $j=1,..,n$ with the same constant $\al>0$. Then the tensor product generator $\otimes A_j(x_1\otimes \cdots  \otimes x_n)=\sum_{j} x_1\otimes  \cdots  \otimes x_{j-1}\otimes  A_j(x_j)\otimes x_{j+1}\otimes  \cdots  \otimes x_n$ generates a standard nc-diffusion semigroup $(T_t^{\otimes A_j})$ with \[ \Gamma_2^{\otimes A_j}\gl \al \Gamma^{\otimes A_j} \pl.\]
\end{prop}
\begin{proof}
Note that $T_t^{\otimes A_j}= T_t^{A_1}\otimes\cdots \otimes T_t^{A_n}$. Since $(T_t^{A_j})$ is a standard nc-diffusion semigroup for $j=1,\cdots, n$, so is $T_t^{\otimes A_j}$. We prove the $\Ga_2$-condition by induction. The case $n=2$ follows from Lemma \ref{vnte}. The general case follows by induction and repeatedly invoking Lemma \ref{tform} to deal with ``cross terms'' like $\Ga^{I\otimes  \cdots  A_i \cdots  \otimes I}\otimes\Ga^{I\otimes  \cdots  A_j \cdots  \otimes I}$.
\end{proof}
\begin{exam}[Tensor product of matrix algebras]\rm
Let $A$ be the generator of the semigroup $T_t$ acting on $M_n$ considered in Proposition \ref{mncur}. Let $\Ga$ be the gradient form associated to $\sum_{i=1}^m I\otimes \cdots  \otimes A\otimes \cdots  \otimes I$ where $A$ is in the $i$th position. Then it follows from Proposition \ref{tega} that $\Ga_2\ge \frac{2+n}{2n} \Ga $ in $\otimes_{i=1}^m M_n$.
\end{exam}
\begin{exam}[Random matrices]\rm
 Let $(\Om, \pz)$ be a probability space. Consider $I\otimes T_t$ acting on $L_\8(\Om,\pz)\otimes M_n$ where $T_t$ is the semigroup considered in Proposition \ref{mncur}. By Lemma \ref{vnte}, $I\otimes T_t$ is a standard nc-diffusion semigroup and satisfies $\Ga_2\ge \frac{2+n}{2n} \Ga $ in $L_\8(\Om,\pz)\otimes M_n$. Hence our results apply for random matrices.
\end{exam}

\begin{exam}[Product measure]\rm Here we consider $A_j=I-E_j$ for $E_j$ a conditional expectation on $\N_j$ for $j=1, \cdots , n$. By example \ref{cond}, $A_j$ generates a standard nc-diffusion semigroup and $\Ga_2^{A_j}\ge\frac12\Ga^{A_j}$. Then we deduce from Proposition \ref{tega} that $\Ga_2^A\ge \frac12\Ga^A$ for the tensor product generator $A =\otimes A_j$. For $x=x_1\otimes  \cdots  \otimes x_n$, put $\Ga_j(x,x)=x_1^*x_1\otimes \cdots  \otimes \Ga^{A_j}(x_j,x_j)\otimes  \cdots  \otimes x_n^*x_n$. Then we have
\[\Ga(x,x)\lel\sum_{j=1}^n \Ga_j(x,x)\pl. \]

We want to investigate an easy consequence of our general theory for the product measure space. Let $(\Om_i, \pz_i)$, $i=1, \cdots , n$ be a family of probability spaces and denote by $(\Om, \pz)$ the product probability space. Then $L_\8(\Om, \pz)=\otimes_{i=1}^n L_\8(\Om_i, \pz_i)$. Define $E_i(f)=\int f d\pz_i$ for $f\in L_\8(\Om, \pz)$ and put $A_i=I-E_i$. Then
 \begin{align*}
  \Ga_i(f,f)&\lel \frac12 (|f|^2 - f\int \bar f d\pz_i- \bar f\int f d\pz_i +\int |f|^2 d\pz_i)\\
  &\lel \frac12\Big(|f-\int f d\pz_i|^2+ \int \Big(|f|^2-\Big|\int f d\pz_i\Big|^2\Big) d\pz_i\Big)\pl.
 \end{align*}
 It is straightforward to check that the fixed point subalgebra of the semigroup $e^{-t(\otimes A_i)}$ is $\cz 1$. Hence $E_{\rm Fix} f = \ez f$ for $f\in L_\8(\Om, \pz)$ where $\ez$ is the expectation operator of $\pz$. Then \eqref{devi} yields
 \begin{equation}
 \begin{aligned}
   &\pz(f-\ez(f)\ge t) \kl \exp\Big(-\frac{ct^2}{\|\sum_{i=1}^n \Ga_i(f,f)\|_\8}\Big)\\
  \kl& \exp\Big(-\frac{2ct^2}{\sum_{i=1}^n\| f-\int fd\pz_i\|_\8^2+\|\int (|f|^2-|\int f d\pz_i |^2 ) d\pz_i\|_\8}\Big)\pl.
 \end{aligned}
 \end{equation}
Note that we do not impose any concrete condition on the probability spaces. This shows that the sub-Gaussian tail behavior is always true for product measures. We do not know whether such results were known before.
\end{exam}

\subsection{Free products with amalgamation}
Here we want to prove that the condition $\Gamma_2\gl\al \Gamma$ is stable under free products. Our general reference is \cite{VDN}. We need some preliminary facts about free product of semigroups $T_t= \ast_{k}T_t^{ A_k}$ acting on $\N:=\ast_{\D, k} \N_k$ with generators $A_k$ acting on von Neumann algebra $\N_k\supset \D$. Here $\D$ is a von Neumann subalgebra of all $\N_k$. Similar to the tensor products considered before, if $(T_t^{A_k})$ is a standard nc-diffusion semigroup for $k=1,\cdots,n$, so is $\ast_{k}T_t^{ A_k}$. We assume that $A_k$ commutes with the conditional expectation $E: \N_k\to \D$ for which we amalgamate and even
 \[ A_kE \lel EA_k \lel 0 \pl .\]
Our first task is to calculate the gradient $\Gamma$. For simplicity of notation, we always assume the elements we consider are chosen so that $\Ga$ and $\Ga_2$ are well-defined. Let us now consider elementary words $x=a_1 \cdots  a_m$ and $y=b_1 \cdots  b_n$ of mean $0$ elements $a_k\in \N_{i_k}$, $b_k\in \N_{j_k}$. Recall that the free product generator is given by
 \[ A(b_1 \cdots  b_n)\lel \sum_{l=1}^n b_1 \cdots  b_{l-1}A_{j_l}(b_l)b_{l+1} \cdots  b_n \pl .\]
In the future we will ignore the index for $A$. If we want to apply the free product generator $A$ on the product $x^*y$,  we have to know the mean $0$ decomposition
 \begin{align*}
 x^*y =& a_m^* \cdots  a_1^*b_1 \cdots  b_n \\
  =& \sum_{k=1}^{\min(n,m)} a_m^* \cdots  a_{k+1}^*\stackrel{\circ{  }}{\overbrace{a_k^*E(a_{k-1}^* \cdots  a_1^*b_1 \cdots  b_{k-1})b_k}}b_{k+1} \cdots  b_n \\
  &+\begin{cases}
   E(a_m^* \cdots  a_1^* b_1 \cdots  b_m)b_{m+1} \cdots  b_n, \mbox{if } m\le n,\\
   a_m^* \cdots  a_{n+1}^*E(a_n^* \cdots  a_1^* b_1 \cdots  b_m), \mbox{if } m> n.\\
  \end{cases}
  \end{align*}
Here $ \mathring{x}=x-E(x)$. Let $k_0=\inf\{i\in\nz: k_i\neq j_i\}$. The equality  $T_t(E(x))=E(x)$ implies that
 \begin{equation}\label{modu}
   A(E(x)y) \lel \lim_{t \to 0} \frac{T_t(E(x)y)-E(x)y}{t}
 \lel E(x)A(y) \pl.
 \end{equation}
It is easy to see that all terms containing $A(a_i^*)$, $A(b_i)$ for $i\ge k_0$ will cancel out in $\Ga(x,y)$ and thus
\begin{align*}
 2\Ga(x,y)\lel& \sum_{i=1}^{k_0-1} a_m^* \cdots  a_{i+1}^*A(a_i^*)a_{i-1} \cdots  a_1^*b_1 \cdots  b_n + \sum_{i=1}^{k_0-1} a_m^* \cdots   a_1^*b_1 \cdots  b_{i-1}A(b_i)b_{i+1} \cdots  b_n\\
 &- a_m^* \cdots  a_{k_0}^*A(a_{k_0-1}^* \cdots  a_1^*b_1 \cdots  b_{k_0-1})b_{k_0} \cdots   b_n\\
 \lel & 2a_m^* \cdots  a_{k_0}^* \Ga(a_1 \cdots  a_{k_0-1}, b_1 \cdots  b_{k_0-1}) b_{k_0} \cdots  b_n\pl.
\end{align*}
\begin{lemma}\label{recu}
 Let $a_i,b_i\in \N_{k_i}$ be mean $0$ elements for $i=1, \cdots , r$. Then
 \[
 \Ga(a_1 \cdots  a_r, b_1 \cdots  b_r)\lel a_{r}^*\Ga(a_1 \cdots  a_{r-1}, b_1 \cdots  b_{r-1})b_r+ \Ga(a_r,E(a_{r-1}^* \cdots  a_1^*b_1 \cdots  b_{r-1})b_r)\pl.
 \]
\end{lemma}
\begin{proof}
 Using the mean 0 decomposition, we have
 \begin{align*}
  &2\Ga(a_1 \cdots  a_r, b_1 \cdots  b_r)\lel A(a_r^*)a_{r-1}^* \cdots  b_{r-1}b_r + a_r^*A(a_{r-1}^* \cdots  a_1^*)b_1 \cdots  b_r \\
  &+a_r^* \cdots  a_1^*A(b_1 \cdots  b_{r-1})b_r+a_r^* \cdots  b_{r-1}A(b_r)\\
  &- A(a_r^*)\Big(\sum_{i=1}^{r-1} a_{r-1}^* \cdots  a_{i+1}^*\stackrel{\circ{  }}{\overbrace{a_i^* E(a_{i-1}^* \cdots  a_1^*b_1 \cdots  b_{i-1})b_i}}b_{i+1} \cdots  b_{r-1} \Big)b_r\\
  & - a_r^*\Big(\sum_{i=1}^{r-1} a_{r-1}^* \cdots  a_{i+1}^*\stackrel{\circ{  }}{\overbrace{a_i^* E(a_{i-1}^* \cdots  a_1^*b_1 \cdots  b_{i-1})b_i}}b_{i+1} \cdots  b_{r-1} \Big)A(b_r)\\
  &-a_r^*A(a_{r-1}^* \cdots  a_1^*b_1 \cdots  b_{r-1})b_r-A(a_r^*E(a_{r-1}^* \cdots  b_{r-1})b_r)\\
  \lel & 2a_r^*\Ga(a_1 \cdots  a_{r-1}, b_1 \cdots  b_{r-1})b_r + A(a_r^*)E(a_{r-1}^* \cdots  b_{r-1})b_r\\
  &+a_r^*E(a_{r-1}^* \cdots  b_{r-1})A(b_r)-A(a_r^*E(a_{r-1}^* \cdots  b_{r-1})b_r)
 \end{align*}
which completes the proof with the help of \eqref{modu}.
\end{proof}
The recursion formula immediately yields that
 \begin{lemma} Let $k_0$ be as above. Then
 \[ \Ga(x,y) \lel \sum_{k=1}^{k_0-1} \Ga^{(k)}[x,y]\pl .\]
 where
 \[ \Ga^{(k)}[x,y] \lel a_{m}^* \cdots  a_{k+1}^*\Ga[a_k,E(a_{k-1}^* \cdots  a_1^*b_1 \cdots  b_{k-1})b_k]b_{k+1} \cdots  b_n \pl .\]
\end{lemma}

In order to calculate $\Gamma_2$, we have to analyze
  \[ \Ga^{(k)}[A(x),y]+\Ga^{(k)}[x,A(y)]-A\Ga^{(k)}[x,y] \pl .\]
Observe that for $j<k$ all terms containing $A(a_j^*)$ or $A(b_j)$ appear inside the conditional expectation $E$ in $\Ga^{(k)}[A(x),y]+\Ga^{(k)}[x,A(y)]$ and there is no counterpart in $A\Ga^{(k)}$. Hence we find
\begin{align*}
 I^{(k)}(x,y) \lel&  a_{m}^* \cdots  a_{k+1}^*\Ga[a_k, E(A(a_{k-1}^* \cdots  a_1^*)b_1 \cdots  b_{k-1})b_k]b_{k+1} \cdots  b_n \\
      & + a_{m}^* \cdots  a_{k+1}^*\Ga[a_k, E(a_{k-1}^* \cdots  a_1^*A(b_1 \cdots  b_{k-1}))b_k]b_{k+1} \cdots  b_n\pl.
\end{align*}
For $k\le j<k_0$ we are left with the following terms
\begin{align*}
 II^{(k)}(x,y)\lel &\sum_{j=k}^{k_0-1} a_m^* \cdots  A(a_j^*) \cdots  a_{k+1}^*\Ga[a_k,E(a_{k-1}^* \cdots  a_1^*b_1 \cdots  b_{k-1})b_k]b_{k+1} \cdots  b_n\\
 +& \sum_{j=k}^{k_0-1} a_m^* \cdots  a_{k+1}^*\Ga[a_k,E(a_{k-1}^* \cdots  a_1^*b_1 \cdots  b_{k-1})b_k]b_{k+1} \cdots  A(b_j) \cdots  b_n\\
 -& a_m^* \cdots  A(a_{k_0-1}^* \cdots  a_{k+1}^*\Ga[a_k,E(a_{k-1}^* \cdots  a_1^*b_1 \cdots  b_{k-1})b_k]b_{k+1} \cdots  b_{k_0-1}) \cdots  b_n
\end{align*}
where we understand, when $j=k$, that $A(a_k)$ and $A(b_k)$ are inside the $\Ga$-form. Since
$$\Ga[a_k,E(a_{k-1}^* \cdots  a_1^*b_1 \cdots  b_{k-1})b_k]\in \N_{i_k},$$
we are in the situation of Lemma \ref{recu}. The recursion formula gives that
\begin{align*}
 II^{(k)}(x,y)\lel& \sum_{j=k+1}^{k_0-1} F_{jk}(x,y)+ 2a_{m}^* \cdots  a_{k+1}^*\Ga_2[a_k,E(a_{k-1}^* \cdots  a_1^*b_1 \cdots  b_{k-1})b_k]b_{k+1} \cdots  b_n
\end{align*}
where
\begin{align*}
F_{jk}(x,y)\lel& a_{m}^* \cdots  a_{j+1}^* \Ga[a_j,E(a_{j-1}^* \cdots  \Ga[a_k,\\
&E(a_{k-1}^* \cdots  a_1^*b_1 \cdots  b_{k-1})b_k] \cdots  b_{j-1})b_j]b_{j+1} \cdots  b_n\pl.
\end{align*}
Therefore we find $\Ga_2$.
\begin{lemma} Using the above notation, we have
 \begin{align*}
  2\Ga_2(x,y)\lel &\sum_{k=1}^{k_0-1} \Ga^{(k)}[A(x),y]+\Ga^{(k)}[x,A(y)]-A\Ga^{(k)}[x,y] \\
  \lel & \sum_{k=1}^{k_0-1}I^{(k)}(x,y)+II^{(k)}(x,y)
  \end{align*}
\end{lemma}

In order to show $\Ga_2\ge \al\Ga$, we need a technical lemma which is an application of the Hilbert $W^*$-module theory; see \cites{Pas,Lan95}.
\begin{lemma}\label{frpo}
 Let $\Phi: \A\times\A\to \N$ be a sesquilinear form where $\A$ is a separable *-algebra contained in the domain of $\Phi$ and $\N$ is a von Neumann algebra. Then $\Phi$ is a positive form if and only if there exists a map $v:\A\to C(\N)$ such that $\Phi(x,y)=v(x)^*v(y)$ for $x,y\in\A$ where $C(\N)=\ell_2 \otimes\N$ denotes the Hilbert $\N$-module, or the column space of $\N$.
\end{lemma}
\begin{proof}
The sufficiency is obvious. Conversely, following the KSGNS construction \cites{Pas, Lan95}, we consider the algebraic tensor product $\A\otimes \N$ and define $\lge\sum_i x_i\otimes a_i, \sum_j y_j\otimes b_j\rge = \sum_{i,j}a_i^*\Phi(x_i,y_j)b_j$ for $x_i,y_j\in \A$ and $a_i,b_j\in\N$. Set $\K = \{x\in A\otimes \N: \lge x,x\rge = 0\}$. Then $\A\otimes \N/\K$ is a pre-Hilbert $\N$-module with $\N$-valued inner product $\lge x+\K, y+\K\rge =\lge x, y\rge$ for $x,y\in \A\otimes \N$. Let $\A\otimes_\Phi \N$ be the completion of $\A\otimes \N/\K$. Then $\A\otimes_\Phi \N$ is a Hilbert $\N$-module. Since $\A$ is separable, $\A\otimes_\Phi \N$ is countably generated. It follows from \cite{Lan95}*{Theorem 6.2} that there exists a right module map $u: \A\otimes_\Phi \N \to \ell_2\otimes \N$ such that
\[ \sum_{i,j}a_i^*\Phi(x_i,y_j)b_j \lel \lge u(\sum_i x_i\otimes a_i+\K), u(\sum_j y_j\otimes b_j+\K)\rge\pl.
\]
 In particular, $\Phi(x,y)= u(x\otimes 1+\K)^*u(y\otimes 1+\K)$. Define $v(x)=u(x\otimes 1+\K)$. This is the desired map.
\end{proof}
\re
Since we only consider finitely many elements for the sake of positivity of a form in the following, the separability assumption on $\A$ in the previous lemma is automatically satisfied. If we want to remove separability, we can use the fact that every Hilbert right module over $\N$ embeds isometrically in a self-dual module. Indeed, this follows from \cite{Pas}. Let us sketch the approach from \cite{JS05}. Consider the $L_{1/2}(\N)$ module
 \[  Y= X \ten_\N L_1(\N). \]
Then the antilinear dual $Y^*$ is self-dual and obviously contains $X$ isometrically.
\mar

By an argument similar to \eqref{modu}, we have $A^{1/2}(zx)=zA^{1/2}(x)$ for $z\in \D$ and $x\in \N$. Then for $x,y\in\N$, \[\tau(zE(A(x)y))=\tau(E(A(zx)y))=\tau(A^{1/2}(zx)A^{1/2}(y))=\tau(zE(A^{1/2}(x)A^{1/2}(y))).
\]
 Hence, $E(A(x)y)=E(A^{1/2}(x)A^{1/2}(y))$ and we find
  \[ I^{(k)}(x,y) \lel 2 a_{m}^* \cdots  \Ga[a_k, E(A^{1/2}(a_{k-1}^* \cdots  a_1^*)A^{1/2}(b_1 \cdots  b_{k-1}))b_k]b_{k+1} \cdots  b_n \pl .\]
We claim that this is a positive form. Indeed, $I^{(k)}$ is nontrivial only if $a_k$ and $b_k$ are in the same $\N_{i_k}$. Using Lemma \ref{frpo} with $\Phi=\Ga$, we find $\bt_k:\N_{i_k}\to C(\N_{i_k})$ such that $\Ga(a,b)=\bt_k(a)^*\bt_k(b)$ for $a,b\in\N_{i_k}$.  Similarly with $\Phi(x,y)=E(x^*y)$, we find $v_k:\N\to C(\D)$ such that $E(x^*y)=v_k(x)^*v_k(y)$ for $x,y\in\N$. Define
\[ u_k(b_1 \cdots  b_n)=e_{i_1, \cdots ,i_k}\otimes(\bt_k\otimes Id)(v_k(A^{1/2}(b_1 \cdots  b_{k-1}))b_k)b_{k+1} \cdots  b_n\pl.\]
Note that by the module property \eqref{modu} $\Ga(z^*a,b)=\Ga(a,zb)$ for $z\in\D, x,y\in \N_{i_k}$. Write $v_k(x)=(v_k^l(x))_l$ where $v_k^l(x)\in \D$. It follows that
\begin{align*}
 I^{(k)}(x,y)&\lel 2 a_m^* \cdots  a_{k+1}^*\Ga(a_k,~\sum_l v_k^l[A^{1/2}(a_1 \cdots  a_{k-1})]^*v_k^l[A^{1/2}(b_1 \cdots  b_{k-1})] b_k)
 b_{k+1} \cdots  b_n\\
 &\lel 2 \sum_l a_m^* \cdots  a_{k+1}^*\Ga(v_k^l[A^{1/2}(a_1 \cdots  a_{k-1})]a_k, ~v_k^l[A^{1/2}(b_1 \cdots  b_{k-1})] b_k)b_{k+1} \cdots  b_n\\
 &\lel 2\sum_l a_m^* \cdots  a_{k+1}^*\bt_k(v_k^l[A^{1/2}(a_1 \cdots  a_{k-1})]a_k)^*\bt_k(v_k^l[A^{1/2}(b_1 \cdots  b_{k-1})]b_k)b_{k+1} \cdots  b_n\\
 &\lel 2u_k(a_1 \cdots  a_m)^*u_k(b_1 \cdots  b_n)\pl.
\end{align*}
By Lemma \ref{frpo}, $I^{(k)}$ is a positive form.

Now we claim that $F_{jk}$ are positive forms for $j=k+1, \cdots , k_0-1$.
Indeed, define
\begin{align*}
 u_{jk}(b_1 \cdots  b_n)\lel& e_{i_{k+1}, \cdots ,i_j}\otimes (\bt_j\otimes Id)((v_j\otimes Id )[e_{i_1, \cdots ,i_k}\\
 &\otimes (\bt_k\otimes Id)[v_k(b_1 \cdots  b_{k-1})b_k]b_{k+1} \cdots  b_{j-1}]b_j)b_{j+1} \cdots  b_n \pl.
\end{align*}
Then similar to the argument for $I^{(k)}$, we find $F_{jk}(x,y)=u_{jk}(a_1 \cdots  a_m)^*u_{jk}(b_1 \cdots  b_n)$. By Lemma \ref{frpo}, $F_{jk}$ is a positive form. Hence, we find
\[ II^{(k)}(x,y)\ge 2a_{m}^* \cdots  a_{k+1}^*\Ga_2[a_k,E(a_{k-1}^* \cdots  a_1^*b_1 \cdots  b_{k-1})b_k]b_{k+1} \cdots  b_n\pl.\]
Therefore we deduce the main result

\begin{prop}\label{free} Let $A_j$ be self-adjoint generators of standard nc-diffusion semigroups $(T_t^{A_j})$ and $\Gamma_{A_j}^2\gl \al \Gamma_{A_j}$ respectively for $j=1,..,n$ with the same constant $\al>0$. Then the free product generator $\ast A_j(a_1 \cdots  a_n)=\sum_{j} a_1 \cdots  a_{j-1}A(a_j)a_{j+1} \cdots  a_n$ generates a standard nc-diffusion semigroup $(T_t^{\ast A_j})$ with
 \[ \Gamma^2_{\ast A_j}\gl \al \Gamma_{\ast A_j} \pl.\]
\end{prop}
\begin{exam}\rm
 The free product of all the examples considered so far satisfies the $\Ga_2$-criterion. In particular, the free product of matrix algebra $\ast_i M_n$ admits a standard nc-diffusion semigroup with the $\Ga_2$-criterion.
\end{exam}

\begin{exam}[Block length function] \rm  Consider the free product of groups $G_i$, $G=\ast_{i\in I} G_i$ with the block length function $\psi$, i.e.
$$\psi(g_1^{k_1} \cdots  g_n^{k_n})=n$$
for $g_1\in G_{i_1}, \cdots ,g_n\in G_{i_n}$, $i_1\neq i_2\neq \cdots  \neq i_n$ and $k_i\in\zz$. Fix $i$ and denote by $\la$ the left regular representation of $G_i$. Define the conditional expectation $E: \Lc(G_i)\to \cz 1$ to be
 \[
  E(\la(g))\lel \tau(\la(g))1\lel \begin{cases}
                 1, \mbox{ if } g=e,\\
                 0, \mbox{ if } g\neq e.
                \end{cases}
 \]
 Here $e$ is the identity element of $G_i$ and $1$ is the identity element of $\Lc(G_i)$. Example \ref{cond} says that $T_t \la(g)=e^{-t(I-E)}\la(g)$ is a standard nc-diffusion semigroup with $\Ga_2\ge \frac12\Ga$ where
 \[
  T_t\la(g)\lel \begin{cases}
                 \la(g), \mbox{ if } g = e,\\
                 e^{-t} \la(g), \mbox{ if } g\neq e.
                \end{cases}
 \]
 Since $\Lc(G)=\ast_{i\in I}\Lc(G_i)$, using Proposition \ref{free} and the relation $\la(g_1 \cdots  g_n)=\la(g_1) \cdots \la(g_n)$ for $g_1\in G_{i_1}, \cdots ,g_n\in G_{i_n}$ and $i_1\neq i_2\neq \cdots  \neq i_n$, we deduce that $(T_t^b)$ is a standard nc-diffusion semigroup acting on $\Lc(G)$ with $\Ga_2\ge\frac12 \Ga$ where
 \[
 T_t^b(\la(g_1^{k_1} \cdots  g_n^{k_n})) = e^{-tn}\la(g_1^{k_1} \cdots  g_n^{k_n}). %\quad g_1^{k_1} \cdots  g_n^{k_n}\in\fz_I \mbox{ freely reduced.}
      \]
 Clearly, the infinitesimal generator of $T_t^b$ is the block length function. In particular, for $G_i=\zz$ we find a standard nc-diffusion semigroup acting on $\Lc(\fz_n)$ which is different from the one considered in Section \ref{frgp}. In fact, our result applies even for free product of groups with amalgamation in general.
\end{exam}

\section{The classical diffusion processes}
We consider classical diffusion semigroups in this section. As explained in Theorem \ref{rbes}, we have stronger results in the this setting thanks to the better constant in the commutative BDG inequality.
\subsection{Ornstein--Uhlenbeck process in $\rz^d$}
Let us start with Ornstein-Uhlenbeck process whose infinitesimal generator is $-A=\Delta-x\cdot \nabla$ in $\rz^d$. We refer the readers to e.g. \cite{Le} for the facts we state in this subsection. Let $T_t=e^{-tA}$ be the semigroup generated by $A$ and $\ga$ denote the canonical  Gaussian measure on $\rz^d$ with density $(2\pi)^{-d/2}e^{-|x|^2/2}$. It is well known that $\ga$ is an invariant measure of $T_t$ and
\[T_tf(x)=\int_{\rz^d}f(e^{-t}x+(1-e^{-2t})^{1/2}y)d\ga(y).\]
Let $\A=C_c^\8(\rz^d)$, the compactly supported smooth functions. Clearly $\A$ is weakly dense in $\N=L_\8(\rz^d,\ga)$ and $T_t$ is self-adjoint with respect to $\ga$. Clearly $\A$ is dense in $\Dom(A^{1/2})$ in the graph norm. Note that $\Ga(f,f)=|\nabla f|^2$ and that for $f\in\Dom(A^{1/2})$
\[
\|\Ga(f,f)\|_1\lel \lge A^{1/2}f,A^{1/2} f\rge_{L_2(\rz^d,\ga)}\pl.
\]
Therefore $(T_t)$ is a standard nc-diffusion semigroup satisfying the assumptions in Lemma \ref{ga22}. It is easy to check that
\[\Ga_2(f,f)=|\nabla f|^2+\|\Hess f\|_{HS}^2\ge\Ga(f,f),\quad f\in C_c^\8(\rz^d).\]
Here $\Hess f$ denotes the Hessian of $f$ and $\|\cdot\|_{HS}$ denotes the Hilbert--Schmidt norm.
Note that $Af=0$ only if $f$ is a constant. Thus the fixed point algebra is trivial.
Theorem \ref{poin} with $\al=1$ and Theorem \ref{rbes} immediately lead to the following result.
\begin{cor}
Let $2\le p<\8$. Then there exist a constant $C$ such that for all real valued functions $f\in W^{1,p}(\rz^d,\ga)$
\begin{equation}\label{oupc}
\Big\|f-\int f d\ga\Big\|_{L_p(\rz^d,\ga)}\kl C \sqrt{p} \|\pl |\nabla f|\pl \|_{L_p(\rz^d,\ga)}
\end{equation}
where $W^{1,p}(\rz^d,\ga)$ denotes the Sobolev space consisting of all $L_p(\rz^d,\ga)$ functions with first order weak derivatives also in $L_p(\rz^d,\ga)$.
\end{cor}
This result can be generalized to infinite dimension. Let $(W,H,\mu)$ be an abstract Wiener space and $L$ the Ornstein--Uhlenbeck operator on $W$. Then it can be checked that the gradient form associated with $L$ satisfies
\[\Ga_2(F,F)\lel (\nabla F, \nabla F)+\|\nabla^2 F\|^2_{HS}\gl \Ga(F,F)\]
for $F(w)\in {\rm Cylin}(W)$, the cylindrical functions on $W$. Based on standard facts from Malliavin calculus, an argument similar to the $\rz^d$ case shows that the Ornstein--Uhlenbeck semigroup $T_t$ is a standard nc-diffusion semigroup satisfying the assumptions in Lemma \ref{ga22}. Moreover, the fixed point algebra ${\rm Fix}$ is trivial. See \cites{Fan, Nua} for more details.
Hence our Poincar\'e type inequality \eqref{oupc} holds in this setting.
\subsection{Diffusion processes on Riemannian manifolds}
Consider an elliptic differential operator $-A$ on a connected smooth manifold $M$ of dimension $d$ with invariant probability measure $\mu$ on Borel sets which is equivalent to Lebesgue measure. We can write it in a local coordinate chart as
\[-Af(x)=\sum_{i,j}g^{ij}\frac{\partial^2f}{\partial x^i\partial x^j}(x) +\sum_ib^i \frac{\partial f}{\partial x^i}(x)\]
where $g^{ij}$ and $b^i$ are smooth functions and $(g^{ij})$ is a nonnegative definite matrix. The inverse of $(g^{ij})$ then defines a Riemannian metric. It can be checked that
$$\Ga(f,h)=\sum_{ij} g^{ij}\frac{\partial f}{\partial x^i}\frac{\partial h}{\partial x^j}$$
for all $f,h\in C_c^\8(M)$.
To give an example, we take $-A=\Delta+ Z$ where $\Delta$ is the Laplace--Beltrami operator on a stochastically complete Riemannian manifold and $Z$ is a $C^1$-vector field $Z$ on a Riemannian manifold $M$ such that
\begin{equation}\label{curcon}
{\rm Ric}(X,X)-\langle\nabla_X Z, X\rangle\ge \al |X|^2, \quad X\in TM
\end{equation}
for some $\al>0$. By the Bochner identity, this inequality is equivalent to (see e.g. \cite{Wa})
\[\Ga_2(f,f)\ge \al \Ga(f,f), \quad f\in C^\8(M).\]
Take $\A=C_c^\8(M)$ and $T_t=e^{-tA}$. The following result follows from Theorem \ref{rbes} and the martingale problem on differentiable manifolds \cite{Hsu}.
\begin{cor}\label{pman}
Assume \eqref{curcon} and the following conditions
\begin{enumerate}
 \item
$\int T_t(f)  g d\mu=\int f T_t (g) d\mu$ (i.e. $T_t$ is symmetric);
%\item $\lim_{t\to\8}\|T_t f-\int f d \mu\|_{L_p(M,\mu)} = 0$ (i.e. $T_t$ is ergodic in $L_p$);
\item $|\nabla f|\in L_2(M, \mu)$ whenever $\lge A^{1/2} f,A^{1/2} f\rge <\8$.
\end{enumerate}
Then for all $2\le p<\8$ and real valued functions $f\in W^{1,p}(M,\mu)$,
\begin{equation}\label{lipo}
\|f-E_{\rm Fix}f\|_{L_p(M,\mu)}\kl C \sqrt{p/\al} \|~ |\nabla f|~ \|_{L_p(M,\mu)}
\end{equation}
where $W^{1,p}(M,\mu)$ is the Sobolev space on the Riemannian manifold $M$.
 \end{cor}
\begin{rem}
  Since compactly supported smooth functions are dense in the graph norm of $A^{1/2}$, the work of Cipriani and Sauvageot \cite{CSa} shows that Condition (2) is automatically satisfied.
\end{rem}
Functional inequalities related to diffusion processes on Riemannian manifolds have been studied extensively; see \cite{Wan05} for more details on this subject. To give an even more concrete example, let $\nu$ be the normalized volume measure and $\mu(dx)=e^{-V(x)}\nu(dx)$ a probability measure for $V\in C^2(M)$. Suppose \eqref{curcon} holds. It is clear that the semigroup $T_t$ with generator $-A=\Delta-\nabla V\cdot \nabla$ fulfills the assumptions of Corollary \ref{pman} and the fixed point algebra is trivial. It follows that
\[
\|f-\int f d\mu\|_{L_p(M,\mu)}\kl C \sqrt{p/\al} \|~ |\nabla f|~ \|_{L_p(M,\mu)}\pl.
\]
This improves X.-D. Li's result \cite{Li08}*{Theorem 1.2, Theorem 5.2} for $p\ge2$ which was proved by using his sharp estimate of the $L_p$-norm of Riesz transform. Indeed, his Poincar\'e inequality has constant $p/\sqrt{\al}$.
\re
\eqref{lipo} is true only for scalar-valued functions. If one is interested in some noncommutative objects, e.g., matrix-valued functions on manifolds or free product of manifolds, one has to apply the noncommutative theory and then the Poincar\'e inequalities are in the form of Theorem \ref{poin}. Of course, the deviation and the transportation inequalities still hold in all those situations.
\mar

\section*{Acknowledgements}
We would like to thank Joel Tropp for bringing Oliveira's paper \cite{Ol} to our attention. Q. Z. is grateful to Stephen Avsec, Yi Hu, Huaiqian Li, Jian Liang, Tao Mei, Sepideh Rezvani, and Tianyi Zheng for helpful conversations on various examples. We sincerely thank the referee for very careful reading and many constructive suggestions, which have improved the paper.

\bibliographystyle{plain}
\bibliography{fdm_ref}

\begin{bibdiv}
\begin{biblist}

\bib{AW13}{article}{
      author={{Adamczak}, R.},
      author={{Wolff}, P.},
       title={{Concentration inequalities for non-Lipschitz functions with
  bounded derivatives of higher order}},
        date={2013-04},
     journal={ArXiv e-prints},
      eprint={1304.1826},
}

\bib{AS94}{article}{
      author={Aida, S.},
      author={Stroock, D.},
       title={Moment estimates derived from {P}oincar\'e and logarithmic
  {S}obolev inequalities},
        date={1994},
        ISSN={1073-2780},
     journal={Math. Res. Lett.},
      volume={1},
      number={1},
       pages={75\ndash 86},
         url={http://dx.doi.org/10.4310/MRL.1994.v1.n1.a9},
      review={\MR{1258492 (95f:60086)}},
}

\bib{Ar}{article}{
      author={Araki, Huzihiro},
       title={Golden-{T}hompson and {P}eierls-{B}ogolubov inequalities for a
  general von {N}eumann algebra},
        date={1973},
        ISSN={0010-3616},
     journal={Comm. Math. Phys.},
      volume={34},
       pages={167\ndash 178},
      review={\MR{0341114 (49 \#5864)}},
}

\bib{AJ}{article}{
      author={Avsec, Stephen},
      author={Junge, Marius},
       title={q-{G}aussian random variables and co-action},
        date={2012},
     journal={Preprint},
}

\bib{BE85}{incollection}{
      author={Bakry, D.},
      author={{\'E}mery, Michel},
       title={Diffusions hypercontractives},
        date={1985},
   booktitle={S\'eminaire de probabilit\'es, {XIX}, 1983/84},
      series={Lecture Notes in Math.},
      volume={1123},
   publisher={Springer},
     address={Berlin},
       pages={177\ndash 206},
         url={http://dx.doi.org/10.1007/BFb0075847},
      review={\MR{889476 (88j:60131)}},
}

\bib{BY82}{article}{
      author={Barlow, M.~T.},
      author={Yor, M.},
       title={Semimartingale inequalities via the {G}arsia-{R}odemich-{R}umsey
  lemma, and applications to local times},
        date={1982},
        ISSN={0022-1236},
     journal={J. Funct. Anal.},
      volume={49},
      number={2},
       pages={198\ndash 229},
         url={http://dx.doi.org/10.1016/0022-1236(82)90080-5},
      review={\MR{680660 (84f:60073)}},
}

\bib{BV}{article}{
      author={Biane, P.},
      author={Voiculescu, D.},
       title={A free probability analogue of the {W}asserstein metric on the
  trace-state space},
        date={2001},
        ISSN={1016-443X},
     journal={Geom. Funct. Anal.},
      volume={11},
      number={6},
       pages={1125\ndash 1138},
         url={http://dx.doi.org/10.1007/s00039-001-8226-4},
      review={\MR{1878316 (2003d:46087)}},
}

\bib{BG}{article}{
      author={Bobkov, S.~G.},
      author={G{\"o}tze, F.},
       title={Exponential integrability and transportation cost related to
  logarithmic {S}obolev inequalities},
        date={1999},
        ISSN={0022-1236},
     journal={J. Funct. Anal.},
      volume={163},
      number={1},
       pages={1\ndash 28},
         url={http://dx.doi.org/10.1006/jfan.1998.3326},
      review={\MR{1682772 (2000b:46059)}},
}

\bib{BKS}{article}{
      author={Bo{\.z}ejko, Marek},
      author={K{\"u}mmerer, Burkhard},
      author={Speicher, Roland},
       title={{$q$}-{G}aussian processes: non-commutative and classical
  aspects},
        date={1997},
        ISSN={0010-3616},
     journal={Comm. Math. Phys.},
      volume={185},
      number={1},
       pages={129\ndash 154},
         url={http://dx.doi.org/10.1007/s002200050084},
      review={\MR{1463036 (98h:81053)}},
}

\bib{BO}{book}{
      author={Brown, Nathanial~P.},
      author={Ozawa, Narutaka},
       title={{$C^*$}-algebras and finite-dimensional approximations},
      series={Graduate Studies in Mathematics},
   publisher={American Mathematical Society},
     address={Providence, RI},
        date={2008},
      volume={88},
        ISBN={978-0-8218-4381-9; 0-8218-4381-8},
      review={\MR{2391387 (2009h:46101)}},
}

\bib{Ch74}{article}{
      author={Choi, Man~Duen},
       title={A {S}chwarz inequality for positive linear maps on {$C^{\ast} \
  $}-algebras},
        date={1974},
        ISSN={0019-2082},
     journal={Illinois J. Math.},
      volume={18},
       pages={565\ndash 574},
      review={\MR{0355615 (50 \#8089)}},
}

\bib{CSa}{article}{
      author={Cipriani, Fabio},
      author={Sauvageot, Jean-Luc},
       title={Derivations as square roots of {D}irichlet forms},
        date={2003},
        ISSN={0022-1236},
     journal={J. Funct. Anal.},
      volume={201},
      number={1},
       pages={78\ndash 120},
         url={http://dx.doi.org/10.1016/S0022-1236(03)00085-5},
      review={\MR{1986156 (2004e:46080)}},
}

\bib{CJ}{article}{
      author={Collins, B.},
      author={Junge, M.},
       title={What is a noncommutative {B}rownian motion?},
        date={2012},
     journal={Preprint},
}

\bib{Con}{book}{
      author={Connes, Alain},
       title={Noncommutative geometry},
   publisher={Academic Press Inc.},
     address={San Diego, CA},
        date={1994},
        ISBN={0-12-185860-X},
      review={\MR{1303779 (95j:46063)}},
}

\bib{Dav}{book}{
      author={Davidson, Kenneth~R.},
       title={{$C^*$}-algebras by example},
      series={Fields Institute Monographs},
   publisher={American Mathematical Society},
     address={Providence, RI},
        date={1996},
      volume={6},
        ISBN={0-8218-0599-1},
      review={\MR{1402012 (97i:46095)}},
}

\bib{DL}{article}{
      author={Davies, E.~Brian},
      author={Lindsay, J.~Martin},
       title={Noncommutative symmetric {M}arkov semigroups},
        date={1992},
        ISSN={0025-5874},
     journal={Math. Z.},
      volume={210},
      number={3},
       pages={379\ndash 411},
         url={http://dx.doi.org/10.1007/BF02571804},
      review={\MR{1171180 (93f:46088)}},
}

\bib{DZ}{book}{
      author={Dembo, Amir},
      author={Zeitouni, Ofer},
       title={Large deviations techniques and applications},
     edition={Second},
      series={Applications of Mathematics (New York)},
   publisher={Springer-Verlag},
     address={New York},
        date={1998},
      volume={38},
        ISBN={0-387-98406-2},
      review={\MR{1619036 (99d:60030)}},
}

\bib{ELP}{article}{
      author={Efraim, L.~Ben},
      author={Lust-Piquard, F.},
       title={Poincar\'e type inequalities on the discrete cube and in the
  {CAR} algebra},
        date={2008},
        ISSN={0178-8051},
     journal={Probab. Theory Related Fields},
      volume={141},
      number={3-4},
       pages={569\ndash 602},
         url={http://dx.doi.org/10.1007/s00440-007-0094-x},
      review={\MR{2391165 (2009c:46085)}},
}

\bib{FK}{article}{
      author={Fack, Thierry},
      author={Kosaki, Hideki},
       title={Generalized {$s$}-numbers of {$\tau$}-measurable operators},
        date={1986},
        ISSN={0030-8730},
     journal={Pacific J. Math.},
      volume={123},
      number={2},
       pages={269\ndash 300},
         url={http://projecteuclid.org/getRecord?id=euclid.pjm/1102701004},
      review={\MR{840845 (87h:46122)}},
}

\bib{Fan}{book}{
      author={Fang, Shizan},
       title={Introduction to {M}alliavin calculus},
      series={Mathematics Series for Graduate Student},
   publisher={Tsinghua University Press and Springer},
     address={Beijing},
        date={2005},
      volume={3},
        ISBN={7-302-09181-1},
}

\bib{Gui}{book}{
      author={Guionnet, Alice},
       title={Large random matrices: lectures on macroscopic asymptotics},
      series={Lecture Notes in Mathematics},
   publisher={Springer-Verlag},
     address={Berlin},
        date={2009},
      volume={1957},
        ISBN={978-3-540-69896-8},
         url={http://dx.doi.org/10.1007/978-3-540-69897-5},
        note={Lectures from the 36th Probability Summer School held in
  Saint-Flour, 2006},
      review={\MR{2498298 (2010d:60018)}},
}

\bib{Gun}{article}{
      author={Gundy, Richard~F.},
       title={Sur les transformations de {R}iesz pour le semi-groupe
  d'{O}rnstein-{U}hlenbeck},
        date={1986},
        ISSN={0249-6291},
     journal={C. R. Acad. Sci. Paris S\'er. I Math.},
      volume={303},
      number={19},
       pages={967\ndash 970},
      review={\MR{877182 (88c:60108)}},
}

\bib{Haa}{article}{
      author={Haagerup, Uffe},
       title={An example of a nonnuclear {$C^{\ast} $}-algebra, which has the
  metric approximation property},
        date={1978/79},
        ISSN={0020-9910},
     journal={Invent. Math.},
      volume={50},
      number={3},
       pages={279\ndash 293},
         url={http://dx.doi.org/10.1007/BF01410082},
      review={\MR{520930 (80j:46094)}},
}

\bib{Hsu}{book}{
      author={Hsu, Elton~P.},
       title={Stochastic analysis on manifolds},
      series={Graduate Studies in Mathematics},
   publisher={American Mathematical Society},
     address={Providence, RI},
        date={2002},
      volume={38},
        ISBN={0-8218-0802-8},
      review={\MR{1882015 (2003c:58026)}},
}

\bib{Ol}{article}{
      author={{Imbuzeiro Oliveira}, R.},
       title={{Concentration of the adjacency matrix and of the Laplacian in
  random graphs with independent edges}},
        date={2009-11},
     journal={ArXiv e-prints},
      eprint={0911.0600},
}

\bib{JKo}{article}{
      author={Junge, M.},
      author={K\"{o}stler, C.},
      author={Perrin, M.},
      author={Xu, Q.},
       title={Stochastic integrals in noncommutative {$L_p$} spaces},
     journal={Preprint},
}

\bib{JMP}{article}{
      author={{Junge}, M.},
      author={{Mei}, T.},
      author={{Parcet}, J.},
       title={{Smooth Fourier multipliers on group von Neumann algebras}},
        date={2010-10},
     journal={ArXiv e-prints},
      eprint={1010.5320},
}

\bib{JP}{article}{
      author={{Junge}, M.},
      author={{Perrin}, M.},
       title={{Theory of H\_p-spaces for continuous filtrations in von Neumann
  algebras}},
        date={2013-01},
     journal={ArXiv e-prints},
      eprint={1301.2071},
}

\bib{Ju}{article}{
      author={Junge, Marius},
       title={Doob's inequality for non-commutative martingales},
        date={2002},
        ISSN={0075-4102},
     journal={J. Reine Angew. Math.},
      volume={549},
       pages={149\ndash 190},
         url={http://dx.doi.org/10.1515/crll.2002.061},
      review={\MR{1916654 (2003k:46097)}},
}

\bib{JMe}{article}{
      author={Junge, Marius},
      author={Mei, Tao},
       title={Noncommutative {R}iesz transforms---a probabilistic approach},
        date={2010},
        ISSN={0002-9327},
     journal={Amer. J. Math.},
      volume={132},
      number={3},
       pages={611\ndash 680},
         url={http://dx.doi.org/10.1353/ajm.0.0122},
      review={\MR{2666903 (2012b:46136)}},
}

\bib{JMe2}{article}{
      author={Junge, Marius},
      author={Mei, Tao},
       title={{BMO spaces associated with semigroups of operators}},
        date={to appear},
     journal={Math. Ann.},
      eprint={arXiv 1111.6194},
}

\bib{JRS}{article}{
      author={{Junge}, Marius},
      author={{Ricard}, E.},
      author={Shlyakhtenko, D.},
       title={Noncommutative diffusion semigroups and free probability},
        date={2010},
     journal={Preprint},
}

\bib{JS05}{article}{
      author={Junge, Marius},
      author={Sherman, David},
       title={Noncommutative {$L^p$} modules},
        date={2005},
        ISSN={0379-4024},
     journal={J. Operator Theory},
      volume={53},
      number={1},
       pages={3\ndash 34},
      review={\MR{2132686 (2006a:46077)}},
}

\bib{JX03}{article}{
      author={Junge, Marius},
      author={Xu, Quanhua},
       title={Noncommutative {B}urkholder/{R}osenthal inequalities},
        date={2003},
        ISSN={0091-1798},
     journal={Ann. Probab.},
      volume={31},
      number={2},
       pages={948\ndash 995},
         url={http://dx.doi.org/10.1214/aop/1048516542},
      review={\MR{1964955 (2004f:46078)}},
}

\bib{JX07}{article}{
      author={Junge, Marius},
      author={Xu, Quanhua},
       title={Noncommutative maximal ergodic theorems},
        date={2007},
        ISSN={0894-0347},
     journal={J. Amer. Math. Soc.},
      volume={20},
      number={2},
       pages={385\ndash 439},
         url={http://dx.doi.org/10.1090/S0894-0347-06-00533-9},
      review={\MR{2276775 (2007k:46109)}},
}

\bib{JZ}{article}{
      author={{Junge}, Marius},
      author={{Zeng}, Qiang},
       title={{Noncommutative Bennett and Rosenthal Inequalities}},
        date={to appear},
     journal={Ann. Probab.},
      eprint={arXiv 1111.1027},
}

\bib{Lan}{article}{
      author={Lance, Christopher},
       title={On nuclear {$C^{\ast} $}-algebras},
        date={1973},
     journal={J. Functional Analysis},
      volume={12},
       pages={157\ndash 176},
      review={\MR{0344901 (49 \#9640)}},
}

\bib{Lan95}{book}{
      author={Lance, E.~C.},
       title={Hilbert {$C^*$}-modules},
      series={London Mathematical Society Lecture Note Series},
   publisher={Cambridge University Press},
     address={Cambridge},
        date={1995},
      volume={210},
        ISBN={0-521-47910-X},
         url={http://dx.doi.org/10.1017/CBO9780511526206},
        note={A toolkit for operator algebraists},
      review={\MR{1325694 (96k:46100)}},
}

\bib{Le}{article}{
      author={Ledoux, Michel},
       title={The geometry of {M}arkov diffusion generators},
        date={2000},
        ISSN={0240-2963},
     journal={Ann. Fac. Sci. Toulouse Math. (6)},
      volume={9},
      number={2},
       pages={305\ndash 366},
         url={http://www.numdam.org/item?id=AFST_2000_6_9_2_305_0},
        note={Probability theory},
      review={\MR{1813804 (2002a:58045)}},
}

\bib{Li08}{article}{
      author={Li, Xiang-Dong},
       title={Martingale transforms and {$L^p$}-norm estimates of {R}iesz
  transforms on complete {R}iemannian manifolds},
        date={2008},
        ISSN={0178-8051},
     journal={Probab. Theory Related Fields},
      volume={141},
      number={1-2},
       pages={247\ndash 281},
         url={http://dx.doi.org/10.1007/s00440-007-0085-y},
      review={\MR{2372971 (2008k:58080)}},
}

\bib{Luc}{article}{
      author={{\L}uczak, Andrzej},
       title={Continuity of a quantum stochastic process},
        date={2008},
        ISSN={0020-7748},
     journal={Internat. J. Theoret. Phys.},
      volume={47},
      number={1},
       pages={236\ndash 244},
         url={http://dx.doi.org/10.1007/s10773-007-9439-6},
      review={\MR{2377050 (2008j:46048)}},
}

\bib{MS10}{article}{
      author={Markiewicz, Daniel},
      author={Shalit, Orr~Moshe},
       title={Continuity of {CP}-semigroups in the point-strong operator
  topology},
        date={2010},
        ISSN={0379-4024},
     journal={J. Operator Theory},
      volume={64},
      number={1},
       pages={149\ndash 154},
      review={\MR{2669432 (2011j:46115)}},
}

\bib{Nua}{book}{
      author={Nualart, David},
       title={The {M}alliavin calculus and related topics},
     edition={Second},
      series={Probability and its Applications (New York)},
   publisher={Springer-Verlag},
     address={Berlin},
        date={2006},
        ISBN={978-3-540-28328-7; 3-540-28328-5},
      review={\MR{2200233 (2006j:60004)}},
}

\bib{OR}{article}{
      author={Ozawa, Narutaka},
      author={Rieffel, Marc~A.},
       title={Hyperbolic group {$C^*$}-algebras and free-product
  {$C^*$}-algebras as compact quantum metric spaces},
        date={2005},
        ISSN={0008-414X},
     journal={Canad. J. Math.},
      volume={57},
      number={5},
       pages={1056\ndash 1079},
         url={http://dx.doi.org/10.4153/CJM-2005-040-0},
      review={\MR{2164594 (2006e:46080)}},
}

\bib{Pas}{article}{
      author={Paschke, William~L.},
       title={Inner product modules over {$B^{\ast} $}-algebras},
        date={1973},
        ISSN={0002-9947},
     journal={Trans. Amer. Math. Soc.},
      volume={182},
       pages={443\ndash 468},
      review={\MR{0355613 (50 \#8087)}},
}

\bib{Pau}{book}{
      author={Paulsen, Vern},
       title={Completely bounded maps and operator algebras},
      series={Cambridge Studies in Advanced Mathematics},
   publisher={Cambridge University Press},
     address={Cambridge},
        date={2002},
      volume={78},
        ISBN={0-521-81669-6},
      review={\MR{1976867 (2004c:46118)}},
}

\bib{Pet}{article}{
      author={Peterson, Jesse},
       title={A 1-cohomology characterization of property ({T}) in von
  {N}eumann algebras},
        date={2009},
        ISSN={0030-8730},
     journal={Pacific J. Math.},
      volume={243},
      number={1},
       pages={181\ndash 199},
         url={http://dx.doi.org/10.2140/pjm.2009.243.181},
      review={\MR{2550142 (2010f:22009)}},
}

\bib{Pin}{article}{
      author={Pinelis, Iosif},
       title={Optimum bounds for the distributions of martingales in {B}anach
  spaces},
        date={1994},
        ISSN={0091-1798},
     journal={Ann. Probab.},
      volume={22},
      number={4},
       pages={1679\ndash 1706},
  url={http://links.jstor.org/sici?sici=0091-1798(199410)22:4<1679:OBFTDO>2.0.CO;2-2&origin=MSN},
      review={\MR{1331198 (96b:60010)}},
}

\bib{Pis}{incollection}{
      author={Pisier, Gilles},
       title={Riesz transforms: a simpler analytic proof of {P}.-{A}. {M}eyer's
  inequality},
        date={1988},
   booktitle={S\'eminaire de {P}robabilit\'es, {XXII}},
      series={Lecture Notes in Math.},
      volume={1321},
   publisher={Springer},
     address={Berlin},
       pages={485\ndash 501},
         url={http://dx.doi.org/10.1007/BFb0084154},
      review={\MR{960544 (89m:60178)}},
}

\bib{Ran}{article}{
      author={Randrianantoanina, Narcisse},
       title={Conditioned square functions for noncommutative martingales},
        date={2007},
        ISSN={0091-1798},
     journal={Ann. Probab.},
      volume={35},
      number={3},
       pages={1039\ndash 1070},
         url={http://dx.doi.org/10.1214/009117906000000656},
      review={\MR{2319715 (2009d:46112)}},
}

\bib{RY99}{book}{
      author={Revuz, Daniel},
      author={Yor, Marc},
       title={Continuous martingales and {B}rownian motion},
     edition={Third},
      series={Grundlehren der Mathematischen Wissenschaften [Fundamental
  Principles of Mathematical Sciences]},
   publisher={Springer-Verlag},
     address={Berlin},
        date={1999},
      volume={293},
        ISBN={3-540-64325-7},
      review={\MR{1725357 (2000h:60050)}},
}

\bib{Rie}{incollection}{
      author={Rieffel, Marc~A.},
       title={Compact quantum metric spaces},
        date={2004},
   booktitle={Operator algebras, quantization, and noncommutative geometry},
      series={Contemp. Math.},
      volume={365},
   publisher={Amer. Math. Soc.},
     address={Providence, RI},
       pages={315\ndash 330},
         url={http://dx.doi.org/10.1090/conm/365/06709},
      review={\MR{2106826 (2005h:46099)}},
}

\bib{Ru}{article}{
      author={Ruskai, M.~B.},
       title={Inequalities for traces on von {N}eumann algebras},
        date={1972},
        ISSN={0010-3616},
     journal={Comm. Math. Phys.},
      volume={26},
       pages={280\ndash 289},
      review={\MR{0312284 (47 \#846)}},
}

\bib{Sau}{incollection}{
      author={Sauvageot, J.-L.},
       title={Tangent bimodule and locality for dissipative operators on
  {$C^*$}-algebras},
        date={1989},
   booktitle={Quantum probability and applications, {IV} ({R}ome, 1987)},
      series={Lecture Notes in Math.},
      volume={1396},
   publisher={Springer},
     address={Berlin},
       pages={322\ndash 338},
         url={http://dx.doi.org/10.1007/BFb0083561},
      review={\MR{1019584 (90k:46145)}},
}

\bib{Sp92}{article}{
      author={Speicher, Roland},
       title={A noncommutative central limit theorem},
        date={1992},
        ISSN={0025-5874},
     journal={Math. Z.},
      volume={209},
      number={1},
       pages={55\ndash 66},
         url={http://dx.doi.org/10.1007/BF02570820},
      review={\MR{1143213 (93a:46125)}},
}

\bib{Tak}{book}{
      author={Takesaki, M.},
       title={Theory of operator algebras. {I}},
      series={Encyclopaedia of Mathematical Sciences},
   publisher={Springer-Verlag},
     address={Berlin},
        date={2002},
      volume={124},
        ISBN={3-540-42248-X},
        note={Reprint of the first (1979) edition, Operator Algebras and
  Non-commutative Geometry, 5},
      review={\MR{1873025 (2002m:46083)}},
}

\bib{Tro}{article}{
      author={Tropp, Joel~A.},
       title={Freedman's inequality for matrix martingales},
        date={2011},
        ISSN={1083-589X},
     journal={Electron. Commun. Probab.},
      volume={16},
       pages={262\ndash 270},
         url={http://dx.doi.org/10.1214/ECP.v16-1624},
      review={\MR{2802042}},
}

\bib{Vil}{book}{
      author={Villani, C{\'e}dric},
       title={Optimal transport},
      series={Grundlehren der Mathematischen Wissenschaften [Fundamental
  Principles of Mathematical Sciences]},
   publisher={Springer-Verlag},
     address={Berlin},
        date={2009},
      volume={338},
        ISBN={978-3-540-71049-3},
         url={http://dx.doi.org/10.1007/978-3-540-71050-9},
        note={Old and new},
      review={\MR{2459454 (2010f:49001)}},
}

\bib{VDN}{book}{
      author={Voiculescu, D.~V.},
      author={Dykema, K.~J.},
      author={Nica, A.},
       title={Free random variables},
      series={CRM Monograph Series},
   publisher={American Mathematical Society},
     address={Providence, RI},
        date={1992},
      volume={1},
        ISBN={0-8218-6999-X},
        note={A noncommutative probability approach to free products with
  applications to random matrices, operator algebras and harmonic analysis on
  free groups},
      review={\MR{1217253 (94c:46133)}},
}

\bib{Wal}{book}{
      author={Walters, Peter},
       title={An introduction to ergodic theory},
      series={Graduate Texts in Mathematics},
   publisher={Springer-Verlag},
     address={New York},
        date={1982},
      volume={79},
        ISBN={0-387-90599-5},
      review={\MR{648108 (84e:28017)}},
}

\bib{Wa}{article}{
      author={Wang, Feng-Yu},
       title={Equivalence of dimension-free {H}arnack inequality and curvature
  condition},
        date={2004},
        ISSN={0378-620X},
     journal={Integral Equations Operator Theory},
      volume={48},
      number={4},
       pages={547\ndash 552},
         url={http://dx.doi.org/10.1007/s00020-002-1264-y},
      review={\MR{2047597 (2004m:58061)}},
}

\bib{Wan05}{book}{
      author={Wang, Feng-Yu},
       title={Functional inequalities, {M}arkov properties, and spectral
  theory},
   publisher={Science Press},
     address={Beijing},
        date={2005},
}

\bib{Ze12}{article}{
      author={{Zeng}, Q.},
       title={{Kolmogorov's law of the iterated logarithm for noncommutative
  martingales}},
        date={2012-12},
     journal={ArXiv e-prints},
      eprint={1212.1504},
}

\end{biblist}
\end{bibdiv}
\end{document}